\documentclass{amsart}
\usepackage{amsfonts}
\usepackage{amsmath}
\usepackage{amssymb}
\usepackage{amsthm}
\usepackage{graphicx}
\usepackage{color}
\usepackage[square,numbers]{natbib}
\usepackage{hyperref}
\usepackage{amsaddr}
\newtheorem{theorem}{Theorem}[section]
\newtheorem{lemma}[theorem]{Lemma}
\newtheorem{proposition}[theorem]{Proposition}
\newtheorem{corollary}[theorem]{Corollary}
\theoremstyle{definition}
\newtheorem{definition}[theorem]{Definition}
\newtheorem{example}[theorem]{Example}
\newtheorem{remark}[theorem]{Remark}
\newtheorem{notation}[theorem]{Notation}
\numberwithin{equation}{section} 

\newcommand*{\doi}[1]{\href{http://dx.doi.org/#1}{doi: #1}}

\DeclareMathOperator*{\esssup}{ess\,sup}

\begin{document}

\title[Linear Differential-Algebraic Operators
with Periodic Coefficients]{On the Spectral Theory of Linear Differential-Algebraic Equations
with Periodic Coefficients}

\author{Bader Alshammari and Aaron Welters}
\email{awelters@fit.edu}
\address{Department of Mathematical Sciences\\ Florida Institute of Technology\\ Melbourne, FL 32937, USA}%

\date{\today}

\thanks{*This work is based in part on the Ph.D.\ dissertation of the first author.}



\begin{abstract}
In this paper, we consider the spectral theory of linear differential-algebraic equations (DAEs) for periodic DAEs in canonical form, i.e.,
\begin{equation*}
   J \frac{df}{dt}+Hf=\lambda Wf,  
\end{equation*}
 where  $J$ is a constant skew-Hermitian $n\times n$ matrix that is not invertible, both $H=H(t)$ and $W=W(t)$ are $d$-periodic Hermitian $n\times n$-matrices with Lebesgue measurable functions as entries, and $W(t)$ is positive semidefinite and invertible for a.e.\ $t\in \mathbb{R}$ (i.e., Lebesgue almost everywhere). Under some additional hypotheses on $H$ and $W$, called the local index-1 hypotheses, we study the maximal and the minimal operators $L$ and $L_0'$, respectively, associated with the differential-algebraic operator  $\mathcal{L}=W^{-1}(J\frac{d}{dt}+H)$, both treated as an unbounded operators in a Hilbert space $L^2(\mathbb{R};W)$ of weighted square-integrable vector-valued functions. We prove the following: (i) the minimal operator $L_0'$ is a densely defined and closable operator; (ii) the maximal operator $L$ is the closure of $L_0'$; (iii) $L$ is a self-adjoint operator on $L^2(\mathbb{R};W)$ with no eigenvalues of finite multiplicity, but may have eigenvalues of infinite multiplicity. As an important application, we show that for 1D photonic crystals with passive lossless media, Maxwell's equations for the electromagnetic fields become, under separation of variables, periodic DAEs in canonical form satisfying our hypotheses so that our spectral theory applies to them (a primary motivation for this paper).$^*$
\end{abstract}
\maketitle

\section{Introduction}\label{sec:intro}
The spectral theory of linear differential-algebraic equations (DAEs) with periodic coefficients is of significant interest from both a
theoretical point of view and in applications (see, for instance, \cite{98LM, 00AD, 03LM, 10DV} and \cite[Sec.\ 6.6.7]{05RK}) and this is especially true for electromagnetic problems involving 1D photonic crystals (see, for instance, \cite{FV, AW, 12SW, 13SW, 16SW}). For linear ordinary differential equations (ODEs), the spectral theory is well-developed (see, for instance, \cite{DE, 87JW}). But this is not true for DAEs when the interval under consideration is unbounded or when the spectral problem is posed on a Hilbert space of weighted $L^2$ functions. Yet, such DAEs arise naturally when considering electromagnetic phenomena in periodic layered media (1D photonic crystals) involving passive lossless media (especially when the materials as anisotropic, biisotropic, or bianisotropic instead of isotropic). In this case, solving for the time-harmonic electromagnetic fields using separation of variables, Maxwell's equations are reduced to exactly such a spectral problem on the whole real line for DAEs in canonical form.

Motivated by such applications, we will consider the spectral theory for periodic DAEs in the following implicit canonical (cf.\ \cite{90KU}) or Hamiltonian form:
\begin{equation}
   J \frac{df}{dt}+Hf=\lambda Wf,  \label{def:intro:DAEsStandardForm}
\end{equation}
 where  $J$ is a constant skew-Hermitian $n\times n$ matrix that is not invertible (i.e., $\det J=0$), both $H=H(t)$ and $W=W(t)$ are $d$-periodic Hermitian $n\times n$ matrices with Lebesgue measurable functions as entries, and $W(t)$ is positive semidefinite and invertible for a.e.\ $t\in \mathbb{R}$ (i.e., Lebesgue almost everywhere).

In this paper we study the maximal and the minimal operators $L$ and $L_0'$ (see Def.\ \ref{DefMinMaxOpGenerByTheDAEsOp}), respectively, associated with the differential-algebraic (DA) operator  $\mathcal{L}=W^{-1}(J\frac{d}{dt}+H)$ (see Def.\ \ref{def:DiffAlgDAEOp}), both treated as an unbounded operators in a Hilbert space $L^2(\mathbb{R};W)$ [as defined by (\ref{def:WeightedHilbertSpL2Funcs}) with inner product (\ref{def:WeightedHilbertSpL2FuncsInnerProd}); see also Remark \ref{rem:IsometricIsomorphismWeightedHilbertSpace}] of weighted square-integrable vector-valued functions. Our main result is Theorem \ref{thm:MainResultOfThePaper} [that requires the local index-1 hypotheses (see Def.\ \ref{DefIndex1Hyp}) are true for $H-z_0W, W$ with respect to $J$ on the interval $\mathbb{R}$ for some $z_0\in\mathbb{C}\setminus\mathbb{R}$] which can be summarized as follows: (i) the minimal operator $L_0'$ is a densely defined and closable operator; (ii) the maximal operator $L$ is the closure of $L_0'$ (i.e, $L=\overline{L_0'}$); (iii) $L$ is a self-adjoint operator on $L^2(\mathbb{R};W)$ (that is, $L^*=L$) with no eigenvalues of finite multiplicity, but may have eigenvalues of infinite multiplicity [or, equivalently, $\dim \ker(L-\lambda I)\in \{0,\infty\}$ for each $\lambda\in\mathbb{C}$]. In addition, we give two examples of such an $L$ that does have an eigenvalue of infinite multiplicity (see Examples \ref{ex:MaxOpSelfAdjWithEigenvalueInftMult} and \ref{ex:GenerizationOfExMaxOpSelfAdjWithEigenvalueInftMult}).

An important result of this paper is Theorem \ref{thm:StrongerIndex1Hyp} which gives simple hypotheses on the coefficients $H,W$ in terms of $J$ that are sufficient for Theorem \ref{thm:MainResultOfThePaper} to be true [i.e., for (i)-(iii) to be true] and they can be stated as follows: Let $v_1,\ldots, v_{n_1}$ and $v_{n_1+1},\ldots, v_{n}$ be an orthonormal basis of $n\times 1$ column vectors for the range of $J$ and kernel (nullspace) of $J$, respectively, and define the $n\times n$ unitary matrix $V$ to be the column matrix $V=[v_1|\cdots|v_{n_1}|v_{n_1+1}|\cdots|v_{n}]$. We next define $n_2=n-n_1$ and the $n_i\times n_j$ matrix-valued functions $H_{ij}$ and $W_{ij}$ (for $i,j=1,2$) by forming a $2\times 2$ block matrix partitioning of $V^{-1}HV$ and $V^{-1}WV$ as
\begin{align}
    V^{-1}HV=\begin{bmatrix}
    H_{11} & H_{12}\\
    H_{21} & H_{22}
    \end{bmatrix},\;V^{-1}WV=\begin{bmatrix}
    W_{11} & W_{12}\\
    W_{21} & W_{22}
    \end{bmatrix}.
\end{align}
Then Theorem \ref{thm:StrongerIndex1Hyp} tells us that if all the entries of $H_{11}, W_{11},$ and $H_{12}W_{22}^{-1}H_{21}$ are (i.e., integrable on the interval $[0,d]$) then Theorem \ref{thm:MainResultOfThePaper} is true.

Let us compare all this to the spectral theory for periodic ODEs in the form (\ref{def:intro:DAEsStandardForm}) where $J$ is invertible ($\det J\not=0$), also known as canonical or Hamiltonian (system of) ODEs (see, for instance, \cite{64VA, 75YSv1, DE, 87JW, 02DK, 12LS, 12AD, 13SS, 18CR, 20BS}). In \cite{DE}, V.\ I. Derguzov proves, under the hypotheses $H(t)$ and $W(t)$ are also smooth symmetric matrix-valued functions, that the maximal operator $L$ is self-adjoint on $L^2(\mathbb{R};W)$ by giving an explicit spectral representation of $L$ expanded in terms of the Bloch solutions of the periodic ODEs (\ref{def:intro:DAEsStandardForm}) which proves $L$ has purely absolutely continuous (AC) spectrum and is unitarily equivalent to multiplication by $\lambda$ in a special Hilbert space of square-integrable functions.  Furthermore, he proves that the spectral multiplicity at each $\lambda$ is equal to the number of linearly independent bounded solutions (corresponding to the dimension of the space of those Bloch solutions) of (\ref{def:intro:DAEsStandardForm}). In \cite{87JW},  J.\ Weidmannn has shown using different methods that you can weaken the hypotheses that Derguzov had to $H$ and $W$ are Hermitian with entries in $L^1([0,d])$ and still prove the maximal operator $L$ is self-adjoint on $L^2(\mathbb{R};W)$ \cite[Theorem 12.3]{87JW} with purely AC spectrum and $L=\overline{L_0'}$ is the closure of the minimal operator $L_0'$ \cite[Theorem 12.4]{87JW}. In particular, the self-adjoint operator $L$ has no eigenvalues. It should be pointed out that we are not aware of any weaker hypotheses that allow such a broad class of coefficients $H, W$ which yield these results and in this sense, \cite{87JW} is the ``state-of-the-art" for the spectral theory of canonical ODEs with periodic coefficients. 

Our goal here is to develop such a spectral theory for DAEs in the form (\ref{def:intro:DAEsStandardForm}) that is comparable to those results above for ODEs, under the weakest hypotheses on $H,W$ we can get. Based on techniques from \cite{87JW} adapted to DAEs together with the spectral theory for unbounded operators \cite{75RS, 80RS} and the theory of Schur complements \cite{05FZ}, we are able to prove the results (i)-(iii) in Theorem \ref{thm:MainResultOfThePaper} which represent significant progress toward this goal. Furthermore, as the local index-$1$ hypotheses in that theorem may not be easy to verify in applications, we also provide Proposition \ref{prop:MainResultOfPaperOnIndex1HypEquiv} and Theorem \ref{thm:StrongerIndex1Hyp} to help simplify those hypotheses.  Finally, given the difficulties of proving our main results (i)--(iii), we do not try to prove the self-adjoint operator $L$ has no singular continuous spectrum (as done for ODEs in \cite{DE} and \cite{87JW}) nor do we try to develop an analog spectral representation of $L$ (as was done in \cite{DE} for ODEs) in terms of an expansion in Bloch solutions. Nevertheless, for a complete spectral theory on $L$, it would be desirable to do so for the applications mentioned above, but this is left for future work.

As a final remark, the literature on DAEs using functional analysis and operator theory that is closest in spirit to our approach, especially Section \ref{sec:MinMaxDAEOps}, is \cite{89MH, 07SZ,15RM, 20HR}. They are mainly interested in closedness and normal solvability of the maximal operator $L$ on bounded or compact intervals $I$ for Hilbert spaces of $L^2$ functions and its adjoint relationship to the minimal operator $L_0'$. However, they do not treat the self-adjoint spectral theory of DAEs nor do they treat unbounded intervals nor weighted Hilbert spaces of $L^2$ functions as we need to in this paper. In addition, if we were to use the approach that they do then the weakest hypotheses on $H, W$ that we could use would require its entries to be in $L^{\infty}$ (i.e., essentially bounded functions) among other things. Nevertheless, these works are complementary to our paper.

The rest of this paper is organized as follows. In Section \ref{sec:1DPCs}, we motivate our studies on the spectral theory of periodic DAEs by considering electromagnetism and Maxwell's equations associated with 1D photonic crystals and passive lossless media. In Section \ref{sec:Preliminaries}, we briefly introduce some basic notations, conventions, and definitions for matrices, operators, and functions spaces needed. We also include some preliminary results from the elementary theory of linear ODEs and spectral theory for unbounded operators. In Section \ref{sec:MinMaxDAEOps}, we study the differential operators $L_0', L, \mathcal{L}$ associated with the DAEs on an arbitrary nonempty interval $I\subseteq \mathbb{R}$ (both for bounded as well as unbounded intervals). In Section \ref{sec:SpecThyPeriodicDAEs}, we study the periodic case of these operators (with $I=\mathbb{R}$) and prove our main results (i)--(iii) above. Finally, in Subsec.\ \ref{subsec:SimplifyIndex1Hyp}, we show how to simplify the hypotheses needed to prove those results.

\section{\label{sec:1DPCs}One-dimensional photonic crystals}
In this section, we motivate our studies on periodic linear DAEs by showing how they arise in electromagnetism (EM) involving 1D photonic crystals. In particular, starting from Maxwell's equations, we consider time-harmonic electromagnetic fields in a periodic layered media (i.e., a 1D photonic crystal) with passive lossless materials, and show how to proceed to reduce the problem from one involving partial differential equations (PDEs) to periodic linear DAEs in the canonical form (\ref{def:intro:DAEsStandardForm}). 

To set up the problem, we recall that Maxwell's equations describe the electric and magnetic fields arising from distributions of electric charges and currents, and how those fields change in time (for more details, see \cite{75JJ, 95LL, 86JK, 11JJ}). Mathematically, the (macroscopic)
Maxwell's equations (in Gaussian units) are the PDEs:
\begin{gather}\label{eq1}
  \nabla\times E=-\frac{1}{c}\frac{\partial B}{\partial t },\qquad \nabla\times H=\frac{4\pi}{c}J+\frac{1}{c}\frac{\partial D}{\partial t },
\end{gather}
\begin{gather}\label{eq3}
  \nabla\cdot D=4\pi \rho,\qquad  \nabla\cdot B=0,
\end{gather}
where $E$ and $H$ are the (macroscopic) electric and magnetic fields, $D$ and $B$ are the electric displacement and magnetic induction fields, $\rho$ and $J$ are the free charge and current densities (with $\rho=0, J=0$ in there are no sources), and $c$ is the speed of light (in a vacuum).

For passive lossless (linear) homogeneous media, the constitutive relations are of the form
\begin{gather}
    D=\varepsilon E + \xi H,\label{def:ConstitutiveRelations1}\\
    B=\zeta E + \mu H,\label{def:ConstitutiveRelations2}
\end{gather}
where $\varepsilon,\mu,\xi,\mu\in M_{3}(\mathbb{C})$ (dielectric permittivity, magnetic permeability, and magnetoelectric coupling tensors), i.e., are $3\times 3$ matrices with entries in $\mathbb{C}$, such that the $2\times 2$ block matrix $M=[M_{ij}]_{i,j=1,2}\in M_6(\mathbb{C})$ defined by
\begin{gather}
    M=\begin{bmatrix}
    \varepsilon & \xi\\
    \zeta & \mu
    \end{bmatrix},\label{def:ConstitutiveRelations3}
\end{gather}
satisfies the following self-adjoint and positivity constraints (see, for instance, \cite[Sec.\ 80]{95LL}, \cite{95LS, 14AW}):
\begin{gather}
    M^*=M>0,\label{def:ConstitutiveRelations4}
\end{gather}
(i.e., $M$ is a Hermitian matrix that is positive semidefinite and invertible). In particular, we have
\begin{align}
    \varepsilon^*=\varepsilon,\;\mu^*=\mu,\;\zeta=\xi^*.
\end{align}

The media can now be classified as follows (cf.\ \cite{86JK, 94LS, 07AS, 09SL}). In the case $\xi=0$, the media is called isotropic if $\varepsilon$ and $\mu$ are scalar matrices (i.e., scalar multiples of $I_3$, the $3\times 3$ identity matrix) and if not then it is anisotropic. In the case $\xi\not =0$, the media is called biisotropic if $\varepsilon, \mu, \xi$ are scalar matrices and if not then it is bianisotropic.

We consider now a one-dimensional (1D) photonic crystal made of plane parallel layers of these materials. More precisely, a $d$-periodic layered (or stratified) medium whose layers are normal to $z$-axis, homogeneous in the $xy$-plane, and consisting only of passive linear lossless layers for a.e.\ $z\in \mathbb{R}$, i.e., the constitutive relations are of the form (\ref{def:ConstitutiveRelations1}, \ref{def:ConstitutiveRelations2}), but now the matrix $M$ in (\ref{def:ConstitutiveRelations3}) is a $6\times 6$ matrix-valued function $M=M(z)\in M_6(\mathcal{M}(\mathbb{R}))$ satisfying
\begin{gather}
    M(z+d)=M(z),\;\;M(z)^*=M(z)>0 \text{ for a.e. }z\in \mathbb{R}.
\end{gather}
Then Maxwell's equations without sources (i.e., $\rho=0,J=0$) for the electromagnetic fields $E,H$ in this medium admit time-harmonic solutions of the form
\begin{align}
    E=E(z)e^{i(k_1x+k_2y-\omega t)},\;H=H(z)e^{i(k_1x+k_2y-\omega t)},\label{DefTimeHarmonicEMFieldsZDir}
\end{align}
where $\omega$ is the frequency, $i=\sqrt{-1}$, and $(k_1,k_2)\in \mathbb{R}^2$ is the wavevector parallel to the layers. Using separation of variables, Maxwell's equations (without sources) for these time-harmonic solutions are reduced to the \textbf{periodic Maxwell's DAEs}: 

\begin{subequations}\label{def:MaxwellsPeriodicDAEs}
\begin{equation}
       J\frac{d}{dz}f+ Hf=\lambda Wf,\tag{\ref{def:MaxwellsPeriodicDAEs}}
     \end{equation}
\begin{align}
    &\lambda=\frac{\omega}{c},\label{def:MaxwellsPeriodicDAEsLambda}\\
    &f(z)=\begin{bmatrix}
    E(z)\\
    H(z)
    \end{bmatrix},\label{def:MaxwellsPeriodicDAEsfFunction}\\
    &J=i^{-1}\begin{bmatrix}
    0 & -e_3\times \\
      e_3\times & 0 \\
    \end{bmatrix},\, e_3\times=\begin{bmatrix}
    0 & -1&0 \\
     1 & 0&0 \\
     0&0&0\\
    \end{bmatrix},\label{def:MaxwellsPeriodicDAEsJ}\\
    &H=\begin{bmatrix}
                                0 & k_\bot\times \\
                                -k_\bot\times & 0 \\
                              \end{bmatrix},\, k_\bot\times=\begin{bmatrix}
                        0 & 0 & k_2 \\
                        0 & 0 & -k_1 \\
                        -k_2 & k_1 & 0 \\
                      \end{bmatrix},\label{def:MaxwellsPeriodicDAEsH}\\
    &W=M=\begin{bmatrix}
    \varepsilon & \xi\\
    \zeta & \mu
    \end{bmatrix}.\label{def:MaxwellsPeriodicDAEsW}
\end{align}
\end{subequations}
In particular, these are periodic DAEs in canonical form (\ref{def:intro:DAEsStandardForm}) (with the independent variable $t$ replaced by $z$) such that $J$ is a constant skew-Hermitian $6\times 6$ matrix that is not invertible (i.e., $\det J=0$), both $H$ (a constant function of $z$) and $W=W(z)$ are $d$-periodic Hermitian $6\times 6$ matrices with Lebesgue measurable functions as entries, and $W(z)$ is positive semidefinite and invertible for a.e.\ $z\in \mathbb{R}$.

Two other types of applied electromagnetic problems for 1D photonic crystals that are worth briefly mentioning which have the canonical form (\ref{def:MaxwellsPeriodicDAEs}), where $f,J$ remains the same, but the $\lambda, H,W$  (assuming $\omega >0$) change as follows.  
  \begin{itemize} 
      \item Disorder added to the periodic unit cell: e.g., $\varepsilon(\lambda):=\varepsilon_0+\lambda\varepsilon_1,$
      $\mu(\lambda):=\mu_0+\lambda\mu_1$ ($\lambda$ scales amount of disorder),
      \begin{align*}
          H=\left(
                              \begin{array}{cc}
                                0 & k_\bot\times \\
                                -k_\bot\times & 0 \\
                              \end{array}
                            \right)-\frac{\omega}{c}\left(
                      \begin{array}{cc}
                        \varepsilon_0 & 0 \\
                        0 & \mu_0 \\
                      \end{array}
                    \right),\;\;W= \frac{\omega}{c}\left(
                      \begin{array}{cc}
                        \varepsilon_1 & 0 \\
                        0 & \mu_1 \\
                      \end{array}
                    \right).
      \end{align*}
      \item Lossy photonic crystals: e.g., $\varepsilon(\gamma):= \operatorname{Re}\varepsilon+i\gamma \operatorname{Im} \varepsilon $, $\mu(\gamma):= \operatorname{Re}\mu+i\gamma \operatorname{Im} \mu $, $|\gamma|\ll 1$ (low-loss), $|\gamma|\gg 1$ (high-loss), $\lambda:=i\gamma$ ($\lambda$ scales amount of loss),
      \begin{align*}
          H=\left(
                              \begin{array}{cc}
                                0 & k_\bot\times \\
                                -k_\bot\times & 0 \\
                              \end{array}
                            \right)-\frac{\omega}{c}\left(
                      \begin{array}{cc}
                        \operatorname{Re}\varepsilon & 0 \\
                        0 & \operatorname{Re}\mu \\
                      \end{array}
                    \right),\;\;W= \frac{\omega}{c}\left(
                      \begin{array}{cc}
                        \operatorname{Im}\varepsilon & 0 \\
                        0 & \operatorname{Im}\mu \\
                      \end{array}
                    \right).
      \end{align*}
  \end{itemize}
  
 In summary, the electrodynamics of 1D photonic crystals in passive linear lossless media provides a model that motivates this paper on the spectral theory of periodic DAEs in the canonical form (\ref{def:intro:DAEsStandardForm}). For as we have shown in this model, Maxwell's equations become the periodic DAEs (\ref{def:MaxwellsPeriodicDAEs}) in this form and so the spectral theory developed in this paper applies to these DAEs (a primary motivation for this paper).
 
 Now before we move on, let us use this example to demonstrate here how our approach would begin when we want to solve DAEs by reducing them to ODEs [which is part of the reason why we need the (local) index-$1$ hypotheses]. As $\det J=0$, let us introduce the unitary matrix $V\in M_6(\mathbb{C})$, i.e.,
\begin{align}
    V\in M_6(\mathbb{C}),\;\;V^*=V^{-1},
\end{align}
defined by
\begin{gather}
    V=\begin{bmatrix}
    1 & 0  & 0 & 0 & 0 & 0\\
    0 & 1  & 0 & 0 & 0 & 0\\
    0 & 0  & 0 & 0 & 1 & 0\\
    0 & 0  & 1 & 0 & 0 & 0\\
    0 & 0  & 0 & 1 & 0 & 0\\
    0 & 0  & 0 & 0 & 0 & 1\\
    \end{bmatrix}.
\end{gather}
Then $V^{-1}JV=[J_{ij}]_{i,j=1,2}$ has the $2\times2$ block partitioned matrix form
\begin{align}
    V^{-1}JV=[J_{ij}]_{i,j=1,2}=\begin{bmatrix}
               J_{11} & J_{12} \\
               J_{21} & J_{22} \\
             \end{bmatrix}=\begin{bmatrix}
               J_{11} & 0 \\
               0 & 0 \\
             \end{bmatrix},\quad \det(J_{11})\neq 0,
\end{align}
and, more precisely,
\begin{align}
    J_{ij}\in M_{n_i\times n_j}(\mathbb{C}), i,j=1,2;\;\;J_{ij}=0,\;\;(i,j)\not = (1,1),\;\;\det(J_{11})\neq 0,
\end{align}
where we define $n_1,n_2\in \mathbb{N}$ by
\begin{align}
    n_1:=\operatorname{rank} J=4,\;\;n_2:=\dim \ker J=\operatorname{nullity}(J)=6-n_1=2.
\end{align}
In particular, $J_{11}\in M_{4}(\mathbb{C})$ is given by
\begin{gather}
    J_{11}=i^{-1}\begin{bmatrix}
    0 & 0  & 0 & 1 \\
    0 & 0  & -1 & 0 \\
    0 & -1  & 0 & 0 \\
    1 & 0  & 0 & 0 \\
    \end{bmatrix}, \;i=\sqrt{-1}.
\end{gather}
The relevance of this matrix $V$ is as follows.  Consider the components of the electromagnetic fields
\begin{align}
    E=\begin{bmatrix}E_1&E_2&E_3\end{bmatrix}^T,\; H=\begin{bmatrix}H_1&H_2&H_3\end{bmatrix}^T.
\end{align}
Then in terms of the tangential components (i.e., those with index $i=1,2$) and the normal components (i.e., those index by $i=3$) we have
\begin{align}
 V^{-1}\begin{bmatrix}
    E&
    H
    \end{bmatrix}^T=\begin{bmatrix}
    E_1&
    E_2&
    H_1&
    H_2&
    E_3&
    H_3
    \end{bmatrix}^T.
\end{align}
Hence, for the time-harmonic fields $E,H$ of the form (\ref{DefTimeHarmonicEMFieldsZDir}) and in terms of the corresponding $6\times 1$ column vector-valued function $f(z)=\begin{bmatrix}E(z)&H(z)\end{bmatrix}^T$ (written as in a $2\times 1$ block vector form), the vector function $V^{-1}f$ has the following $2\times 1$ block vector form: 
\begin{align}
    V^{-1}f=\begin{bmatrix}
    f_1\\
    f_2
    \end{bmatrix},\;f(z)=\begin{bmatrix}E(z)\\ H(z)\end{bmatrix},\;f_1(z)=\begin{bmatrix}E_1(z)\\E_2(z)\\ H_1(z)\\H_2(z)\end{bmatrix},\;f_2(z)=\begin{bmatrix}E_3(z)\\H_3(z)\end{bmatrix}.
\end{align}
The importance of introducing $V$ and $J_{11}$ is then the following: The tangential components $f_1(z)=\begin{bmatrix}
    E_1(z)&
    E_2(z)&
    H_1(z)&
    H_2(z)
    \end{bmatrix}^T$
give rise to the differential part of the periodic Maxwell's DAEs (\ref{def:MaxwellsPeriodicDAEs}) (see also Example \ref{ex:1DPCMaxwellsDAEOp}), that is, they are the solutions of a linear system of canonical ODEs with periodic coefficients in which $J_{11}$ is now the matrix coefficient for the derivative term, whereas the normal components $f_2(z)=\begin{bmatrix}
    E_3(z)&
    H_3(z)
    \end{bmatrix}^T$
give rise to the algebraic part of (\ref{def:MaxwellsPeriodicDAEs}), that is, they are the output of left multiplication by a matrix function of $z$ acting on $f_1(z)$. This statement can be made rigorous using Lemma \ref{lem:AltCharDAEOpAndItsDomain} and Proposition \ref{PropKeySolvabilityLocalIndex1DAEs} [by taking the interval $I=\mathbb{R}$ and the coefficients $J,H,W$ to be (\ref{def:MaxwellsPeriodicDAEsJ}), (\ref{def:MaxwellsPeriodicDAEsH}), (\ref{def:MaxwellsPeriodicDAEsW})], two key results in this paper on solvability of linear DAEs.

We conclude this section with a comment and then two remarks. One thing to keep in mind here is that the example above demonstrates the complexities that arise when considering linear DAEs versus linear ODEs, but also shows that there is a need to develop the operator and spectral theory for them. But we do have simpler examples in this paper for which we apply our theory, see Examples \ref{ex:MaxOpSelfAdjWithEigenvalueInftMult} and \ref{ex:GenerizationOfExMaxOpSelfAdjWithEigenvalueInftMult}.

\begin{remark}
The methods we use in this paper to reduced linear DAEs to linear ODEs is implicit in a standard and well known approach in studying the electrodynamics of layered media \cite{72DW, 99IA, 00IA, FV, 09NT, 09RL, 13GBa, 13GBb, 17PP} (as opposed to other standard methods such as those in \cite{77PYa, 77PYb, 79PY, 05PY}) and has been used effectively by the second author of this paper, for instance, in \cite{AW, 13SW, 16SW, 16CW}. But in this context, our approach that we've developed in the sections below, using DA operators and spectral theory, is completely new and deserves to be further studied from this perspective. Our paper represents a first step in this direction of research.
\end{remark}

\begin{remark}
We are motivated to study the spectral theory of linear DAEs because we are interested specifically in electromagnetic applications involving 1D photonic crystals. More broadly, it is known that linear DAEs arise in computational electromagnetics \cite{19GS} and in circuit theory \cite{08RR}, and as such it would be interesting and worth investigating how our methods and results would be useful in these contexts too.
\end{remark}

\section{\label{sec:Preliminaries}Preliminaries}

This paper lies at the intersection of the several different areas, namely, ODEs and DAEs, spectral theory for unbounded operators, and electromagnetics (as discussed in the previous section). As such, we expect that the reader's level of expertise on these topics may widely vary. Thus, we try to give here some necessary background to make our paper more approachable and self-contained. First, we will introduce the standard function spaces that are needed for describing existence and uniqueness theorems on initial-value problems (IVPs) for linear ODEs/DAEs. After this we recall, for unbounded operators on Hilbert spaces, the following classes of operators: densely defined, closable, closed, symmetric, and self-adjoint.

\subsection{Function spaces}
Let $\mathbb{N}, \mathbb{R},\mathbb{C}$ denote the natural, real, and complex numbers, respectively. For $-\infty \leq a<b\leq \infty,$ an open, closed, or half-open interval $I\subseteq \mathbb{R}$ is denoted by $(a,b),$ $[a,b],$ and $[a,b),(a,b]$, respectively; if $a=-\infty$ or $b=\infty$, then the interval is unbounded, otherwise it is bounded (e.g., compact intervals are precisely those intervals that are closed and bounded) and, as usual, $[a,\infty]=[a,\infty), (a,\infty]=(a,\infty), [-\infty, b]=(-\infty, b], [-\infty,b)=(-\infty,b)$. 

For any $z\in \mathbb{C}$, we denote its complex conjugate and (Hilbert space) norm by $\overline{z}$ and $|z|=(\overline{z}z)^{1/2},$ respectively. 

For any interval $I$, the complex vector space of Lebesgue measurable functions (with equality in the sense of equal a.e.\ on $I$) is denoted by $\mathcal{M}(I)$ [if $I=(a,b)$ then instead of writing $(I)=((a,b))$ we drop the extra parentheses, e.g., $\mathcal{M}(I)=\mathcal{M}(a,b)$]; for each $p\in [1,\infty)$, the subspace $L^p(I)$ of $\mathcal{M}(I)$ is defined by
\begin{align}
    L^p(I)=\left\{f\in \mathcal{M}(I):\int_{I}|f(t)|^p dt<\infty\right\},
\end{align}
which is a Banach space with norm
\begin{align}
    ||f||_p=\left(\int_{I}|f(t)|^p dt \right)^{1/p},\;\;f\in  L^p(I)
\end{align}
and, in the case $p=2$, is a Hilbert space with inner product
\begin{align}
    \langle f,g \rangle_2=\int_{I}\overline{f(t)}g(t) dt,\;\;f,g\in L^2(I);
\end{align}
the subspace $L^{\infty}(I)$ of $\mathcal{M}(I)$ is defined in terms of the essential supremum by
\begin{align}
    L^{\infty}(I)=\left\{f\in \mathcal{M}(I):\esssup_{t\in I}|f(t)|<\infty\right\},
\end{align}
which is a Banach space with norm
\begin{align}
    ||f||_{\infty}=\esssup_{t\in I}|f(t)|,\;\;f\in  L^{\infty}(I).
\end{align}
For any compact interval $I=[a,b]$, we denote the Banach space of all complex-valued absolutely continuous functions on the interval $I$ by $AC(I)$ with norm
\begin{align}
    ||f||_{AC(I)} = |f(a)|+\int_{I}\left|\frac{df}{dt}(\tau)\right| d\tau,\;\;f\in AC(I).
\end{align}
For any interval $I$, the subspace $L_{loc}^p(I)$ of $\mathcal{M}(I)$ and the complex vector space $AC_{loc}(I)$ are defined by
\begin{align}
    L_{loc}^p(I)=\{f\in \mathcal{M}(I):f\in L^p([a,b]),\text{ for every compact interval }[a,b]\subseteq I\},\\
     AC_{loc}(I)=\{f:I\rightarrow \mathbb{C}\;|\;f\in AC([a,b]),\text{ for every compact interval }[a,b]\subseteq I\},
\end{align}
respectively. In particular,
\begin{align}
    L_{loc}^p(I)=L^p(I), AC_{loc}(I)=AC(I),\text{ if } I \text{ is a compact interval.}
\end{align}
The subspace $W^{1,p}(I)$ of $L^p(I)$ is defined by 
\begin{align}
    W^{1,p}(I)=\left\{f\in L^p(I):f\in AC_{loc}(I), \frac{df}{dt}\in L^p(I)\right\},
\end{align}
which is a Banach space with norm
\begin{align}
    ||f||_{1,p}=||f||_p+\left\Vert\frac{df}{dt}\right\Vert_p,\;\;f\in  W^{1,p}(I).
\end{align}
The subspace $W^{1,p}_{loc}(I)$ of $L^p_{loc}(I)$ is defined by 
\begin{align}
    W^{1,p}_{loc}(I)=\{f\in\mathcal{M}(I):f\in W^{1,p}([a,b]),\text{ for every compact interval }[a,b]\subseteq I\},
\end{align}
in particular,
\begin{align}
    W^{1,p}_{loc}(I)= W^{1,p}(I),\text{ if } I \text{ is a compact interval.}
\end{align}
Note that for each $f\in W_{loc}^{1,p}(I)$, there is a unique $g\in AC_{loc}(I)$ such that $f(t)=g(t)$ for a.e.\ $t\in I$, and as such, we will always use this representative of $f$ when we evaluate $f$ at a point, i.e., for each $t_0\in I$ will define $f(t_0):=g(t_0)$. With this convention we have the integral representation
\begin{align}
    f(t)=f(t_0)+\int_{t_0}^t\frac{df}{dt}(\tau)d\tau,\;\; t_0,t\in I,\label{IntegralRepresW1pFunctions}
\end{align}
and as such $f$ is a continuous function on the interval $I$. Similarly, if $f\in W^{1,p}(I)$ then $f$ has a representation in $AC_{loc}(I)$ which has the integral representation (\ref{IntegralRepresW1pFunctions}) and is continuous on $I$. Using this identification of such functions with their integral representations, we can abuse notation and make the identification
\begin{gather}
    W^{1,1}(I)\subseteq W^{1,1}_{loc}(I)=AC_{loc}(I),\text{ for any interval $I\subseteq \mathbb{R}$;}\\
     W^{1,1}(I)=W^{1,1}_{loc}(I)=AC_{loc}(I)=AC(I),\text{ if $I$ is a compact interval.}
\end{gather}

Let $m,n\in \mathbb{N}$, $I$ an interval, $p\in [1,\infty]$, and
\begin{align}
    \mathcal{V}\in\{\mathbb{C},\;\mathcal{M}(I),\; L_{loc}^p(I),\; W^{1,p}_{loc}(I),\; AC_{loc}(I)\}.
\end{align} 
We denote the set of all $n\times m$ matrices with entries in $\mathcal{V}$ by $M_{n,m}(\mathcal{V})$, and define
\begin{align}
    \mathcal{V}^n=M_{n,1}(\mathcal{V}),\;\;M_{n}(\mathcal{V})=M_{n,n}(\mathcal{V})
\end{align}
and identify $\mathcal{V}$ with $\mathcal{V}^1.$ If $\mathcal{V}$ is a Banach space (Hilbert space) with norm $||\cdot||_{\mathcal{V}}$ then $M_{n,m}(\mathcal{V})$ will denote the Banach space (Hilbert space) with norm
\begin{align}
    ||[a_{ij}]||=\left(\sum_{i=1}^n\sum_{j=1}^m||a_{ij}||_{\mathcal{V}}^2\right)^{1/2},\;\;[a_{ij}]\in M_{n,m}(\mathcal{V}).
\end{align}
Any matrix-valued function $A:I\rightarrow M_{n,m}(\mathbb{C})$ can be written uniquely as $A=[a_{ij}]$, where $a_{ij}:I\rightarrow \mathbb{C}$ and $[a_{ij}](t)=[a_{ij}(t)]$ for all $t\in I$ and all $i=1,\ldots, n, j=1, \ldots, m$; the complex conjugation $\overline{A}:I\rightarrow M_{n,m}(\mathbb{C})$, transpose $A^T:I\rightarrow M_{m,n}(\mathbb{C})$, and conjugate transpose $A^*:I\rightarrow M_{m,n}(\mathbb{C})$ are defined by
\begin{align}
    \overline{A}(t)=\overline{A(t)}=[\overline{a_{ij}(t)}],\;\;A^T(t)=A(t)^T=[a_{ji}(t)],\;\;A^*(t)=A(t)^*=[\overline{a_{ji}(t)}],\;\;\forall t\in I.\label{DefFuncConjTransp}
\end{align}
Similarly, the complex conjugation, transpose, and conjugate transpose of any  $[a_{ij}]\in M_{n,m}(\mathcal{V})$ satisfies
\begin{align}
    \overline{[a_{ij}]}=[\overline{a_{ij}}]\in M_{n,m}(\mathcal{V}),\;\;[a_{ij}]^T=[a_{ji}],\;\;[a_{ij}]^*=[\overline{a_{ji}}]\in M_{m,n}(\mathcal{V}).\label{DefFuncSpaceConjTransp}
\end{align}
In particular, for the Hilbert spaces $\mathbb{C}^n$ and $(L^2(I))^n$, their inner products $\langle \cdot, \cdot \rangle$ and $\langle \cdot,\cdot\rangle_2$, respectively, are defined as
\begin{gather}
    \langle x, y \rangle=x^*y,\;\;x,y\in \mathbb{C}^n,\\
    \langle f,g\rangle_2=\int_I\langle f(t),g(t)\rangle dt,\;\;f,g\in (L^2(I))^n.
\end{gather}
If $\mathcal{V}$ is a functional space in which integration $\int_{U}(\cdot)\; dt$ over an interval $U\subseteq I$ is well-defined then we define the integration $\int_{U}(\cdot)\; dt$ on $M_{n,m}(\mathcal{V})$ by
\begin{align}
    \int_{U}[a_{ij}](t)dt=\left[\int_{U}a_{ij}(t)dt\right],\;\;[a_{ij}]\in M_{n,m}(\mathcal{V}).
\end{align}
Similarly, if $\mathcal{V}$ is a functional space in which differentiation $\frac{d}{dt}$ is defined (either in classical or weak sense) then we define differentiation $\frac{d}{dt}$ on $M_{n,m}(\mathcal{V})$ by
\begin{align}
    \frac{d[a_{ij}]}{dt}=\left[\frac{da_{ij}}{dt}\right],\;\;\;\;[a_{ij}]\in M_{n,m}(\mathcal{V}).
\end{align}

The following well-known theorem \cite{12Ze} will be crucial in our paper.

\begin{theorem}[ODE IVP -- Existence and uniqueness of solutions]\label{ThmODEIVPExistenceUniquenessSolns}
  Let $I$ be any interval and $m,n\in \mathbb{N}$. If 
  \begin{align}
      A\in M_n(L^1_{loc}(I)),\;\;F\in M_{n,m}(L^1_{loc}(I))\label{ThmODEIVPExistenceUniquenessSolnsHypL1Loc}
  \end{align}
  then every initial-value problem (IVP)
  \begin{gather}
      \frac{dX}{dt}+AX=F,\label{ThmODEIVPExistenceUniquenessSolnsDefIVP1}\\
      X(t_0)=C,\;\;t_0\in I,\;\;C\in M_{n,m}(\mathbb{C})\label{ThmODEIVPExistenceUniquenessSolnsDefIVP2}
  \end{gather}
  on $I$, has a unique solution 
  \begin{align}
      X\in M_{n,m}(W^{1,1}_{loc}(I)).\label{ThmODEIVPExistenceUniquenessSolnsDefIVP3}
  \end{align}
  Similarly, in the case in which the interval $I$ is bounded, the statement is true if the ``loc" is dropped in the hypotheses (\ref{ThmODEIVPExistenceUniquenessSolnsHypL1Loc}) and in the conclusion (\ref{ThmODEIVPExistenceUniquenessSolnsDefIVP3}). 
\end{theorem}

\subsection{Unbounded operators on Hilbert spaces}
Here we recall from \cite{80RS} some of the basic notions, definitions, and results for unbounded operators on densely defined, closed and closable operators, and adjoints in Hilbert spaces that will be used in this paper. All proofs of the statements below can be found in \cite{80RS} and so are omitted with the exception of Lemma \ref{LemLinearOpsExtEqualityKerRanConds} and Theorem \ref{thm:prem:FundResultAdjointsKerRanDecomp} (which are crucial to our paper, interesting and useful in their own right, and deserve simple proofs).

\begin{definition}
A (linear) operator $T$ on a Hilbert space $\mathcal{H}$ is a linear map from its domain, a linear subspace of $\mathcal{H}$, into $\mathcal{H}$. This subspace, which we denote by $D(T)$, is called the domain of the operator $T$. Moreover, $T$ is said to be densely defined if $D(T)$ is dense in $\mathcal{H}$ [i.e., $\overline{D(T)}=\mathcal{H}$].
\end{definition}

\begin{definition}
 The graph $\Gamma(T)$ of an operator $ T:D(T)\rightarrow \mathcal{H}$  is the set
$$\Gamma(T)=\left\{( \varphi,T\varphi): \varphi\in D(T)\right\}\subseteq \mathcal{H}\times \mathcal{H},$$ and $T$ is said to be a closed operator  if  $\Gamma(T)$  is a closed set [i.e., $\overline{\Gamma(T)}=\Gamma(T)$] in the Hilbert space  $\mathcal{H}\times \mathcal{H}$ with inner product
$$\langle (\varphi_1,\psi_1),(\varphi_2,\psi_2)\rangle=\langle\varphi_1,\psi_1\rangle+\langle\varphi_2,\psi_2\rangle, \quad  \varphi_i,\psi_i\in \mathcal{H}  ,i=1,2.$$
\end{definition}

\begin{definition}
Let $T$ and $T_1$ be operators on a Hilbert space $\mathcal{H}$. If $\Gamma(T)\subseteq \Gamma(T_1)$, then $T_1$ is said to be extension of $T$ and we write $T\subset T_1$. Equivalently, $T\subset T_1$ if and only if $D(T)\subseteq D(T_1)$ and $T_1\varphi=T\varphi $ for all $\varphi\in D(T).$
\end{definition}
\begin{definition}
   An operator  $T$ is closable if it has a closed extension. Moreover, every closable operator  $T$ has smallest closed extension  called its closure, which is denoted by $\overline{T}$.
   \end{definition}

  \begin{definition}
   Let $T$ be a densely defined linear operator on a Hilbert space $\mathcal{H}$. Let $D(T^*)$ be the set of $\varphi\in \mathcal{H}$ for which there is an $\eta\in\mathcal{H}$ with
   $$\langle T\psi,\varphi\rangle=\langle \psi,\eta\rangle\text{ for all } \psi\in D(T).$$
   For each such $\varphi\in D(T^*)$, we define $T^*\varphi=\eta$. This linear operator $T^*:D(T^*)\rightarrow \mathcal{H}$ is called the adjoint of $T$. 
 \end{definition}
 It is useful to note that, unlike the case of bounded operators, the domain of $T^*$ may not be dense [for instance, it is possible to have $D(T^*)=\{0\}$]. But, if $T^*$ is densely defined, then $T^*$ also has an adjoint $(T^*)^*$ which we abbreviate by $T^{**}$, i.e, $T^{**}=(T^*)^*$. Also, in terms of extensions and adjoints, if $S,T$ are densely defined then $S\subset T$ implies $T^*\subset S^*$. The next theorem (see \cite[Theorem VIII.1]{80RS}) gives a simple and useful relationship between the notions of adjoint and closure.

   \begin{theorem}
     Let $T$ be a densely defined operator on a Hilbert space $\mathcal{H}$. Then:
     \begin{itemize}
       \item[(a)] $T^*$ is closed.
       \item[(b)] $T$ is closable if and only if $D(T^*)$ is dense in which case $\overline{T}=T^{**}$.
       \item[(c)] If $T$ is closable, then $(\overline{T})^*=T^*$.
     \end{itemize}
   \end{theorem}
\begin{definition}
     A densely defined operator $T$ on a Hilbert space is called symmetric if $T\subset T^*$, that is, if $D(T)\subseteq D(T^*)$ and $T\varphi=T^*\varphi$ for all $\varphi\in D(T)$. Equivalently, $T$ is symmetric if and only if 
     \begin{align}
         \langle T\varphi,\psi\rangle=\langle \varphi, T\psi\rangle, \text{ for all } \varphi,\psi\in D(T).\label{def:FormallySelfAdj}
     \end{align}
     An operator
     $T$ is called self-adjoint if $T=T^*$, that is, if and only if $T$ is symmetric and $D(T)=D(T^*)$.
   \end{definition}

\begin{remark}
We want to clear up some confusion that can occur for unbounded operators between the notions of Hermitian, symmetric, and self-adjoint operators (see \cite[p.\ 414]{02PL} for an amusing historical anecdote on this). A symmetric operator is, by our definition (from \cite{80RS}), densely defined. In \cite{80RS}, no distinction is made between symmetric and Hermitian operators (see Sec.\ VIII.2, p.\ 255 in \cite{80RS}). However, in \cite{80JW}, a distinction is made between the two (see Sec.\ 4.4, p.\ 72 in \cite{80JW}). Specifically, in \cite{80JW}, a linear operator $T$ satisfying (\ref{def:FormallySelfAdj}), but not necessarily densely defined, is called Hermitian (or formally self-adjoint, in accordance with the definition of formal adjoint in \cite{80JW}, see Sec.\ 4.4, p.\ 67) and if, in addition, it is densely defined then it is called symmetric, and if, in addition, it equals its adjoint then it is called self-adjoint. For these definitions from \cite{80JW}, these classes are different, but only for unbounded operators (i.e., examples exist of unbounded operators that are Hermitian but not symmetric, and ones that are symmetric but not self-adjoint).
\end{remark}

We will need the following notation in this paper. For any linear operator $T:V\rightarrow W$ between vector spaces $V$ and $W$ (over a common field $F$), we denote its kernel (i.e., nullspace) and range by $\ker(T)$ and $\operatorname{ran}(T)$, respectively. Similarly, if $A$ is a matrix with entries in a field then we denote its kernel and range by $\ker(A)$ and $\operatorname{ran}(A)$ [more precisely, if $T_A$ denotes the operator of left-multiplication by $A$ then $\ker(A)=\ker(T_A), \operatorname{ran}(A)=\operatorname{ran}(T_A)$].

The next theorem from \cite{19ST} will be useful when we discuss the adjoints of the minimal and maximal operators associated with linear DAEs. As the proof in \cite{19ST} uses additional results not presented above and as this theorem is so fundamental to our paper (as well as being interesting and useful more generally), we give an alternative elementary proof of it here. To do so, we first need the next lemma.
\begin{lemma}\label{LemLinearOpsExtEqualityKerRanConds}
Let $X,Y$ be vector spaces, $D(S)$ and $D(T)$ subspaces of $X$, and $S:D(S)\rightarrow Y$, $T: D(T)\rightarrow Y$ be linear operators satisfying $D(S)\subseteq D(T)$ and $Sx=Tx$ for every $x\in D(S)$. Then $S=T$ if and only if $\ker S=\ker T$ and $\operatorname{ran}S=\operatorname{ran}T$.
\end{lemma}
\begin{proof}
Assume the hypotheses. If $S=T$ then obviously $\ker S=\ker T$ and $\operatorname{ran}S=\operatorname{ran}T$. Conversely, suppose $\ker S=\ker T$ and $\operatorname{ran}S=\operatorname{ran}T$. To prove $S=T$ it suffices to prove that $D(T)\subseteq D(S)$. Let $x\in D(T).$ Then, as $\operatorname{ran}S=\operatorname{ran}T$, there exists $u\in D(S)$ such that $Tx=Su$ and hence by hypotheses, $Tx=Su=Tu$. Thus, by linearity it follows that $x-u\in \ker T$ and since $\ker T=\ker S\subseteq D(S)$ and $D(S)$ a subspace of $X$, then $x=x-u+u\in D(S)$. This proves $D(T)\subseteq D(S)$ which completes proof.
\end{proof}

In the following theorem, we will consider the more general notion of linear operators $A,B$ between two Hilbert spaces $H,K$ that are densely defined closed operators which are adjoints of each other. We have already defined these notions above in the case $H=K$, and definitions are similar in the case $H\not=K$ so their definitions are omitted (but can be found in \cite{80JW}, for instance).  
\begin{theorem}\label{thm:prem:FundResultAdjointsKerRanDecomp}
  If $H,K$ are Hilbert spaces with inner products $\langle \cdot, \cdot \rangle_H$, $\langle \cdot, \cdot \rangle_K$, respectively, $D(A)$ and $D(B)$ are subspaces of $H$ and $K$, respectively, and $A:D(A)\rightarrow K$, $B:D(B)\rightarrow H$ are linear operators satisfying
  \begin{gather}
      \langle Ax, y \rangle_K=\langle x, By\rangle_H,\;\;\text{for all }x\in D(A), y\in D(B),\label{AdjointABRelation}\\
      \ker A +\operatorname{ran}B=H\label{AdjointABKerRanCond},\\
      \ker B +\operatorname{ran}A=K,\label{AdjointBAKerRanCond}
  \end{gather}
  then $A$ and $B$ are densely defined closed operators with closed ranges and are adjoints of each other, i.e.,
  \begin{align}
      A^*=B,\;\;B^*=A.
  \end{align}
\end{theorem}
\begin{proof}
Assume the hypotheses. First, as $A:D(A)\rightarrow K$, $B:D(B)\rightarrow H$ are linear operators then $\ker A, \operatorname{ran}B$ and $\ker B, \operatorname{ran}A$ are subspaces of $H$ and $K$, respectively. Next, we claim that the following orthogonality relations hold:
\begin{gather}
    \ker A=(\operatorname{ran}B)^{\perp},\;\;\ker B=(\operatorname{ran}A)^{\perp},\label{AdjointABKerPerpRanProp}\\
    (\ker A)^{\perp}=\operatorname{ran}B,\;\;(\ker B)^{\perp}=\operatorname{ran}A.\label{AdjointABPerpKerRanProp}
\end{gather}
We now prove this claim. Let $x\in (\operatorname{ran}B)^{\perp}.$ Then, by (\ref{AdjointABKerRanCond}), $x=u+By$ for some $u\in \ker A, y\in D(B)$ and hence together with (\ref{AdjointABRelation}) we have
\begin{gather}
    0=\langle x, By \rangle_H=\langle u+By, By\rangle_H=\langle u, By\rangle_H+\langle By, By\rangle_H\\
    =\langle Au, y\rangle_K+\langle By, By\rangle_H=\langle By, By\rangle_H
\end{gather}
implying $By=0$ so that $x=u\in \ker A$. This proves that $(\operatorname{ran}B)^{\perp}\subseteq \ker A.$ Conversely, let $x\in\ker A$. Then for any $y\in D(B)$ it follows by (\ref{AdjointABRelation}) that
\begin{gather}
    0=\langle Ax, y \rangle_K=\langle x, By\rangle_H,
\end{gather}
which implies $x\in (\operatorname{ran}B)^{\perp}.$ This proves that $\ker A\subseteq (\operatorname{ran}B)^{\perp}$ and thus, $\ker A=(\operatorname{ran}B)^{\perp}$. A similar proof shows that $\ker B=(\operatorname{ran}A)^{\perp}$. Next, it follows from the relationship $\ker A=(\operatorname{ran}B)^{\perp}$ for the subspaces $\ker A, \operatorname{ran}B$ that $(\ker A)^{\perp}=(\operatorname{ran}B)^{{\perp}{\perp}}=\overline{\operatorname{ran}B}\supseteq \operatorname{ran}B$. Conversely, if $v\in (\ker A)^{\perp}$ then by (\ref{AdjointABKerRanCond}) it follows that $v=u+By$ for some $u\in \ker A,y\in D(B)$ and hence by (\ref{AdjointABRelation}) we have
\begin{gather}
    0=\langle u, v\rangle_H=\langle u, u+By\rangle_H=\langle u, u\rangle_H+\langle u, By\rangle_H=\langle u, u\rangle_H+\langle Au, y\rangle_H=\langle u, u\rangle_H
\end{gather}
implying $v=By\in \operatorname{ran}B$. This proves that $(\ker A)^{\perp}\subseteq\operatorname{ran}B$ and thus it follows that $(\ker A)^{\perp}=\overline{\operatorname{ran}B}=\operatorname{ran}B$. A similar proof shows that $(\ker B)^{\perp}=\overline{\operatorname{ran}A}=\operatorname{ran}A$. This completes the proof of our claim.

We will now prove that $A$ and $B$ are closed operators. Suppose $\{x_n\}_{n\in \mathbb{N}}$ is a sequence in $D(A)$ converging in $H$ to $x$ and the sequence $\{Ax_n\}_{n\in \mathbb{N}}$ converges to $z$ in $K$. Then, for each $y\in \ker B$, we have by (\ref{AdjointABRelation}) that
\begin{align}
    0=\langle x, By \rangle_H=\lim_{n\rightarrow \infty } \langle x_n, By\rangle_H=\lim_{n\rightarrow \infty } \langle Ax_n, y\rangle_K=\langle z, y\rangle_K
\end{align}
implying $z\in (\ker B)^{\perp}$. From this and (\ref{AdjointABPerpKerRanProp}) we have $Aw=z$ for some $w\in D(A)$. Hence, for any $y\in D(B)$, we have
\begin{gather}
    \langle x, By \rangle_H=\lim_{n\rightarrow \infty } \langle x_n, By\rangle_H=\lim_{n\rightarrow \infty } \langle Ax_n, y\rangle_K=\langle z, y\rangle_K=\langle Aw, y\rangle_K=\langle w, By\rangle_K
\end{gather}
implying $x-w\in (\operatorname{ran}B)^{\perp}=\ker A\subseteq D(A)$, where the latter equality follows from (\ref{AdjointABKerPerpRanProp}). As $D(A)$ is a subspace of $H$, it follows that $x=x-w+w\in D(A)$ and hence $Ax=Aw=z$. This proves that $A$ is a closed operator. A similar proof shows that $B$ is a closed operator.

We will now prove that $D(A)$ and $D(B)$ are dense in $H$ and $K$, respectively. Let $u\in D(A)^{\perp}$. Then since $(\operatorname{ran}B)^{\perp} =\ker A\subseteq D(A)$ it follows from this and the fact that $\operatorname{ran}B$ is closed, that $D(A)^{\perp}\subseteq \operatorname{ran}B$. Hence, $u=By$ for some $y\in D(B)$ and for every $x\in D(A)$ we have
\begin{gather}
     \langle Ax, y \rangle_H= \langle x, By \rangle_H=\langle x, u \rangle_H=0
\end{gather}
implying $y\in (\operatorname{ran}A)^{\perp}=\ker B$ and thus, $u=By=0$. This proves $D(A)^{\perp}=\{0\}$ so that $\overline{D(A)}=D(A)^{{\perp}{\perp}}=\{0\}^{\perp}=H$. A similar proof shows that $\overline{D(B)}=K$. 

We have proven that $A:D(A)\rightarrow K$ and $B:D(B)\rightarrow H$ are densely defined closed operator with closed ranges. Now by (\ref{AdjointABRelation}) it follows that
\begin{gather}
    D(A)\subseteq D(B^*),\;\; B^*x=Ax,\;\;\text{for all }x\in D(A),\\
    D(B)\subseteq D(A^*),\;\; A^*y=By,\;\;\text{for all }y\in D(B),
\end{gather}
and hence
\begin{gather}
    \ker A\subseteq \ker B^*,\;\;\operatorname{ran}A\subseteq \operatorname{ran}B^*,\\
    \ker B\subseteq \ker A^*,\;\;\operatorname{ran}B\subseteq \operatorname{ran}A^*.
\end{gather}
It follows from this, (\ref{AdjointABKerPerpRanProp}), (\ref{AdjointABPerpKerRanProp}), and (by general properties of adjoints) $\ker B\subseteq (\operatorname{ran}B^*)^{\perp}, \ker A\subseteq (\operatorname{ran}A^*)^{\perp}$, that 
\begin{gather}
    \ker B^*=\ker A,\;\;\operatorname{ran}B^*=\operatorname{ran}A,\\
    \ker A^*=\ker B,\;\;\operatorname{ran}A^*=\operatorname{ran}B.
\end{gather}
Therefore, by Lemma \ref{LemLinearOpsExtEqualityKerRanConds} we conclude that $B^*=A$ and $A^*=B$. This completes the proof.
\end{proof}

\section{\label{sec:MinMaxDAEOps}The minimal and maximal operators associated with linear differential-algebraic equations}

In this section we will consider the linear operators associated to the linear DAEs in the following form:
\begin{gather}
    J \frac{df}{dt}+Hf=Wg\label{def:LinearDAEsInhomo}
\end{gather}
on an interval $I$ with coefficients $J,H,W$ satisfying the following:
\begin{align}
    &I \text{ is an interval (open, closed, half-open, bounded, or unbounded)},\label{HypForab}\\
    &J\in M_n(\mathbb{C}),\;\;J\not=0,\;\;J^*=-J,\label{HypForJ}\\
    &H, W:I\rightarrow M_n(\mathbb{C}),\;\;H,W\in M_n(\mathcal{M}(I)),\label{HypForHOnab}\\
    &W(t)^*=W(t)\geq 0,\;\;\det W(t)\not=0,\;\; \text{for a.e. } t\in I.\label{HypForWOnab}
\end{align}

\begin{definition}\label{def:DiffAlgDAEOp}
The linear operator $\mathcal{L}$ defined on the subspace $D(\mathcal{L})$ of $[\mathcal{M}(I)]^n$ by
\begin{gather}
 D(\mathcal{L})=\left\{f\in [\mathcal{M}(I)]^n:Jf\in[W^{1,1}_{loc}(I)]^n\right\}\label{DefDomainOfGenDAEOperator},\\
    \mathcal{L}:D(\mathcal{L})\to [\mathcal{M}(I)]^n,\\
    \mathcal{L}f=W^{-1}\left(\frac{d}{dt}Jf+ Hf\right),\;\;f\in D(\mathcal{L})\label{DefGenDAEOperator},
\end{gather}
is called the \textit{differential-algebraic (DA) operator} associated with $I, J, H, W$.
\end{definition}

The following lemma gives an alternative characterization of $D(\mathcal{L})$ and $\mathcal{L}$ in terms the orthogonal projection of $\mathbb{C}^n$ onto the range of $J$. To do this we will use the Moore-Penrose pseudoinverse \cite{19FIS, 76NV, 03BG} of the matrix $J$.
\begin{lemma}\label{lem:AltCharDAEOpAndItsDomain}
Let $J^+$ denote the Moore-Penrose pseudoinverse of $J$, i.e., the unique matrix $J^+\in M_n(\mathbb{C})$ satisfying the four (Moore-Penrose) conditions: $JJ^+J=J$, $J^+JJ^+=J^+, (J^+J)^*=J^+J, (JJ^+)^*=JJ^+$. Then $J^+J=JJ^+$ and $JJ^+$ is the orthogonal projection of $\mathbb{C}^n$ onto $\operatorname{ran}J$. Moreover, 
\begin{gather}
    D(\mathcal{L})=\left\{f\in [\mathcal{M}(I)]^n:J^+Jf\in[W^{1,1}_{loc}(I)]^n\right\},\\
    \mathcal{L}f=W^{-1}\left[J\frac{d}{dt}(J^+Jf)+ Hf\right],\;\;f\in D(\mathcal{L})
\end{gather}
\end{lemma}
\begin{proof}
Let $J^+$ denote the Moore-Penrose pseudoinverse of $J$. Then it follows immediately from this that $J^+J$ and $JJ^+$ are the orthogonal projections of $\mathbb{C}^n$ onto $\operatorname{ran}J^*$ and $\operatorname{ran}J$, respectively. As $J^*=-J$ this implies $\operatorname{ran}J^*=\operatorname{ran}J$ so that $J^+J=JJ^+$. Next, suppose $f\in [\mathcal{M}(I)]^n$. Then, since $JJ^+J=J$ and $J$ is independent of $t$, it follows that
\begin{gather}
    Jf\in [W^{1,1}_{loc}(I)]^n \iff J^+Jf\in[W^{1,1}_{loc}(I)]^n.
\end{gather}
The proof of the lemma now follows immediately from this.
\end{proof}
\begin{example}\label{ex:1DPCMaxwellsDAEOp}
As an example, consider the periodic Maxwell's DAEs (\ref{def:MaxwellsPeriodicDAEs}). The corresponding DA operator $\mathcal{L}$ is associated with the interval $I=\mathbb{R}$ and the coefficients $J,H,W$ defined by (\ref{def:MaxwellsPeriodicDAEsJ}), (\ref{def:MaxwellsPeriodicDAEsH}), (\ref{def:MaxwellsPeriodicDAEsW}). Let $J^+$ denote the Moore-Penrose pseudoinverse of $J$. Then $J^+J=JJ^+$ is the orthogonal projection of $\mathbb{C}^6$ onto $\operatorname{ran}J$ and, in this example,
\begin{gather}
    J^+=-J,\;J^+J=JJ^+=\begin{bmatrix}
    1 & 0  & 0 & 0 & 0 & 0\\
    0 & 1  & 0 & 0 & 0 & 0\\
    0 & 0  & 0 & 0 & 0 & 0\\
    0 & 0  & 0 & 1 & 0 & 0\\
    0 & 0  & 0 & 0 & 1 & 0\\
    0 & 0  & 0 & 0 & 0 & 0
    \end{bmatrix},\\
    J^+Jf(z)=J^+J\begin{bmatrix}E(z)& H(z)\end{bmatrix}^T=\begin{bmatrix}E_1(z)& E_2(z)&0&H_1(z)&H_2(z)&0\end{bmatrix}^T.
\end{gather}
Thus, from Lemma \ref{lem:AltCharDAEOpAndItsDomain}, we see more clearly what we meant in the comment made in Sec.\ \ref{sec:1DPCs} that it is the tangential components $f_1(z)=\begin{bmatrix}
    E_1(z)&
    E_2(z)&
    H_1(z)&
    H_2(z)
    \end{bmatrix}^T$
which give rise to the differential part of the periodic Maxwell's DAEs (\ref{def:MaxwellsPeriodicDAEs}). In this regard with this example in mind, see Proposition \ref{PropKeySolvabilityLocalIndex1DAEs}. This concludes the example.
\end{example}

Next, define the Hilbert space $L^2(I;W)$ with inner product $\langle\cdot,\cdot\rangle_W$ by
\begin{align}
    &L^2(I;W)=\left\{f\in [\mathcal{M}(I)]^n:\int_{I}\langle W(t)f(t),f(t)\rangle dt<\infty\right\},\label{def:WeightedHilbertSpL2Funcs}\\
    &\langle f,g \rangle _W=\int_{I}\langle W(t)f(t),g(t)\rangle dt,\;\;f,g\in L^2(I;W),\label{def:WeightedHilbertSpL2FuncsInnerProd}
\end{align}
and the subspace $L^2_{loc}(I;W)$ of $\mathcal{M}(I)$ by
\begin{align}
    L^2_{loc}(I;W)=\{f\in  [\mathcal{M}(I)]^n:f\in L^2([a,b];W),\text{ for every compact interval }[a,b]\subseteq I\}.
\end{align}
In particular,
\begin{align}
     L^2_{loc}(I;W)=L^2(I;W),\text{ if } I \text{ is a compact interval.}
\end{align}

\begin{remark}\label{rem:IsometricIsomorphismWeightedHilbertSpace}
It easy to verify that left multiplication by $W^{-1/2}$ is an isometric isomorphism between the Hilbert spaces $[L^2(I)]^n$ and $L^2(I;W)$ from which it follows that
\begin{gather}
    L^2(I;W)=W^{-1/2}[L^2(I)]^n,\;L^2_{loc}(I;W)=W^{-1/2}[L^2_{loc}(I)]^n.
\end{gather}
In the case that $\det J=0$, there will be an alternative characterizations of $L^2(I;W)$ and $L^2_{loc}(I;W)$ in Lemma \ref{LemCharL2WSpace} that plays a key role in our paper.
\end{remark}

\begin{definition}\label{DefMinMaxOpGenerByTheDAEsOp}
The \textit{maximal operator $L:D(L)\rightarrow L^2(I;W)$ generated by $\mathcal{L}$} and the \textit{minimal operator $L_0':D(L_0')\rightarrow L^2(I;W)$ generated by  $\mathcal{L}$} are defined by
\begin{gather}
    D(L)=\{f\in L^2(I;W):f\in D(\mathcal{L}),\mathcal{L}f\in L^2(I;W)\},\\
    Lf=\mathcal{L}f,\;\;\text{for }f\in D(L),\\
    D(L_0')=\{f\in D(L):Jf \text{ has compact support contained in the interior of }I\},\\
    L_0'f=\mathcal{L}f,\;\;\text{for }f\in D(L_0').
\end{gather}
\end{definition}

\begin{notation}
When we need to be explicit about the dependence of $\mathcal{L}, L$ or $L_0'$ on $H$ and/or $J$ we will use the subscript $(\cdot)_H$ or $(\cdot)_{J,H}$ with these operators, e.g., $\mathcal{L}_{H}, L_{H}$ or $(L_0')_{H}$ and for the latter, we will just write $L_{H,0}'$ instead. Similarly, $(L_0')_{J,H}$ will be written instead as $L_{J,H,0}'$.
\end{notation}

As we shall see later, the following lemma and theorem are fundamental in the developing the spectral theory associated with the operator $L$ on the Hilbert space $L^2(I;W)$.
\begin{lemma}\label{LemDomainsAreSubspacesForL0primeAndLWhichAreWellDefLinearOps}
The sets $D(L_0')$ and $D(L)$ are subspaces of the Hilbert space $L^2(I;W)$. Moreover, $L_0':D(L_0')\rightarrow L^2(I;W)$ and $L:D(L)\rightarrow L^2(I;W)$ are linear operators with
\begin{gather}
    D(L_0')\subseteq D(L)\subseteq D(\mathcal{L})\cap L^2(I;W),\\
    L_0'f=Lf=\mathcal{L}f,\;\;f\in D(L_0').
\end{gather}
\end{lemma}
\begin{proof}
The proof is straightforward and so will be omitted.
\end{proof}

\begin{remark}
It follows from (\ref{DefFuncConjTransp}) and (\ref{DefFuncSpaceConjTransp}) that the linear operators $\mathcal{L}_{H^*}, L_{H^*},$ and $L_{H^*,0}'$ can be defined as above in terms of $H^*$ instead of $H$ and, in particular, Lemma \ref{LemDomainsAreSubspacesForL0primeAndLWhichAreWellDefLinearOps} is true too for these operators.
\end{remark}

\begin{theorem}\label{ThmAdjointCalcForMinMaxOpsRelatedToSymmetricOps}
  a) For $f\in D(L_0')$ and $g\in D(L_{H^*})$ we have
  \begin{align}
      \langle L_0'f,g\rangle_W = \langle f,L_{H^*}g\rangle_W.
  \end{align}
  b) For $f\in D(L_0')$ and $g\in D(L_{H^*,0}')$ we have
  \begin{align}
       \langle L_0'f,g\rangle_W = \langle f,L_{H^*,0}'g\rangle_W.
  \end{align}
\end{theorem}
\begin{proof}
a): Let $f\in D(L_0')$. Then there exists a compact interval $[a,b]\subseteq I$ such that $(Jf)(t)=0$ for every $t\in I\setminus (a,b)$ and, in particular, $(Jf)(a)=(Jf)(b)=0$. Hence, for any $g\in D(L_{H^*})$ and by Lemma \ref{lem:AltCharDAEOpAndItsDomain}, we have
\begin{align*}
      \langle L_0'f,g\rangle_W &= \langle \mathcal{L}f,g\rangle_W\\
      &=\int_{I}\left\langle \left[\frac{d}{dt}Jf(t)+ H(t)f(t)\right],g(t)\right\rangle dt\\
      &=\int_{I}\left\langle \frac{d}{dt}Jf(t),J^+Jg(t)\right\rangle +\left\langle H(t)f(t),g(t)\right\rangle dt\\
      &=\int_{I}\frac{d}{dt}\left\langle Jf(t),J^+Jg\right\rangle-\left\langle Jf(t),\frac{d}{dt}J^+Jg(t)\right\rangle+\left\langle f(t),H(t)^*g(t)\right\rangle dt\\
      &=\int_{I}\frac{d}{dt}\left\langle Jf(t),J^+Jg(t)\right\rangle+\left\langle W(t)f(t),W(t)^{-1}\left[\frac{d}{dt}Jg(t)+H(t)^*g(t)\right]\right\rangle dt\\
      &=\int_{I}\frac{d}{dt}\left\langle Jf(t),J^+Jg(t)\right\rangle dt+\langle f,L_{H^*}g\rangle_W\\
      &=\left\langle (Jf)(b),J^+(Jg)(b)\right\rangle-\left\langle (Jf)(a),J^+(Jg)(a)\right\rangle+\langle f,L_{H^*}g\rangle_W\\
      &=\langle f,L_{H^*}g\rangle_W.
  \end{align*}
  b): This follows immediately from part a) since $D(L_{H^*,0}')\subseteq D(L_{H^*})$ and $L_{H^*,0}'g=L_{H_*}g$ if $g\in D(L_{H^*,0}')$.
\end{proof}

Consider the case that $\det J=0$. Let $V\in M_n(\mathbb{C})$ be any unitary matrix, i.e.,
\begin{align}
    V\in M_n(\mathbb{C}),\;\;V^*=V^{-1},\label{HypOnVMatrix}
\end{align}
such that $V^{-1}JV=[J_{ij}]_{i,j=1,2}$ has the $2\times2$ block partitioned matrix form
\begin{align}
    V^{-1}JV=[J_{ij}]_{i,j=1,2}=\begin{bmatrix}
               J_{11} & J_{12} \\
               J_{21} & J_{22} \\
             \end{bmatrix}=\begin{bmatrix}
               J_{11} & 0 \\
               0 & 0 \\
             \end{bmatrix},\quad \det(J_{11})\neq 0\label{BlockStructJ},
\end{align}
and, more precisely,
\begin{align}
    J_{ij}\in M_{n_i\times n_j}(\mathbb{C}), i,j=1,2;\;\;J_{ij}=0,\;\;(i,j)\not = (1,1),\;\;\det(J_{11})\neq 0,
\end{align}
where we define $n_1,n_2\in \mathbb{N}$ by
\begin{align}
    n_1:=\operatorname{rank} J,\;\;n_2:=\dim \ker J=\operatorname{nullity}(J)=n-n_1.
\end{align}

\begin{remark}
It should be noted that each and every such $V$ can be constructed in the following manner: Let $v_1,\ldots, v_{n_1}$ be an orthonormal basis for $\operatorname{ran} J$ and $v_{n_1+1},\ldots, v_{n_2}$ an orthonormal basis for $\operatorname{ker} J$. Then the $n\times n$ column matrix $V=\begin{bmatrix}
v_1|\cdots|v_{n_1}|v_{n_1+1}|\cdots|v_{n_2}
\end{bmatrix}$ satisfies (\ref{HypOnVMatrix}) and (\ref{BlockStructJ}). The key point in this remark is the converse of this statement is also true and this gives insight into our reason for introducing such a unitary matrix $V$ above.
\end{remark}

Now block partition the matrices $V^{-1}HV=[H_{ij}]_{i,j=1,2}$ and $V^{-1}WV=[W_{ij}]_{i,j=1,2}$ conformal to the block structure of $V^{-1}JV$ in (\ref{BlockStructJ}),
\begin{align}
    &V^{-1}HV=[H_{ij}]_{i,j=1,2}=\begin{bmatrix}
    H_{11} & H_{12}\\
    H_{21} & H_{22}
    \end{bmatrix}\in M_{n}(\mathcal{M}(I)),\label{BlockStructH}\\
     &V^{-1}WV=[W_{ij}]_{i,j=1,2}=\begin{bmatrix}
    W_{11} & W_{12}\\
    W_{21} & W_{22}
    \end{bmatrix}\in M_{n}(\mathcal{M}(I)),\label{BlockStructW}
\end{align}
and, more precisely,
\begin{align}
    H_{ij}, W_{ij}\in M_{n_i\times n_j}(\mathcal{M}(I)),\;\;i,j=1,2.
\end{align}
If, in addition, $H_{22}^{-1}\in M_{n_2}(\mathcal{M}(I))$ [equivalently, $\det H_{22}(t)\not = 0$ for a.e.\ $t\in I$] then, for the Schur complement (see \cite{05FZ}) of $V^{-1}HV$ with respect to $H_{22}$, we have
\begin{align}
    H/H_{22}:=H_{11}-H_{21}H_{22}^{-1}H_{12}\in M_{n_1}(\mathcal{M}(I)).
\end{align}

\begin{lemma}\label{LemWInvWSqrtProperties}
Suppose $W$ satisfies the hypotheses (\ref{HypForHOnab}) and (\ref{HypForWOnab}). Then the inverse $W(t)^{-1}$ of $W(t)$ exists for a.e.\ $t\in I$ (and setting it to say the identity matrix $I_n$ when it doesn't exist), defines the function $W^{-1}:I\rightarrow M_n(\mathbb{C})$ which satisfies  $W^{-1}(t)=W(t)^{-1}$ for a.e.\ $t\in I$. In addition, $W^{-1}$ satisfies the same hypotheses (\ref{HypForHOnab}) and (\ref{HypForWOnab}) as $W$, i.e.,
\begin{align}
    &W^{-1}:I\rightarrow M_n(\mathbb{C}),\;\;W^{-1}\in M_n(\mathcal{M}(I)),\\
    &W^{-1}(t)^*=W^{-1}(t)\geq 0,\;\;\det W^{-1}(t)\not=0,\;\;\text{for a.e.\ }t\in I.
\end{align}
Furthermore, the positive square root $W(t)^{\pm 1/2}$ of $W^{\pm 1}(t)$ exists for a.e.\ $t\in I$ (and setting it to say the identity matrix $I_n$ when it doesn't exist), defines the function $W^{\pm 1/2}:I\rightarrow M_n(\mathbb{C})$ which satisfies  $W^{\pm 1/2}(t)=W(t)^{\pm 1/2}$ for a.e.\ $t\in I$. Moreover, $W^{\pm 1/2}$ satisfies the same hypotheses (\ref{HypForHOnab}) and (\ref{HypForWOnab}) as $W$, i.e.,
\begin{align}
    &W^{\pm 1/2}:I\rightarrow M_n(\mathbb{C}),\;\;W^{\pm 1/2}\in M_n(\mathcal{M}(I)),\\
    &W^{\pm 1/2}(t)^*=W^{\pm 1/2}(t)\geq 0,\;\;\det W^{\pm 1/2}(t)\not=0,\;\;\text{for a.e.\ }t\in I.
\end{align}
\end{lemma}
\begin{proof} The proof is obvious from the elementary theory of matrices and that for positive definite matrices. 
\end{proof}

\begin{lemma}\label{LemWSchurComplProperties}
Let $W$ satisfy the hypotheses (\ref{HypForHOnab}) and (\ref{HypForWOnab}) and $V$ satisfy the hypotheses (\ref{HypOnVMatrix}). Then, with respect to the $2\times 2$ block matrix partitioning $V^{-1}WV=[W_{i,j}]_{i,j=1,2}$ as in (\ref{BlockStructW}), we have $W_{11}^{-1}\in M_{n_1}(\mathcal{M}(I))$ and $W_{22}^{-1}\in M_{n_2}(\mathcal{M}(I))$, and the Schur complements of $V^{-1}WV$ with respect to $W_{11}$ and $W_{22}$, i.e.,
\begin{align}
    &W/W_{11}:=W_{22}-W_{21}W_{11}^{-1}W_{12},\\
    &W/W_{22}:=W_{11}-W_{12}W_{22}^{-1}W_{21},
\end{align}
respectively, have the following properties for $i=1,2$:
\begin{align}
    &W/W_{ii}\in M_{n_i}(\mathcal{M}(I)),\\
    &(W/W_{ii})(t)^*=(W/W_{ii})(t)\geq 0,\;\;\det (W/W_{ii})(t)\not=0,\;\;\text{for a.e.\ }t\in I,\\
    &0\leq (W/W_{ii})(t)\leq W_{jj}(t)\;\;\text{for a.e.\ }t\in I,\text{ for each }j=1,2,\;j\not=i.
\end{align}
\end{lemma}
\begin{proof}
The proof is immediate from elementary properties of $2\times 2$ block partitioned positive definite matrices and their Schur complements (see, for instance, \cite{05FZ}) since $W(t)^*=W(t)\geq 0$ and $\det W(t)\not=0$ for a.e.\ $t\in I$ implies $W(t)$ is a positive definite matrix for a.e.\ $t\in I$ which implies so is $V^{-1}W(t)V=[W_{i,j}(t)]_{i,j=1,2}$ and hence so are the blocks $W_{11}(t)$ and $W_{22}(t)$ with $W_{12}(t)^*=W_{21}(t)$. The results now follows immediately from this.
\end{proof}

We will now characterize, in the case $\det J=0$, the spaces $L^2(I;W)$ and $L^2_{loc}(I;W)$.
\begin{lemma}\label{LemCharL2WSpace}
Suppose $I,W$ satisfy (\ref{HypForab}), (\ref{HypForHOnab}), and (\ref{HypForWOnab}), $V$ satisfy the hypotheses (\ref{HypOnVMatrix}), and block matrix partition $V^{-1}WV=[W_{i,j}]_{i,j=1,2}$ as in (\ref{BlockStructW}). Then
\begin{align}
    f\in (\mathcal{M}(I))^n \iff V^{-1}f=\begin{bmatrix}
    f_1\\
    f_2
    \end{bmatrix},\;\;f_i\in (\mathcal{M}(I))^{n_i},\;\;i=1,2.\label{DefMeasurableVectorBlockMatrixDecomp}
\end{align}
Furthermore,
\begin{gather}
    f\in L^2(I;W) \label{fInL2W}\\
    \iff \notag \\
    (W/W_{11})^{1/2}f_2\in (L^2(I))^{n_2},\;\;W_{11}^{-1/2}(W_{11}f_1+W_{12}f_2)\in (L^2(I))^{n_1}\label{fInL2WEquivCond1}\\
    \iff \notag\\
    (W/W_{22})^{1/2}f_1\in (L^2(I))^{n_1},\;\;W_{22}^{-1/2}(W_{21}f_1+W_{22}f_2)\in (L^2(I))^{n_2}.\label{fInL2WEquivCond2}
\end{gather}
Similarly, the statement remains true if we replace $L^2(I;W), (L^2(I))^{n_1},(L^2(I))^{n_2}$ in (\ref{fInL2W}), (\ref{fInL2WEquivCond1}), (\ref{fInL2WEquivCond2}) by $L^2_{loc}(I;W), (L^2_{loc}(I))^{n_1}, (L^2_{loc}(I))^{n_2},$ respectively.
\end{lemma}
\begin{proof}
The proof of the statement (\ref{DefMeasurableVectorBlockMatrixDecomp}) is obvious and so will be omitted. Also, once we proof the equivalence of the statements (\ref{fInL2W}), (\ref{fInL2WEquivCond1}), and (\ref{fInL2WEquivCond2}) then the equivalence of these statements in the local, i.e., ``loc," case follows immediately. Now we prove that (\ref{fInL2W}) iff (\ref{fInL2WEquivCond2}),  but we will omit the proof of (\ref{fInL2W}) iff (\ref{fInL2WEquivCond1}) as it is similar. By Lemma \ref{LemWInvWSqrtProperties},  we have the block factorization
\begin{align}
    V^{-1}WV=[W_{ij}]_{i,j=1,2}=\begin{bmatrix}
    I_{n_1} & W_{12}W_{22}^{-1}\\
    0 & I_{n_2}
    \end{bmatrix}\begin{bmatrix}
    W/W_{22} & 0\\
    0 & W_{22}
    \end{bmatrix}\begin{bmatrix}
    I_{n_1} & 0\\
    W_{22}^{-1}W_{21} & I_{n_2}
    \end{bmatrix}.
\end{align}
Next, it follows from this and Lemma \ref{LemWSchurComplProperties} together with Lemma \ref{LemWInvWSqrtProperties} applied to $W/W_{22}$ that for any $f\in (\mathcal{M}(I))^n$ we have
\begin{gather}
    \int_{I}(W(t)f(t),f(t))dt=\int_{I}(V^{-1}W(t)V[V^{-1}f(t)],[V^{-1}f(t)])dt\\
    =\int_{I}\left(\begin{bmatrix}
    W/W_{22} & 0\\
    0 & W_{22}
    \end{bmatrix}\begin{bmatrix}
    I_{n_1} & 0\\
    W_{22}^{-1}W_{21} & I_{n_2}
    \end{bmatrix}\begin{bmatrix}
    f_1\\
    f_2
    \end{bmatrix},\begin{bmatrix}
    I_{n_1} & 0\\
    W_{22}^{-1}W_{21} & I_{n_2}
    \end{bmatrix}\begin{bmatrix}
    f_1\\
    f_2
    \end{bmatrix}\right)(t)dt\\
    \int_{I}((W/W_{22})(t)f_1(t),f_2(t))+(W_{22}(t)[W_{22}^{-1}(t)W_{21}(t)f_1(t)+f_2],W_{21}(t)f_1(t)+f_2(t)) dt\\
    =\int_{I}||[(W/W_{22})^{1/2}f_1](t)||^2dt+\int_{I}||[W_{22}^{-1/2}(W_{21}f_1+W_{22}f_2](t))||^2dt.
\end{gather}
The proof that (\ref{fInL2W}) iff (\ref{fInL2WEquivCond2}) follows immediately from this, which completes the proof.
\end{proof}

\begin{definition}[Index-1 hypotheses]\label{DefIndex1Hyp}
The following set of hypotheses are called the \textit{local index-1 hypotheses} for $H, W$ with respect to $J$ on the interval $I$:
\begin{align}
    \text{If } \det J\not = 0 \text{ then } H, W\in M_n(L^1_{loc}(I)).\label{HypIndex1LocalPartNeg1}
\end{align}
If $\det J = 0$ then
\begin{gather}
    H_{22}^{-1}\in M_{n_2}(\mathcal{M}(I)),\label{HypIndex1LocalPartNeg0Pt5}\\
    H_{12}H_{22}^{-1}W_{22}^{1/2}\in M_{n_1\times n_2}(L^{2}_{loc}(I)),\label{HypIndex1LocalPart2}\\
    H/H_{22}, W_{11}\in M_{n_1}(L^1_{loc}(I)),\label{HypIndex1LocalPart3}\\
    W_{22}^{1/2}H_{22}^{-1}W_{22}^{1/2}\in M_{n_2}(L^{\infty}_{loc}(I))\label{HypIndex1LocalPart0},\\
    W_{22}^{1/2}(W_{22}^{-1}W_{21}-H_{22}^{-1}H_{21})\in M_{n_2\times n_1}(L^2_{loc}(I)),\label{HypIndex1LocalPart1}\\
    W/W_{22}\in M_{n_1}(L^1_{loc}(I))\label{HypIndex1LocalPart1Pt5}
\end{gather}
Similarly, if we drop in the hypotheses above the ``loc" then these are called the \textit{index-1 hypotheses}.
\end{definition}

The next lemma and corollary gives useful simplification of these hypotheses.
\begin{lemma}\label{lem:SimplifiedIndex1Hyp}
Suppose $\det J=0$. Then the local index-$1$ hypotheses for $H,W$ with respect to $J$ on the interval $I$ are true if and only if the following conditions are satisfied:
\begin{gather}
    H_{22}^{-1}\in M_{n_2}(\mathcal{M}(I)),\;W_{22}^{1/2}H_{22}^{-1}W_{22}^{1/2}\in M_{n_2}(L^{\infty}_{loc}(I))\\
    H/H_{22}, W_{11}\in M_{n_1}(L^1_{loc}(I)),\\
    H_{12}H_{22}^{-1}W_{22}^{1/2}\in M_{n_1\times n_2}(L^{2}_{loc}(I)),\;W_{22}^{1/2}H_{22}^{-1}H_{21}\in M_{n_2\times n_1}(L^2_{loc}(I)).
\end{gather}
Similarly, the statement with the ``local" and ``loc" dropped is true.
\end{lemma}
\begin{proof}
Suppose $\det J=0$. Then by hypotheses on $W$ we have
\begin{gather}
    W_{11}-W/W_{22}=W_{12}W_{22}^{-1}W_{21}=(W_{22}^{-1/2}W_{21})^*
(W_{22}^{-1/2}W_{21})
\end{gather}
and $0\leq W/W_{22}\leq W_{11}$.
Assume that $W_{11}\in M_{n_1}(L^1_{loc}(I))$. Then this implies $W/W_{22}\in M_{n_1}(L^1_{loc}(I))$ and $W_{22}^{-1/2}W_{21}\in M_{n_2\times n_1}(L^2_{loc}(I))$.
And from this it follows that
\begin{gather}
   \;W_{22}^{1/2}(W_{22}^{-1}W_{21}-H_{22}^{-1}H_{21})\in M_{n_2\times n_1}(L^2_{loc}(I))\\
    \iff W_{22}^{1/2}H_{22}^{-1}H_{21}\in M_{n_2\times n_1}(L^2_{loc}(I)).
\end{gather}
The proof of the lemma now follows immediately from this.
\end{proof}

\begin{corollary}\label{cor:SimplifiedIndex1HypImplyItForAdjoints}
The (local) index-$1$ hypotheses for $H,W$ with respect to $J$ on the interval $I$ are true if and only if the (local) index-$1$ hypotheses for $H^*,W$ with respect to $J$ on the interval $I$ are true.
\end{corollary}
\begin{proof}
The proof of this corollary follows immediately from considering matrix adjoints of the conditions in Lemma \ref{lem:SimplifiedIndex1Hyp} if $\det J=0$ or those in Def.\ \ref{DefIndex1Hyp} if $\det J\not =0.$
\end{proof}

The reason for such hypotheses will become clear as we move forward in this section, but the next proposition gives the main reason why.

\begin{proposition}\label{PropKeySolvabilityLocalIndex1DAEs}
Suppose $I,J,H,W$ satisfy (\ref{HypForab})--(\ref{HypForWOnab}) and, in the case $\det J=0$, that $V$ satisfies the hypotheses (\ref{HypOnVMatrix}), $V^{-1}JV=[W_{i,j}]_{i,j=1,2}$ has the block form (\ref{BlockStructJ}), and $V^{-1}HV=[H_{i,j}]_{i,j=1,2}, V^{-1}WV=[W_{i,j}]_{i,j=1,2}$ have the conformal block structure in (\ref{BlockStructH}), (\ref{BlockStructW}), respectively. Then the following statements are true:\\
\noindent (i) If (\ref{HypIndex1LocalPartNeg0Pt5}) then
\begin{gather}
      \mathcal{L}f=g\label{InhomogEqForMathCalL}\\
      \iff \notag\\
       g\in [\mathcal{M}(I)]^n,\label{InhomogEqForMathCalLEquivSys1}\\
      f_1\in [W^{1,1}_{loc}(I)]^{n_1},\label{InhomogEqForMathCalLEquivSys2}\\
 J_{11}\frac{df_1}{dt}+H/H_{22}f_1=F,\label{InhomogEqForMathCalLEquivSys3}\\
 f_2=H_{22}^{-1}(W_{21}g_1+W_{22}g_2)-H_{22}^{-1}H_{21}f_1,\label{InhomogEqForMathCalLEquivSys4}
\end{gather}
where $F\in [\mathcal{M}(I)]^{n_2}$ is defined by
\begin{align}
    F=(W_{11}g_1+W_{12}g_2)-H_{12}H_{22}^{-1}(W_{21}g_1+W_{22}g_2)\label{DefFfunction},
\end{align}
and $f,g, f_1,f_2, g_1, g_2$ are related by
\begin{align}
    V^{-1}f=\begin{bmatrix}
    f_1\\
    f_2
    \end{bmatrix},\;\;V^{-1}g=\begin{bmatrix}
    g_1\\
    g_2
    \end{bmatrix}.\label{InhomogEqForMathCalLEquivSys5}
\end{align}\\
\noindent (ii) If (\ref{HypIndex1LocalPartNeg1}) or (\ref{HypIndex1LocalPartNeg0Pt5})--(\ref{HypIndex1LocalPart3}) then for any $g\in L^2_{loc}(I;W),$ $t_0\in I,$ and $f_0\in \operatorname{ran} J$ there is a unique solution $f\in D(\mathcal{L})$ to the IVP
\begin{align}
    \mathcal{L}f=g,\;\;(Jf)(t_0)=f_0.\label{DAEIVP}
\end{align}
In addition, in the case $\det J\not =0$ or if (\ref{HypIndex1LocalPart0})--(\ref{HypIndex1LocalPart1Pt5}) in the case $\det J=0$, then \begin{align}
    f\in L^2_{loc}(I;W).\label{L2LocIWInclusionForCompact}
\end{align}
\end{proposition}
\begin{proof}
(i): Assume the hypotheses. ($\Rightarrow$): Suppose (\ref{InhomogEqForMathCalL}) is true. Then $f\in D(L)$ and so $Jf\in [W^{1,1}_{loc}(I)]^n$. This implies that $f_1\in [W^{1,1}_{loc}(I)]^{n_1}, f_2\in [\mathcal{M}(I)]^{n_2},$ where $f_1,f_2$ are related to $f$ by (\ref{InhomogEqForMathCalLEquivSys5}). Next, since $\mathcal{L}f=g$ then $g\in [\mathcal{M}(I)]^{n}$ and $\frac{d}{dt}(Jf)+Hf=W\mathcal{L}f=Wg$ (with equality in $[\mathcal{M}(I)]^{n}$) which yields the system of equations:
\begin{align}
    &J_{11}\frac{df_1}{dt}+H_{11}f_1+H_{12}f_2=W_{11}g_1+W_{12}g_2,\\
    &H_{21}f_1+H_{22}f_2 = W_{21}g_1+W_{22}g_2,
\end{align}
where $g_1,g_2$ are related to $f$ by (\ref{InhomogEqForMathCalLEquivSys5}). Solving these equations for $f_2$ in terms of $f_1,g_1,g_2$ yields the equivalent system of equations:
\begin{align}
    &J_{11}\frac{df_1}{dt}+H/H_{22}f_1=W_{11}g_1+W_{12}g_2-H_{12}H_{22}^{-1}(W_{21}g_1+W_{22}g_2),\\
    &f_2 = H_{22}^{-1}(W_{21}g_1+W_{22}g_2)-H_{22}^{-1}H_{21}f_1,
\end{align}
which are equivalent to the system of equations (\ref{InhomogEqForMathCalLEquivSys3}) and (\ref{InhomogEqForMathCalLEquivSys4}), where $F$ is defined by (\ref{DefFfunction}). Thus, we have proven that (\ref{InhomogEqForMathCalL}) implies (\ref{InhomogEqForMathCalLEquivSys1})--(\ref{DefFfunction}), with $f,g$ and $f_1,f_2, g_1, g_2$ are related by (\ref{InhomogEqForMathCalLEquivSys5}). ($\Leftarrow$): Conversely, suppose $f,g$ satisfy (\ref{InhomogEqForMathCalLEquivSys1})--(\ref{DefFfunction}), where $f,g$ and $f_1,f_2, g_1, g_2$ are related by (\ref{InhomogEqForMathCalLEquivSys5}). Then one can verify that $f\in D(\mathcal{L})$ and $\frac{d}{dt}(Jf)+Hf=Wg$ (with equality in $[\mathcal{M}(I)]^{n}$), which implies $\mathcal{L}f=g$. This completes the proof of (i). (ii): Assume (\ref{HypIndex1LocalPartNeg1}), in the case $\det J\not=0$, or, in the case $\det J=0$, that (\ref{HypIndex1LocalPartNeg0Pt5})--(\ref{HypIndex1LocalPart3}). Let $g\in L^2_{loc}(I;W), t_0\in I, f_0\in \operatorname{ran}J$.

Consider the first case in which $\det J\not=0$ and hypotheses (\ref{HypIndex1LocalPartNeg1}) are true. Then $Wg\in [L^1_{loc}(I)]^n$ since, for any compact interval $[a,b]\subseteq I$, we have $g\in L^2([a,b];W)$ and $W\in M_n(L^1([a,b]))$ by hypothesis so it follows by Holder's inequality  that
\begin{gather}
    \int_{[a,b]}||W(t)g(t)||dt=\int_{[a,b]}(W(t)g(t),W(t)g(t))^{1/2}dt\\
    \leq \int_{[a,b]}||W(t)||^{1/2}||W(t)^{1/2}g(t)||dt\\
    \leq \int_{[a,b]}(||W(t)||^{1/2})^2dt^{1/2}\int_{[a,b]}||W(t)^{1/2}g(t)||^2dt^{1/2}\\
    =\left[\int_{[a,b]}||W(t)||dt\int_{[a,b]}(W(t)g(t),g(t))dt\right]^{1/2}<\infty.
\end{gather}
It also follows from this that $J^{-1}Wg\in [L^1_{loc}(I)]^n$ and by hypotheses that $J^{-1}H\in M_n(L^1_{loc}(I))$. By Theorem \ref{ThmODEIVPExistenceUniquenessSolns}, there exists a unique solution $f\in [W^{1,1}_{loc}(I)]^n=D(\mathcal{L})$ to the ODE IVP
\begin{align}
    \frac{d}{dt}f+J^{-1}Hf=J^{-1}Wg,\;\;f(t_0)=J^{-1}f_0\label{ODEIVPOnI}
\end{align}
on $I$. The proof of the statement (ii) in the case $\det J\not=0$ now follows immediately from this with the exception that we still need to prove $f\in L^2_{loc}(I;W)$. But this follows from the fact that $f\in [W^{1,1}_{loc}(I)]^n\subseteq [L^{\infty}_{loc}(I)]^n$ so by Holder's inequality  we have for any compact interval $[a,b]\subseteq I$,
\begin{gather}
\int_{[a,b]}(W(t)f(t),f(t))dt\leq \int_{[a,b]}||W(t)||||f(t)||^2dt\\
\leq \left(\esssup_{t\in I}||f(t)||\right)^2\int_{[a,b]}||W(t)||dt<\infty,
\end{gather}
implying that $f\in L^2_{loc}(I;W)$.

Consider now, the second case in which $\det J=0$ and hypotheses (\ref{HypIndex1LocalPartNeg0Pt5})--(\ref{HypIndex1LocalPart3}) are true. First, it follows from the block structure (\ref{BlockStructJ}) of $V^{-1}JV$ that
\begin{align}
    V^{-1}f_0=\begin{bmatrix}
    (f_1)_0\\
    0
    \end{bmatrix}
\end{align}
for a unique $(f_1)_0\in \mathbb{C}^{n_1}$. Next, by hypothesis (\ref{HypIndex1LocalPart3}) we know that $W_{11}\in M_{n_1}(L^1_{loc}(I))$ from which it follows that $W_{11}^{1/2}\in M_{n_1}(L^2_{loc}(I))$. This together with $g\in L^2_{loc}(I;W)$, Lemma \ref{LemCharL2WSpace}, and hypothesis (\ref{HypIndex1LocalPart2}) implies by Holder's inequality  that $F$ defined in (\ref{DefFfunction}) satisfies
\begin{gather}
F=W_{11}^{1/2}[W_{11}^{-1/2}(W_{11}g_1+W_{12}g_2)]-H_{12}H_{22}^{-1}W_{22}^{1/2}[W_{22}^{-1/2}(W_{21}g_1+W_{22}g_2)],\\
    F\in [L^1_{loc}(I)]^{n_1}.
\end{gather}
This and hypothesis (\ref{HypIndex1LocalPart3}) implies that $J_{11}^{-1}F\in [L^1_{loc}(I)]^{n_1}$ and $J_{11}^{-1}H/H_{22}\in M_{n_1}(L^1_{loc}(I))$. By Theorem \ref{ThmODEIVPExistenceUniquenessSolns}, there exists a unique solution $f_1\in [W^{1,1}_{loc}(I)]^{n_1}$ to the ODE IVP
\begin{align}
    &\frac{df_1}{dt}+J_{11}^{-1}H/H_{22}f_1=J_{11}^{-1}F,\;\;f_1(t_0)=J_{11}^{-1}(f_1)_0.\label{ReducedODEIVPOnI}
\end{align}
on $I$. Thus, if we take $f_2$ to be defined by (\ref{InhomogEqForMathCalLEquivSys4}) then $f,g$ satisfy (\ref{InhomogEqForMathCalLEquivSys1})--(\ref{DefFfunction}), where $f,g$ and $f_1,f_2, g_1, g_2$ are related by (\ref{InhomogEqForMathCalLEquivSys5}). It follows from this and part (i) of this proposition, that $f\in D(\mathcal{L}),$ $\mathcal{L}f=g$, and we have
\begin{align}
(Jf)(t_0)=(JVV^{-1}f)(t_0)=V\begin{bmatrix}
J_{11}f_1(t_0)\\ 0
\end{bmatrix}=V\begin{bmatrix}
(f_1)_0\\ 0
\end{bmatrix}=f_0.
\end{align}
The uniqueness portion of statement (ii) in this case follows immediately by the uniqueness of the solution to the ODE IVP (\ref{ReducedODEIVPOnI}) on $I$. This proves statement (ii) in the case $\det J=0$ with the exception that we still need to prove $f\in L^2_{loc}(I;W)$ if (\ref{HypIndex1LocalPart0})--(\ref{HypIndex1LocalPart1Pt5}). We do this next.

Suppose, in addition, that (\ref{HypIndex1LocalPart0})--(\ref{HypIndex1LocalPart1Pt5}) are true. To prove $f\in L^2_{loc}(I;W)$ it suffices by Lemma \ref{LemCharL2WSpace} to prove $f_1,f_2$ satisfy (\ref{fInL2WEquivCond2}) in the ``loc" case, i.e., when in (\ref{fInL2WEquivCond2}) we replace $[L^2(I)]^{n_1}, [L^2(I)]^{n_1}$ with $[L^2_{loc}(I)]^{n_1}, [L^2_{loc}(I)]^{n_1}$, respectively. Let $[a,b]\subseteq I$ be a compact interval. First, by the hypothesis (\ref{HypIndex1LocalPart1Pt5}) we have $W/W_{22}\in M_{n_1}(L^1([a,b]))$, and so by Lemma \ref{LemWSchurComplProperties} and then Lemma \ref{LemWInvWSqrtProperties} applied to $W/W_{22}$ (instead of $W$) it follows that $(W/W_{22})^{1/2}\in M_{n_1}(L^2([a,b]))$. Hence, since $f_1\in [W^{1,1}([a,b])]^{n_1}\subseteq [L^{\infty}([a,b])]^{n_1}$, this implies by Holder's inequality  that $(W/W_{22})^{1/2}f_1\in [L^2([a,b])]^{n_1}$. Also, as $f_1\in [L^{\infty}(I)]^{n_1}$ and by hypothesis (\ref{HypIndex1LocalPart1}), it follows by Holder's inequality  that
\begin{align}
    W_{22}^{1/2}(W_{22}^{-1}W_{21}-H_{22}^{-1}H_{21})f_1\in (L^2([a,b]))^{n_2}.
\end{align}
Next, as $g\in L^2([a,b];W)$, it follows by Lemma \ref{LemCharL2WSpace} that
\begin{align}
    W_{22}^{-1/2}(W_{21}g_1+W_{22}g_2)\in (L^2([a,b]))^{n_2}.
\end{align}
Hence, from this and hypothesis (\ref{HypIndex1LocalPart0}), it follows by Holder's inequality  that
\begin{align}
    W_{22}^{1/2}H_{22}^{-1}W_{22}^{1/2}[W_{22}^{-1/2}(W_{21}g_1+W_{22}g_2)]\in (L^2([a,b]))^{n_2}.
\end{align}
Finally, it follows from these facts, the formula (\ref{InhomogEqForMathCalLEquivSys4}) for $f_2$, and Minkowski's inequality that
\begin{gather}
    W_{22}^{-1/2}(W_{21}f_1+W_{22}f_2)\\
    =W_{22}^{-1/2}(W_{21}f_1+W_{22}(H_{22}^{-1}(W_{21}g_1+W_{22}g_2)-H_{22}^{-1}H_{21}f_1)))\\
    =W_{22}^{-1/2}(W_{21}f_1-W_{22}H_{22}^{-1}H_{21}f_1)+W_{22}^{1/2}H_{22}^{-1}(W_{21}g_1+W_{22}g_2)\\
    =W_{22}^{1/2}(W_{22}^{-1}W_{21}-H_{22}^{-1}H_{21})f_1\\
    +W_{22}^{1/2}H_{22}^{-1}W_{22}^{1/2}[W_{22}^{-1/2}(W_{21}g_1+W_{22}g_2)]\in (L^2([a,b]))^{n_2}.
\end{gather}
This proves that $f\in L^2([a,b];W)$. As $[a,b]$ was an arbitrary compact interval in $I$, this proves $f\in L^2_{loc}(I;W)$ which completes the proof of statement (ii) and hence proves the proposition.
\end{proof}

\begin{corollary}\label{CorKeyRanKerRelForIndex1HypL}
If the index-$1$ hypotheses (\ref{HypIndex1LocalPartNeg1}) or (\ref{HypIndex1LocalPartNeg0Pt5})--(\ref{HypIndex1LocalPart1Pt5}) are true on a compact interval $I$ then
\begin{align}
    \operatorname{ran}L=L^2(I;W),\;\;\ker L=\ker \mathcal{L},\;\;\dim \ker L = \operatorname{rank} J<\infty.
\end{align}
\end{corollary}
\begin{proof}
Let $g\in L^2(I;W)$, where $I$ is a compact interval. Fix any $t_0\in I$ and $f_0\in\operatorname{ran}J$. Then by Proposition \ref{PropKeySolvabilityLocalIndex1DAEs} there exists a unique $f\in D(\mathcal{L})$ such that $\mathcal{L}f=g$ with $(Jf)(t_0)=f_0$. Moreover, from this proposition we also know that $f\in L^2_{loc}(I;W)=L^2(I;W)$. It follows from these facts that $f\in D(L)$ and $Lf=\mathcal{L}f=g$ which implies $f\in \operatorname{ran}L.$ This proves that $L^2(I;W)\subseteq\operatorname{ran}L$ and since $\operatorname{ran}L\subseteq L^2(I;W)$ (by definition of $L$), we conclude that $\operatorname{ran}L= L^2(I;W)$. Next, as we can take $g=0$ in this proof, it follows immediately that $\ker L=\ker \mathcal{L}$. It remains to prove $\dim \ker L = \operatorname{rank} J<\infty$. First, as $J\in \mathbb{C}^n, J\not =0$, then $1\leq r:=\operatorname{rank} J=\dim \operatorname{ran} J\leq n<\infty$. Next, let $\beta_1,\ldots, \beta_r$ be a basis for $\operatorname{ran} J$. Then in our proof with $g=0$ and $t_0\in I$ fixed, there exists a unique solution $f_j\in D(L)$ to $Lf_j=0$ with $(Jf_j)(t_0)=\beta_j$, for each $j=1,\ldots, r$. We claim that $f_1,\ldots, f_r$ is a basis for $\ker L$. Obviously, $f_1,\ldots, f_r\in \ker L$. Next, let $c_1,\ldots, c_r\in \mathbb{C}$ be such that
\begin{align}
    c_1f_1+\cdots + c_rf_r=0.
\end{align}
Then
\begin{align}
    0=(J0)(t_0)=[J(c_1f_1+\cdots+c_rf_r)](t_0) = c_1\beta_1\cdots+c_r\beta_r
\end{align}
implying $c_1=\cdots = c_r$. This proves the vectors $f_1,\ldots, f_r$ are linearly independent. Finally, let $h\in \ker L$. Then there exists scalars $c_1,\ldots, c_r\in \mathbb{C}$ such that
\begin{align}
    (Jh)(t_0)=c_1\beta_1\cdots+c_r\beta_r=[J(c_1f_1+\cdots+c_rf_r)](t_0).
\end{align}
By the uniqueness of the solution $f\in D(L)$ to $Lf=0, (Jf)(t_0)=(Jh)(t_0)$, it follows that $h=c_1f_1+\cdots+c_rf_r$. This proves that the vectors $f_1,\ldots, f_r$ span $\ker L$. Therefore, the vectors $f_1,\ldots, f_r$ are a basis for $\ker L$ which proves our claim and also proves that  $\dim \ker L =r= \operatorname{rank} J<\infty$.
\end{proof}

\begin{definition}\label{DefL0InFiniteInterval}
In the case that $I$ is a bounded interval with endpoints $a, b$ and $a<b$, the \textit{closed minimal operator} $L_0:D(L_0)\rightarrow L^2(I;W)$ \textit{generated by $\mathcal{L}$} is defined by
\begin{align}
    D(L_0)&=\{f\in D(L):\lim_{t_0\rightarrow a^+}(Jf)(t_0)=\lim_{t_1\rightarrow b^-}(Jf)(t_1)=0\},\label{DefL0InFiniteInterval1}\\
    L_0f&=\mathcal{L}f,\text{ for }f\in D(L_0).\label{DefL0InFiniteInterval2}
\end{align}
\end{definition}

\begin{lemma}
Suppose $I$ is a bounded interval. Then $D(L_0)$ is a subspace of the Hilbert space $L^2(I;W)$. Moreover, $L_0:D(L_0)\rightarrow L^2(I;W)$ is a linear operator with
\begin{gather}
    D(L_0')\subseteq D(L_0)\subseteq D(L)\subseteq D(\mathcal{L})\cap L^2(I;W),\\
    L_0f=Lf=\mathcal{L}f,\;\;f\in D(L_0),\\
    L_0'u=L_0u,\;\;u\in D(L_0').
\end{gather}
\end{lemma}
\begin{proof}
The proof is straightforward and so will be omitted.
\end{proof}

\begin{notation}
When we need to be explicit about the dependence of $L_0$ on $H$ and/or $J$ we will use the subscript $(\cdot)_H$ or $(\cdot)_{J,H}$ with these operators, e.g., $(L_0)_{H}$ and for the latter, we will just write $L_{H,0}$ instead. Similarly, $(L_0)_{J,H}$ will be written instead as $L_{J,H,0}$.
\end{notation}

The importance of the next two theorems becomes more clear by comparing it to Theorem \ref{thm:prem:FundResultAdjointsKerRanDecomp} with $A=L_0, B=L_{H^*}$ as we use these results to prove Theorem \ref{ThmRegularIndex1MinMaxOpsL0LMainThm} below.
\begin{theorem}\label{ThmAdjointCalcForClosedMinAndMaxOpsRelatedToSymmetricOps}
Let $I$ is a bounded interval. Then the following are true:\\
\noindent a) For $f\in D(L_0)$ and $g\in D(L_{H^*})$ we have
  \begin{align}
      \langle L_0f,g\rangle_W = \langle f,L_{H^*}g\rangle_W.
  \end{align}
  b) For $f\in D(L_0)$ and $g\in D(L_{H^*,0})$ we have
  \begin{align}
       \langle L_0f,g\rangle_W = \langle f,L_{H^*,0}g\rangle_W.
  \end{align}
\end{theorem}
\begin{proof}
The proof is similar to the proof of Theorem \ref{ThmAdjointCalcForMinMaxOpsRelatedToSymmetricOps} and so it is omitted.
\end{proof}

\begin{theorem}\label{ThmKerRanRelBetweenClosedMinAndMaxOps}
Suppose $I$ is a bounded interval and the index-$1$ hypotheses (see Def. \ref{DefIndex1Hyp}) are true for $H, W$ with respect to $J$ on the interval $I$. Then
\begin{align}
    \operatorname{ran}L_0&=(\ker L_{H^*})^{\perp},\;\;(\operatorname{ran}L_0)^{\perp}=\ker L_{H^*},\label{RanL0PerpKerLH*Relationship}\\
    \ker L_0 &= \{0\},\;\;\operatorname{ran}L_{H^*}=(\ker L_0)^{\perp} = L^2(I;W),\label{KerL0Is0}\\
    \ker L_{H^*}&+\operatorname{ran}L_0=\operatorname{ran}L_{H^*}+\ker L_0 = L^2(I;W).\label{RanAndKerSumsOfL0AndH*IsAllL2}
\end{align}
In particular, $\operatorname{ran}L_0$ and $\ker L_{H^*}$ are closed subspaces of $L^2(I;W)$, which are orthogonal to each other.
\end{theorem}
\begin{proof}
Assume the hypotheses. Then by Corollary \ref{cor:SimplifiedIndex1HypImplyItForAdjoints}, the index-$1$ hypotheses are also true for $H^*, W$ with respect to $J$ on the interval $I$. First, we may just assume that $I$ is a compact interval, for if it is not then we can just take its closure instead which is now compact and the result for $I$ will then follow from the statement for that compact interval. Second, (\ref{RanAndKerSumsOfL0AndH*IsAllL2}) follows immediately from (\ref{RanL0PerpKerLH*Relationship}) and (\ref{KerL0Is0}). Third, it follows immediately from Corollary \ref{CorKeyRanKerRelForIndex1HypL} (applied to $L_{H^*}$) that $\operatorname{ran}L_{H^*}=L^2(I;W)$ and by Corollary \ref{CorKeyRanKerRelForIndex1HypL} (applied to $L$) and Proposition \ref{PropKeySolvabilityLocalIndex1DAEs}.(ii) that $\ker L_0 = \{0\}$, which proves the equalities in (\ref{KerL0Is0}). Next, the second identity of (\ref{RanL0PerpKerLH*Relationship}) follows from the first identity since $\ker L_{H^*}$ is finite-dimensional by Corollary \ref{CorKeyRanKerRelForIndex1HypL} (applied to $L_{H^*}$) and hence $(\ker L_{H^*})^{{\perp}{\perp}}=\overline {\ker L_{H^*}}=\ker L_{H^*}$. Thus, it remains to prove that $\operatorname{ran}L_0=(\ker L_{H^*})^{\perp}$ when $I$ is a compact interval, which we do so now.

Let $f\in L^2(I;W)$, where $I=[a,b]$ is a compact interval. Then, by Proposition \ref{PropKeySolvabilityLocalIndex1DAEs}.(ii) and the hypotheses there is a unique solution $u\in D(\mathcal{L})$ to the IVP $\mathcal{L}u=f,(Ju)(a)=0$, and, in addition, $u\in L^2_{loc}(I;W)=L^2(I;W)$ so that $u\in D(L)$ with $Lu=\mathcal{L}u=f$.  By applying Corollary \ref{CorKeyRanKerRelForIndex1HypL} to $L_{H^*}$ we know that $\ker \mathcal{L}_{H^*}=\ker L_{H^*}$ and $\infty>\dim \ker L_{H^*}=\operatorname{rank}J=:r>0$. Furthermore, using a similar proof as in Corollary \ref{CorKeyRanKerRelForIndex1HypL}, we get a basis $\{z_i:1\leq i \leq r\}$ of $\ker L_{H^*}$ satisfying $(Jz_i)(b)=JVe_i, i=1,\ldots, r$, where $e_1,\ldots, e_n$ denote the standard basis vectors for $\mathbb{C}^{n}$. Then it follows that
\begin{gather}
    \langle f,z_i \rangle_W = \langle Lu,z_i \rangle_W=\langle \mathcal{L}u,z_i \rangle_W-\langle u,\mathcal{L}_{H^*}z_i \rangle_W\\
    = \langle (Ju)(b),J^+(Jz_i)(b) \rangle-\langle (Ju)(a),J^+(Jz_i)(a) \rangle  = \langle (Ju)(b),J^+(Jz_i)(b)\rangle \\
    = \langle (Ju)(b),J^+JVe_i\rangle= \langle JJ^+(Ju)(b),Ve_i\rangle = \langle (JJ^+Ju)(b),Ve_i\rangle\\
    = \langle (Ju)(b),Ve_i\rangle= \langle V^*(Ju)(b),e_i\rangle = \langle V^{-1}(Ju)(b),e_i\rangle\\
    = \langle (V^{-1}JVV^{-1}u)(b),e_i\rangle= (V^{-1}JVV^{-1}u)(b)_i = (J_{11}u_1)(b)_i,
\end{gather}
for each $i=1,\ldots, r$, where $V=I_n, J_{11}=J$ if $\det J\not =0$, otherwise,
\begin{align}
    V^{-1}u=\begin{bmatrix}
    u_1\\
    u_2
    \end{bmatrix},\;\;V^{-1}JVV^{-1}u=\begin{bmatrix}
    J_{11}u_1\\
    0
    \end{bmatrix}.
\end{align}
Suppose that $f\in \operatorname{ran} L_0$. Then, for the above we must have $u\in D(L_0),$ $L_0u=Lu=f,$ and $(Ju)(a)=(Ju)(b)=0$ so that by the above calculation it follows that $\langle f,z_i \rangle =0$ for all $i=1,\ldots, r$. As $z_1,\ldots, z_{r}$ is a basis of solutions to $\ker L_{H^*}$ then this implies that $\langle f,z \rangle_W =0$ for all $z\in \ker L_{H^*}$ and hence $f\in (\ker L_{H^*})^{\perp}$. This proves that $\operatorname{ran} L_0\subseteq (\ker L_{H^*})^{\perp}$. Conversely, suppose $f\in (\ker L_{H^*})^{\perp}$. Then $\langle f,z_i \rangle_W =0$ for every $z\in \ker L_{H^*}$. By the above calculation for this $f$ and the corresponding $u$, we have $0=\langle f,z_i \rangle=(J_{11}u_1)(b)_i$ for each $i=1,\ldots, r$ implying $(J_{11}u_1)(b)=0$. As this implies $(Ju)(b)=0$ then it follows that $u\in D(L_0)$ [as $u\in D(L)$ with $(Ju)(a)=(Ju)(b)=0$] and so $L_0u=Lu=f$ which proves $f\in \operatorname{ran} L_0$. This proves $(\ker L_{H^*})^{\perp}\subseteq\operatorname{ran} L_0$. Therefore, $\operatorname{ran} L_0=(\ker L_{H^*})^{\perp}$, which completes the proof.
\end{proof}

\begin{theorem}\label{ThmRegularIndex1MinMaxOpsL0LMainThm}
Suppose $I$ is a bounded interval and the index-$1$ hypotheses (see Def. \ref{DefIndex1Hyp}) are true for $H, W$ with respect to $J$ on the interval $I$. Then the following statements hold:\\
  \noindent a) The subspaces $D(L_0)$, $D(L_{H^*,0})$, $D(L)$, and $D(L_{H^*})$ are all dense in $L^2(I;W)$.\\
  \noindent b) The operators $L_0:D(L_0)\rightarrow L^2(I;W)$ and $L_{H^*}:D(L_{H^*})\rightarrow L^2(I;W)$ are densely defined closed operators with closed ranges and are adjoints of each other. In particular,
  \begin{align}
      (L_0)^*=L_{H^*}=\overline{L_{H^*}}=(L_{H^*})^{**},\;\;L_{H^*}^*=L_0=\overline{L_0}=(L_0)^{**}.
  \end{align}\\
  \noindent c) The operators $L_{H^*,0}:D(L_{H^*,0})\rightarrow L^2(I;W)$ and $L:D(L)\rightarrow L^2(I;W)$ are densely defined closed operators with closed ranges and are adjoints of each other. In particular,
  \begin{align}
      (L_{H^*,0})^*=L=\overline{L}=L^{**},\;\;L^*=L_{H^*,0}=\overline{L_{H^*,0}}=(L_{H^*,0})^{**}.
  \end{align}
\end{theorem}
\begin{proof}
Assume the hypotheses. Then by Corollary \ref{cor:SimplifiedIndex1HypImplyItForAdjoints}, the index-$1$ hypotheses are also true for $H^*, W$ with respect to $J$ on the interval $I$. First, by duality (i.e., $H^{**}=H$) we need only prove the statements for $L_0$ and $L_{H^*}$ as then the results for $L_{H^*,0}$ and $L_{H}$ will follow immediately by duality. Second, once we have proven that $L_0$ and $L_{H^*}$ are densely defined closed operators then it follows immediately from general results on adjoints \cite[Theorem VIII.1.(b)]{80RS} that the closure of $L_0$ and $L_{H^*}$ (i.e., $\overline{L_0}$ and $\overline{L_{H^*}}$) satisfies the relationships $L_{H^*}=\overline{L_{H^*}}=(L_{H^*})^{**}$ and $L_0=\overline{L_0}=(L_0)^{**}$. Third, once we have proven that $L_0$ is a densely defined it will follow immediately that $L$ is densely defined [since $D(L_0)\subseteq D(L)$] and then by duality (i.e., $H^{**}=H$) it follows that $L_{H^*,0}, L_{H^*}$ are also densely defined. Fourth, if we can then prove that $(L_0)^*=L_{H^*}$, it will follows immediately from general results on adjoints \cite[Theorem VIII.1]{80RS} that $L_{H^*}$ is a closed operator (so that $L_{H^*}=\overline{L_{H^*}})$, $L_0$ is a closable operator, and $L_{H^*}^*=\overline{L_0}$. Finally, if we can then prove $L_0$ is a closed operator, i.e., $\overline{L_0}=L_0$, we will have completed the proof of the theorem.

    Thus, to prove the theorem we will now prove the following statements: (i) $L_0$ is densely defined; (ii) $(L_0)^*=L_{H^*}$; (iii) $L_0$ is a closed operator. Again, as in the proof of Theorem \ref{ThmKerRanRelBetweenClosedMinAndMaxOps}, we can just assume in the proof that $I$ is a compact interval as the results in the case of bounded intervals will follow from this.

    (i): Let $h\in D(L_0)^{\perp}$. Then for any solution $g\in D(\mathcal{L_{H^*}})$ of $\mathcal{L_{H^*}}g=h$ we have by Proposition \ref{PropKeySolvabilityLocalIndex1DAEs}.(ii) (applied to $H^*$ instead of $H$) that $g\in L^2_{loc}(I;W)$ and since $I$ is compact we have $L^2_{loc}(I;W)=L^2(I;W)$. It now follows that $g\in D(L_{H^*})$ and $L_{H^*}g=h$. Hence, from this and by Theorem \ref{ThmAdjointCalcForClosedMinAndMaxOpsRelatedToSymmetricOps}.a), it follows that for every $f\in D(L_0)$, we have
    \begin{align}
         \langle L_0f,g \rangle_W=\langle f,L_{H^*}g \rangle_W=\langle f,h \rangle_W=0.
    \end{align}
    This implies that $g\in \operatorname{ran}(L_0)^{\perp}$. By Theorem \ref{ThmKerRanRelBetweenClosedMinAndMaxOps} we know that $\ker L_{H_*}=(\operatorname{ran} L_0)^{\perp}$ so that $h=L_{H^*}g=0$. From which we conclude that $D(L_0)^{\perp}=\{0\}$ and thus $\overline{D(L_0)}=D(L_0)^{{\perp}{\perp}}=\{0\}^{\perp}=L^2(I;W)$. Therefore, $L_0$ is densely defined, which proves (i).

    (ii): From part (i) we know that the linear operator $L_0:D(L_0)\rightarrow L^2(I;W)$ is densely defined and so now we consider it's adjoint $(L_0)^*$. By Theorem \ref{ThmAdjointCalcForClosedMinAndMaxOpsRelatedToSymmetricOps}.a), it follows that $D(L_{H^*})\subseteq D((L_0)^*)$ and $(L_0)^*g=L_{H^*}g$ for all $g\in D(L_{H^*})$. Thus, to prove that $(L_0)^*=L_{H^*}$, it suffices to prove that $D((L_0)^*)\subseteq D(L_{H^*})$. Let $g\in D((L_0)^*)$ and set $h=(L_0)^*g$. Then, since we know by Theorem \ref{ThmKerRanRelBetweenClosedMinAndMaxOps} that $\operatorname{ran}L_{H^*}=L^2(I;W)$, there exists an $f\in D(L_{H^*})$ such that $L_{H^*}f=h$. Hence, from this and Theorem \ref{ThmAdjointCalcForClosedMinAndMaxOpsRelatedToSymmetricOps}.a), it follows that for every $u\in D(L_0)$, we have
    \begin{align}
        \langle L_0u,g \rangle_W=\langle u,(L_0)^*g \rangle_W=\langle u,L_{H^*}f \rangle_W=\langle L_0u,f \rangle_W
    \end{align}
    which implies $g-f\in (\operatorname{ran} L_0)^{\perp}$. By Theorem \ref{ThmKerRanRelBetweenClosedMinAndMaxOps} we know that $(\operatorname{ran} L_0)^{\perp}=\ker L_{H_*}\subseteq D(L_{H_*})$ and hence $g-f\in D(L_{H_*})$. As $f\in D(L_{H^*})$ and $D(L_{H^*})$ is a subspace of $L^2(I;W)$, it follows that $g\in D(L_{H^*})$. Therefore, $D((L_0)^*)\subseteq D(L_{H^*})$, which proves (ii).

    (iii) From part (ii) we know that $D(L_0)$ is dense in $L^2(I;W)$ and since $D(L_0)\subseteq D(L)$ then $D(L)$ is dense in $L^2(I;W)$. Hence, by the hypotheses using $H^*$ instead of $H$, it follows that $D(L_{H^*})$ is also dense in $L^2(I;W)$. Thus, as we proved $(L_0)^*=L_{H^*}$, it follows immediately from general results on adjoints \cite[Theorem VIII.1]{80RS} that $L_{H^*}$ is a closed operator (so that $L_{H^*}=\overline{L_{H^*}})$, $L_0$ is a closable operator, and $L_{H^*}^*=\overline{L_0}$. We will now prove $\overline{L_0}=L_0,$ i.e., $L_0$ is closed. Let $\{u_m\}_{m\in \mathbb{N}}\subseteq D(L_0)$ be a sequence converging in $L^2(I;W)$ converging to $u$ such that $\{L_0u_m\}_{m\in \mathbb{N}}$ converges in $L^2(I;W)$ to $f$. If we can prove that $u\in D(L_0)$ and $L_0u=f$ then this will prove that $L_0$ is a closed operator. First, since $L_0$ is closable and $L_{H^*}^*=\overline{L_0}$, it follows by general results on closable operators \cite[Proposition, p.\ 250]{80RS} that $u\in D(L_{H^*}^*)$ and $L_{H^*}^*u=f$. Next, let $g\in \ker L_{H^*}$. Then by Theorem \ref{ThmAdjointCalcForClosedMinAndMaxOpsRelatedToSymmetricOps}.a) it follows that
    \begin{align}
        \langle f,g \rangle_W=\lim_{m\rightarrow \infty}\langle L_0u_m,g \rangle_W=\lim_{m\rightarrow \infty}\langle u_m,L_{H^*}g \rangle_W=0
    \end{align}
    which implies $f\in (\ker L_{H^*})^{\perp}=\operatorname{ran}L_0$, where the latter equality follows from Theorem \ref{ThmKerRanRelBetweenClosedMinAndMaxOps}. It follows that there exists $v\in D(L_0)$ such that $L_0v=f=L_{H^*}^*u$. Hence, by Theorem \ref{ThmAdjointCalcForClosedMinAndMaxOpsRelatedToSymmetricOps}.a) it follows that for any $h\in D(L_{H^*})$, we have
    \begin{align}
        \langle v,L_{H^*}h \rangle_W=\langle L_0v,h \rangle_W=\langle L_{H^*}^*u,h \rangle_W=\langle u,L_{H^*}h \rangle_W
    \end{align}
    implying $u-v\in (\operatorname{ran}L_{H^*})^{\perp}=\ker L_0=\{0\}$, where the latter two equalities follow from Theorem \ref{ThmKerRanRelBetweenClosedMinAndMaxOps}, and hence $u=v\in D(L_0)$ with $L_0u=L_0v=f$. Thus, we've proved $L_0$ is a closed operator which proves (iii).
    Therefore, we have proven statements (i), (ii), and (iii) which proves the theorem now from our discussion above.
\end{proof}

We now extend Theorem \ref{ThmRegularIndex1MinMaxOpsL0LMainThm} to allow for arbitrary intervals $I\subseteq \mathbb{R}$, which include unbounded intervals, e.g., $I=\mathbb{R}$, but under the local index-$1$ hypotheses now. 

\begin{remark}
The proof of the next theorem is essentially based on the notion of
inductive limits in the category of Hilbert spaces and unbounded operators. And, although we do not elaborate more on this as it would require too much additional background, we highly recommend the references \cite{74AM, 75AM, 88JJ}, \cite[Exercises 11.5.26 and 11.5.27]{86KR}, \cite[Appendix A]{18LT}, and \cite{13LS} as a starting point for the interested reader.
\end{remark}
\begin{theorem}\label{ThmDenselyDefDomainMinMaxOpsLocallyIndex1Hyp}
  Suppose $I\subseteq \mathbb{R}$ is an interval (with nonempty interior) and the local index-$1$ hypotheses (see Def. \ref{DefIndex1Hyp}) are true for $H, W$ with respect to $J$ on the interval $I$. Then the following statements hold:\\
  \noindent i) The subspaces $D(L_0'), D(L_{H^*,0}'), D(L),$ and $D(L_{H^*})$ are all dense in $L^2(I;W)$.\\
  \noindent ii) The operators $L_0':D(L_0')\rightarrow L^2(I;W)$ and $L_{H^*}:D(L_{H^*})\rightarrow L^2(I;W)$ are densely defined. Furthermore, $L_{H^*}$ is a closed operator and $L_0$ is a closable operator such that its closure, i.e., $\overline{L_0'}=(L_0')^{**}$, and $L_{H^*}$ are adjoints of each other. In particular,
  \begin{align}
      (L_0')^*=(\overline{L_0'})^*=L_{H^*}=\overline{L_{H^*}}=(L_{H^*})^{**},\;\;(L_{H^*})^*=\overline{L_0'}.
  \end{align}\\
  \noindent iii) The operators $L_{H^*,0}':D(L_{H^*,0}')\rightarrow L^2(I;W)$ and $L:D(L)\rightarrow L^2(I;W)$ are densely defined. Furthermore, $L$ is a closed operator and $L_{H^*,0}'$ is a closable operator such that its closure, i.e., $\overline{L_{H^*,0}'}=(L_{H^*,0}')^{**}$, and $L$ are adjoints of each other. In particular,
  \begin{align}
      (L_{H^*,0}')^*=(\overline{L_{H^*,0}'})^*=L=\overline{L}=L^{**},\;\;L^*=\overline{L_{H^*,0}'}.
  \end{align}
\end{theorem}
\begin{proof}
First, as the hypotheses are valid for both $H$ and $H^*$ (by Corollary \ref{cor:SimplifiedIndex1HypImplyItForAdjoints}), and $(H^*)^*=H$, then we need only prove the statement for $L_0'$ and $L_{H^*}$ as the other results will follow immediately by this duality. Second, once we have proven that $L_0'$ and $L_{H^*}$ are densely defined operators with $L_0'$ and $L_{H^*}$ a closable and closed operator, respectively, then it follows immediately from general results  that the closure of $L_0'$ and $L_{H^*}$ satisfy the relationships $L_{H^*}=\overline{L_{H^*}}=(L_{H^*})^{**}$ and $\overline{L_0'}=(L_0')^{**}$ and hence once we prove $(L_0')^*=L_{H^*}$, then it will follow immediately that $(L_{H^*})^*=\overline{L_0'}$. On the other hand, once we know that $D(L_0')$ is dense for any arbitrary $H$ satisfying the local index-1 hypotheses then by duality $D(L_{H^*,0}')$ is also dense so that since $D(L_{H^*,0}')\subseteq D(L_{H^*})$, it will follow that $D(L_{H^*})$ is dense which implies then by Theorem \ref{ThmAdjointCalcForMinMaxOpsRelatedToSymmetricOps}.a) that the adjoint $(L_0')^*$ of $L_0'$ is an extension of $L_{H^*}$, i.e., $D(L_{H^*})\subseteq D((L_0')^*)$ with $L_{H^*}f=(L_0')^*f$ for every $f\in D(L_{H^*})$ [which we denote by $L_{H^*}\subset (L_0')^*$], so that $(L_0')^*$ is densely defined from which it follows from general results on adjoints  that $L_0'$ is closable, and finally since $L_{H^*}\subset (L_0')^*$ then to prove $(L_0')^*=L_{H^*}$ (which implies from this by general results  that $L_{H^*}$ is closed), we need only prove $(L_0')^*\subset L_{H^*}$. Thus, to recap, to complete the proof we need only prove that $D(L_0')$ is dense for any arbitrary $H$ satisfying the local index-1 hypotheses and that $(L_0')^*\subset L_{H^*}$, i.e., $D((L_0')^*)\subseteq D(L_{H^*})$ and $L_{H^*}f=(L_0')^*f$ for every $f\in D((L_0')^*)$. We will do this next.

For any arbitrary $\alpha,\beta\in \mathbb{R}$ with $\alpha<\beta$ such that $\triangle=[\alpha,\beta]\subseteq \operatorname{int} I$, we consider $\mathcal{L}$ on the compact interval $\triangle$. Denote by $L_{\triangle,0}, L_{\triangle}$ the closed minimal and maximal operators generated by $\mathcal{L}$ in $L^2(\triangle;W)$. With abuse of notation, we can treat any element of $D(L_{\triangle,0})$ as an element of $D(L_0')$ by zero extension, i.e., setting it equal to zero on $I\setminus \triangle$. It then follows that 
\begin{align}
    \cup_{\triangle \subseteq I}D(L_{\triangle,0})=D(L_0').
\end{align}
As each $D(L_{\triangle,0})$ is dense in $L^2(\triangle;W)$  it then follows that $D(L_0')$ is dense in $L^2(I;W)$. Hence as $D(L_0')\subseteq D(L)$ with $L_0f=Lf$ for all $f\in D(L_0')$, it follows that both the linear operators $L_0':D(L_0')\rightarrow L^2(I;W)$ and $L:D(L)\rightarrow L^2(I;W)$ are densely defined with $L_0'\subset L$. It now follows by duality, using the same argument but with $H^*$ instead of $H$, that both the linear operators $L_{H^*,0}':D(L_{H^*,0}')\rightarrow L^2(I;W)$ and $L:D(L_{H^*})\rightarrow L^2(I;W)$ are densely defined. Hence, as mentioned above, to complete the proof we need only prove that $(L_0')^*\subset L_{H^*}$. We now write $\langle \cdot, \cdot \rangle_{\triangle}$ for the inner product in $L^2(\triangle;W)$; by $h_{\triangle}$ we denote the restriction of a function $h$ from $L^2(I;W)$ to $L^2(\triangle;W)$ (in the natural way). We already know that $L_{\triangle,0}^*=L_{H^*,\triangle}$ and $L_{\triangle,0}\subset L_0'$ (in the abuse of notation sense above), hence $L_{\triangle,0}\subset L_0'\subset L$. We also know that for every $h\in D(L)$ we have $h_{\triangle}\in D(L_{\triangle})$. Let $f\in D((L_0')^*)$. Then for any $g\in L^2(I;W)$ with $g\in D(L_{\triangle,0})$ [so that $g\in D(L_0')$ in the sense above] it follows that 
\begin{align}
    \langle ((L_0')^*f)_{\triangle},g_{\triangle}\rangle_{\triangle}=\langle (L_0')^*f,g\rangle=\langle f,L_0'g\rangle=\langle f,L_{\triangle,0}g\rangle=\langle f_{\triangle},L_{\triangle,0}g_{\triangle}\rangle_{\triangle}.
\end{align}
This implies that $f_{\triangle}\in D(L_{\triangle,0}^*)$  and, as $L_{\triangle,0}^*=L_{H^*,\triangle}$, this also implies 
\begin{align}
    ((L_0')^*f)_{\triangle} = L_{\triangle,0}^*f_{\triangle} = L_{H^*,\triangle}f_{\triangle} = (L_{H^*}f)_{\triangle}.
\end{align}
As this is true for any compact interval $\triangle$ with nonempty interior satisfying $\triangle \subseteq \operatorname{int} I$ then this implies  that $f\in D(\mathcal{L}_{H^*})$  and 
\begin{align}
    \mathcal{L}_{H^*}f=(L_0')^*f\in L^2(I;W),
\end{align}
which implies $f\in D(L_{H^*})$ and $L_{H^*}f=\mathcal{L}_{H^*}f=(L_0')^*f$. This proves that $(L_0')^*\subset L_{H^*}$, which completes the proof of the theorem.
\end{proof}

If the hypotheses of Theorem \ref{ThmDenselyDefDomainMinMaxOpsLocallyIndex1Hyp} are true then this theorem tells us $L_0'$ is a densely defined closable operator. As such, we know that its closure $\overline{L_0'}=(L_0')^{**}$ is the smallest closed operator extension of $L_0'$ (see \cite[Sec.\ VIII.2]{80RS}). Because of this we can extend the Definition \ref{DefL0InFiniteInterval} to unbounded intervals as follows.
\begin{definition}\label{DefL0InInfiniteInterval}
In the case $I$ is an unbounded interval and the local index-$1$ hypotheses (see Def. \ref{DefIndex1Hyp}) are true for $H, W$ with respect to $J$ on the interval $I$, the \textit{closed minimal operator $L_0$ generated by  $\mathcal{L}$} is defined by
\begin{align}
    L_0:=\overline{L_0'}=(L_0')^{**}.\label{DefL0InLocalIndex1Case}
\end{align}
\end{definition}

The next theorem shows in what sense this definition is an extension of the definition of $L_0$ in Definition \ref{DefL0InFiniteInterval} from bounded intervals to unbounded intervals.
\begin{theorem}\label{ThmAnalogyThm3Pt8InWeidmannBook}
If $I$ is a bounded interval and the index-$1$ hypotheses (see Def. \ref{DefIndex1Hyp}) are true for $H, W$ with respect to $J$ on the interval $I$, then \begin{align}
      L_0=\overline{L_0'}=(L_0')^{**}.
  \end{align}
\end{theorem}
\begin{proof}
Assume the hypotheses. Then the local index-$1$ hypotheses (see Def. \ref{DefIndex1Hyp}) are also true for $H^*, W$ with respect to $J$ on the interval $I$ (by Corollary \ref{cor:SimplifiedIndex1HypImplyItForAdjoints}). Thus, by Theorem \ref{ThmDenselyDefDomainMinMaxOpsLocallyIndex1Hyp} we know that $L_0':D(L_0')\rightarrow  L^2(I;W)$ is a densely defined closable operator and hence it follows that $\overline{L_0'}=(L_0')^{**}$ (see \cite[Sec.\ VIII.2]{80RS}). Now, by Theorem \ref{ThmRegularIndex1MinMaxOpsL0LMainThm} we have $L_{H^*}^*=L_0$ and by Theorem \ref{ThmDenselyDefDomainMinMaxOpsLocallyIndex1Hyp} we know that $(L_0')^*= L_{H^*}$. Therefore,  $L_0=L_{H^*}^*=(L_0')^{**}=\overline{L_0'}$.
\end{proof}

One of the main problems in the spectral theory of DAEs is to answer the following question: If $H^*=H$, what additional hypotheses imply the maximal operator $L:D(L)\rightarrow L^2(I;W)$ generated by $\mathcal{L}$ is a self-adjoint operator, i.e., $L^*=L$? The goal of the next section is to consider this question under the additional assumption that $H$ is periodic.

\section{\label{sec:SpecThyPeriodicDAEs}Spectral theory for linear differential-algebraic equations with periodic coefficients}

In this section we will study the linear differential-algebraic equations with periodic coefficients [say, $d$-periodic for some fixed period $d\in (0,\infty)$], i.e., $d$-periodic linear DAEs (\ref{def:LinearDAEsInhomo}), in terms of the spectral problem (\ref{def:intro:DAEsStandardForm}) and the corresponding spectral theory of their associated DA operators $\mathcal{L}, L_0', L_0, L$ on the Hilbert space $L^2(\mathbb{R}; W)$. Thus, we continue to assume that (\ref{HypForJ})-(\ref{HypForWOnab}) are true for $J,H,W$ with the interval $I=(-\infty, \infty)=\mathbb{R}$, but now we have the addition periodicity hypothesis:
\begin{gather}
    H(t+d)=H(t),\;W(t+d)=W(t),\;\text{for all }t\in \mathbb{R}.\label{HypPeriodicityH}
\end{gather}
\begin{remark}
It should be remarked that our conclusions below are unaffected by weakening the latter hypothesis from ``for all $t\in \mathbb{R}$" to ``for a.e.\ $t\in\mathbb{R}$."
\end{remark}

In order to proceed in the more general fashion for the self-adjoint spectral theory, we eventually need to make use of the following additional hypotheses on $H$ and on the minimal and maximal operators $L_0':D(L_0')\rightarrow L^2(\mathbb{R};W)$ and $L:D(L)\rightarrow L^2(\mathbb{R};W)$, respectively, generated by the DA operator $\mathcal{L}:D(\mathcal{L})\rightarrow [\mathcal{M}(\mathbb{R})]^n$ [see (\ref{DefDomainOfGenDAEOperator})-(\ref{DefGenDAEOperator}) and Def.\ \ref{DefMinMaxOpGenerByTheDAEsOp}]:
\begin{gather}
    H(t)^*=H(t),\;\text{for a.e.\ }t\in(0,d).\label{NewHyp0PeriodicSetting}\\
    \text{$D(L_0')$ is dense in $L^2(\mathbb{R};W)$.}\label{NewHyp1PeriodicSetting}\\
    (L_0')^*=L.\label{NewHyp2PeriodicSetting}
\end{gather}
\begin{remark}\label{NewRemImportantOf3HypInThePeriodicSetting}
There are some remarks in order regarding the hypotheses (\ref{NewHyp1PeriodicSetting}), (\ref{NewHyp2PeriodicSetting}), and how they will be used to prove $L$ is self-adjoint under the hypothesis (\ref{NewHyp0PeriodicSetting}). First, by Lemma \ref{LemDomainsAreSubspacesForL0primeAndLWhichAreWellDefLinearOps} we know that $D(L_0'), D(L)$ are subspaces of $L^2(\mathbb{R};W)$ with $D(L_0')\subseteq D(L)$ and $L_0':D(L_0')\rightarrow L^2(\mathbb{R};W)$ and $L:D(L)\rightarrow L^2(\mathbb{R};W)$ defined by $Lg=\mathcal{L}g$ for $g\in D(L)$ and $L_0'f=\mathcal{L}f=Lf$ for $f\in D(L_0')$ are well-defined linear operators. Second, if (\ref{NewHyp1PeriodicSetting}) is true, i.e., $L_0'$ is densely defined, then $L$ is also densely defined since $L_0'\subset L$, and also by Theorem \ref{ThmAdjointCalcForMinMaxOpsRelatedToSymmetricOps} we know $L_0'$ is symmetric, hence closable, with $L\subset (L_0')^*=(\overline{L_0'})^*$ which implies from this and the fact $(L_0')^*$ is closed that $L$ must also be closable with $L\subset \overline{L}\subset (L_0')^*=(\overline{L_0'})^*$. Third, it follows from this that if (\ref{NewHyp1PeriodicSetting}) is true then (\ref{NewHyp2PeriodicSetting}) is equivalent to $D((L_0')^*)\subseteq D(L)$, in which case $(\overline{L_0'})^*=(L_0')^*=L$. Fourth, if (\ref{NewHyp1PeriodicSetting}) is true then by the corollary to \cite[Theorem VIII.3]{80RS} and by \cite[Theorem X.1]{75RS}, the following two conditions are equivalent: a) $L_0'$ is esssentially self-adjoint [i.e., it's closure $\overline{L_0'}=(L_0')^{**}$ is self-adjoint]; b) $\ker((L_0')^*- zI)=\ker((L_0')^*- \overline{z}I)=\{0\}$ for some $z\in \mathbb{C}\setminus \mathbb{R}$. Hence, if (\ref{NewHyp1PeriodicSetting}) and (\ref{NewHyp2PeriodicSetting}) are true then the following two conditions are equivalent: a) $L$ is self-adjoint; b) $\ker(L- zI)=\ker(L- \overline{z}I)=\{0\}$ for some $z\in \mathbb{C}\setminus \mathbb{R}$. Therefore, the goal of this section is to find additional ``natural" hypotheses besides (\ref{NewHyp0PeriodicSetting}), (\ref{NewHyp1PeriodicSetting}), and (\ref{NewHyp2PeriodicSetting}) that allow us to conclude that $\ker(L- zI)=\ker(L- \overline{z}I)=\{0\}$ for some $z\in \mathbb{C}\setminus \mathbb{R}$ as then from this remark we will have proven that $L$ is self-adjoint.
\end{remark}

\begin{definition}
For any $m,n\in \mathbb{N}$, we will need in our analysis below, the left ($+$) and right ($-$) \textit{translation} (or \textit{shift}) \textit{operators} $U_{\pm}:M_{m,n}(\mathcal{M}(\mathbb{R}))\rightarrow M_{m,n}(\mathcal{M}(\mathbb{R}))$, which are the linear operators that are inverses of each other and defined by
\begin{align*}
    f\in M_{m,n}(\mathcal{M}(\mathbb{R})),\;(U_{\pm}f)(t)=f(t\pm d),\; \forall t\in \mathbb{R},
\end{align*}
with equality of functions in the usual sense of equal a.e. on $\mathbb{R}$.
\end{definition}

The reader can verify that these translation operators are well-defined functions which are linear and inverses of each other. Additional properties of these translation operators, that are fundamental in our studies, are given in the following lemma and proposition which are stated in terms of the DA operator $\mathcal{L}$ and the minimal and maximal operators $L_0'$ and $L$, respectively.
\begin{lemma}\label{NewLemDAEOpCommutesWithTranslationOps}
The subspace $D(\mathcal{L})$ of $[\mathcal{M}(\mathbb{R})]^n$ is an invariant subspace of the translation operators $U_{\pm}:[\mathcal{M}(\mathbb{R})]^n\rightarrow [\mathcal{M}(\mathbb{R})]^n$ and
\begin{align}
    U_{\pm}(D(\mathcal{L}))=D(\mathcal{L}).
\end{align}
Moreover, these translation operators commute with the DA operator $\mathcal{L}:D(\mathcal{L})\rightarrow [\mathcal{M}(\mathbb{R})]^n$, i.e.,
\begin{align}
    \mathcal{L}U_{\pm}=U_{\pm} \mathcal{L}.
\end{align}
In addition, $\ker \mathcal{L}$ is an invariant subspace of these translation operators and
\begin{align}
    U_{\pm}(\ker \mathcal{L})=\ker \mathcal{L}.
\end{align}
\end{lemma}
\begin{proof}
First, note that $D(\mathcal{L})$ is a subspace of $[\mathcal{M}(\mathbb{R})]^n$ and $U_{\pm}:[\mathcal{M}(\mathbb{R})]^n\rightarrow [\mathcal{M}(\mathbb{R})]^n$ are invertible linear operators on $[\mathcal{M}(\mathbb{R})]^n$ which are inverses of each other. Thus, if we prove that $U_{\pm}(D(\mathcal{L}))\subseteq D(\mathcal{L})$ then it will follow immediately that $D(\mathcal{L})$ is an invariant subspace of $U_{\pm}$ and $U_{\pm}(D(\mathcal{L}))=D(\mathcal{L})$. We will now prove that $U_{\pm}(D(\mathcal{L}))\subseteq D(\mathcal{L})$. Let $f\in D(\mathcal{L})$ or, equivalently, $f\in [\mathcal{M}(\mathbb{R})]^n$ and $Jf\in [W^{1,1}_{loc}(\mathbb{R})]^n$. Then it follows immediately that $U_{\pm}f\in [\mathcal{M}(\mathbb{R})]^n$ and $J(U_{\pm}f)=U_{\pm}(Jf)\in [W^{1,1}_{loc}(\mathbb{R})]^n$ since $U_{\pm}:[W^{1,1}_{loc}(\mathbb{R})]^n\rightarrow [W^{1,1}_{loc}(\mathbb{R})]^n$, that is, the subspace $[W^{1,1}_{loc}(\mathbb{R})]^n$ of $[\mathcal{M}(\mathbb{R})]^n$ is an invariant subspace of $U_{\pm}:[\mathcal{M}(\mathbb{R})]^n\rightarrow [\mathcal{M}(\mathbb{R})]^n$. This implies that $U_{\pm}f\in D(\mathcal{L})$. Thus we have proven $U_{\pm}(D(\mathcal{L}))\subseteq D(\mathcal{L})$, as desired. Next, we prove that
\begin{align}
    \mathcal{L}(U_{\pm}f)=U_{\pm} (\mathcal{L}f),\;\;\forall f\in D(\mathcal{L}).
\end{align}
Let $f\in D(\mathcal{L})$. Then $\mathcal{L}f\in [\mathcal{M}(\mathbb{R})]^n$ and, as we have just proved, $U_{\pm}f\in D(\mathcal{L})$ so that $ \mathcal{L}(U_{\pm}f), U_{\pm} (\mathcal{L}f)\in [\mathcal{M}(\mathbb{R})]^n$. We will now prove that $\mathcal{L}(U_{\pm}f)$ and $U_{\pm} (\mathcal{L}f)$ are equal as elements in $[\mathcal{M}(\mathbb{R})]^n$. First, it follows from the assumption (\ref{HypPeriodicityH}) that
\begin{align}
    U_{\pm}(Hf)=(U_{\pm}H)(U_{\pm}f)=H(U_{\pm}f),\;U_{\pm}(Wf)=(U_{\pm}W)(U_{\pm}f)=W(U_{\pm}f)
\end{align}
in $[\mathcal{M}(\mathbb{R})]^n$. It also then follows by hypotheses (\ref{HypForWOnab}) that
\begin{align}
    U_{\pm}(W^{-1}f)=(U_{\pm}W^{-1})(U_{\pm}f)=W^{-1}(U_{\pm}f)
\end{align}
in $[\mathcal{M}(\mathbb{R})]^n$. Hence, it follows that
\begin{align}
    U_{\pm} (\mathcal{L}f)=U_{\pm}W^{-1} \left(\frac{d}{dt}Jf+Hf\right)=W^{-1}\left(\frac{d}{dt}J (U_{\pm}f)+H(U_{\pm}f)\right)=\mathcal{L}(U_{\pm}f)
\end{align}
in $[\mathcal{M}(\mathbb{R})]^n$. Thus, to complete the proof it remains only to prove that $U_{\pm}(\ker \mathcal{L})=\ker \mathcal{L}$. Again as $U_{+}$ and $U_{-}$ are inverses of each other, we need only prove that $U_{\pm}(\ker \mathcal{L})\subseteq \ker \mathcal{L}$. Let $f\in \ker \mathcal{L}$. Then it follows from the above that
\begin{align}
    \mathcal{L}(U_{\pm}f)=(\mathcal{L}U_{\pm})f=(U_{\pm}\mathcal{L})f=U_{\pm}(\mathcal{L}f)=U_{\pm}0=0
\end{align}
and so $U_{\pm}f\in \ker \mathcal{L}$. This proves that $U_{\pm}(\ker \mathcal{L})\subseteq \ker \mathcal{L}$, which completes the proof.
\end{proof}

\begin{proposition}\label{NewPropLAndL0PrimeCommutesWithTranslationOps}
The subspaces $D(L_0')$ and $D(L)$ of $L^2(\mathbb{R};W)$ are invariant subspaces of the translation operators $U_{\pm}:L^2(\mathbb{R};W)\rightarrow L^2(\mathbb{R};W)$ and
\begin{align}
    U_{\pm}(D(L_0'))=D(L_0'),\;\;U_{\pm}(D(L))=D(L).
\end{align}
Moreover, these translation operators commute with the minimal and maximal operators $L_0':D(L_0')\to L^2(\mathbb{R};W)$ and $L:D(L)\to L^2(\mathbb{R};W)$, respectively, i.e.,
\begin{align}
    L_0'U_{\pm}=U_{\pm} L_0',\;\; LU_{\pm}=U_{\pm} L.
\end{align}
In addition, $\ker L_0'$ and $\ker L$ are invariant subspace of these translation operators and
\begin{align}
    U_{\pm}(\ker L_0')=\ker L_0',\;\;U_{\pm}(\ker L)=\ker L.
\end{align}
\end{proposition}
\begin{proof}
First, it follows immediately from the hypotheses (\ref{HypForWOnab}) and (\ref{HypPeriodicityH}) that the subspace $L^2(\mathbb{R};W)$ of $[\mathcal{M}(\mathbb{R})]^n$ is an invariant subspace of the translation operators $U_{\pm}:[\mathcal{M}(\mathbb{R})]^n\rightarrow [\mathcal{M}(\mathbb{R})]^n$ and their restriction $U_{\pm}:L^2(\mathbb{R};W)\rightarrow L^2(\mathbb{R};W)$ are linear operators which are inverses of each other. Second, on $D(L_0')$ we know that $L_0'=\mathcal{L}$ and on $D(L)$ we know that $L=\mathcal{L}$. It follows from this and Lemma \ref{NewLemDAEOpCommutesWithTranslationOps} that to prove the proposition we need only prove that
\begin{align}
    U_{\pm}(D(L_0'))\subseteq D(L_0'),\;\;U_{\pm}(D(L))\subseteq D(L).
\end{align}
Let $f\in D(L)$. Then $f\in L^2(\mathbb{R};W),$ $f\in D(\mathcal{L}),$ and $\mathcal{L}f\in L^2(\mathbb{R};W)$. It follows from this and Lemma \ref{NewLemDAEOpCommutesWithTranslationOps} that
\begin{align}
    U_{\pm}f\in L^2(\mathbb{R};W),\;\; U_{\pm}f\in D(\mathcal{L}),\;\;\mathcal{L}(U_{\pm}f)=U_{\pm}(\mathcal{L}f)\in L^2(\mathbb{R};W)
\end{align}
which implies $U_{\pm}f\in D(L)$. This proves $U_{\pm}(D(L))\subseteq D(L)$. Similarly, let $f\in D(L_0')$. Then $f\in D(L)$ and $Jf$ has compact support in $\mathbb{R}$. It follows that $U_{\pm}f\in D(L)$ and $J(U_{\pm}f)=U_{\pm}(Jf)$ has compact support in $\mathbb{R}$  which implies $U_{\pm}f\in D(L_0')$. This proves $U_{\pm}(D(L_0'))\subseteq D(L_0')$. Therefore, the proposition is proved.
\end{proof}

The next theorem is our main result on the spectral theory of $\mathcal{L}, L_0',$ and $L$ in regards to eigenvalues of finite multiplicities, in particular, whether they can have them or not. As it turns out, $L_0'$ and $L$ do not, whereas $\mathcal{L}$ can and in which case additional important information can be extracted. The key to the analysis is the commutivity of these operators with the translation operators $U_{\pm}$. 
\begin{theorem}\label{ThmFiniteDimKerLImpliesTrivialKerL}
Let $z\in \mathbb{C}$. Then following statements are true:\\
\noindent (a) $\{0\}\subseteq \ker(L_0'-zI)\subseteq \ker(L-zI)\subseteq \ker(\mathcal{L}-zI)$.\\
\noindent (b) If $\dim \ker (\mathcal{L}-zI)<\infty$ then either $\ker (\mathcal{L}-zI)=\{0\}$, in which case $\{0\}=\ker (L_0'-zI)=\ker (L-zI)=\ker (\mathcal{L}-zI)$, or there exists $f\in \ker (\mathcal{L}-zI)$ which is an eigenvector of the left translation operator $U_{+}:[\mathcal{M}(\mathbb{R})]^n\rightarrow [\mathcal{M}(\mathbb{R})]^n$ with a nonzero eigenvalue $\lambda$, i.e.,
\begin{align}
    \mathcal{L}f=zf,\;\;U_{+}f=\lambda f,\;\;f\not =0,\;\;\lambda\not =0,\label{MotivationForDefBlochSolutionsFloquetMultipliers}
\end{align}
for some $\lambda\in \mathbb{C}$.\\
\noindent (c) The translation operators $U_{\pm}:L^2(\mathbb{R};W)\rightarrow L^2(\mathbb{R};W)$ are unitary operators which are inverses of each other and have no eigenvalues.\\
\noindent (d) If $\dim \ker (L-zI)<\infty$ then $\ker (L_0'-zI)=\ker (L-zI)=\{0\}$.\\
\noindent (e) If $\dim \ker (L_0'-zI)<\infty$ then $\ker (L_0'-zI)=\{0\}$.
\end{theorem}
\begin{proof}
Let $z\in \mathbb{C}$. As the proof for $z\not =0$ is the same for the proof with $z=0$ (as we can just replace $H$ with $H-zW$ and use the fact that $\mathcal{L}_{H}-zI=\mathcal{L}_{H-zW}, L_{H}-zI=L_{H-zW}, L_{H,0}'-zI=L_{H-zW,0}'$), then we can assume without loss of generality that $z=0$.

(a): This follows immediately from the fact that $\mathcal{L}, L, L_0'$ are linear operators on their respective domains $D(\mathcal{L}), D(L), D(L_0')$ and by their definitions we have that $Lf=\mathcal{L}f$ for every $f\in D(L)$ and $L_0'f=\mathcal{L}f$ for every $f\in D(L_0')$ with $D(L_0')\subseteq D(L)\subseteq D(\mathcal{L})$.

(b): By Lemma \ref{NewLemDAEOpCommutesWithTranslationOps} we know that $U_{\pm}:\ker \mathcal{L}\to \ker \mathcal{L}$ are well-defined linear operators on the vector space $\ker L$ over the field $\mathbb{C}$ which are inverses of each other. Hence, if $\dim \ker \mathcal{L}<\infty$ and $\ker \mathcal{L}\not=\{0\}$ then this implies by elementary linear algebra that there exists an eigenvector $f\in \ker \mathcal{L}$ and corresponding eigenvalue $\lambda \in \mathbb{C}$ of $U_+$ and since $U_+$ is invertible then $\lambda\not =0$. This proves either $\ker \mathcal{L}=\{0\}$, in which case it follows from (a) that $\{0\}=\ker L_0'=\ker L=\ker \mathcal{L}$, or there exists $f\in \ker L$ which is an eigenvector of the left translation operator $U_{+}$ with a nonzero eigenvalue $\lambda$.

(c): First, its obvious that $U_{\pm}:L^2(\mathbb{R};W)\to L^2(\mathbb{R};W)$ are well-defined linear operators which are inverses of each other. Thus, to prove $U_{-}$ is unitary it suffices to prove that its inverse $U_{+}$ is unitary. Let $f,g\in L^2(\mathbb{R};W)$. Then
\begin{gather*}
    \langle U_{+}f,U_{+}g\rangle_W=\int_{\mathbb{R}}\langle W(t)(U_{+}f)(t),(U_{+}g)(t)\rangle dt=\int_{\mathbb{R}}\langle W(t)f(t+d),g(t+d)\rangle dt\\
    =\int_{\mathbb{R}}\langle W(t-d)f(t),g(t)\rangle dt=\int_{\mathbb{R}}\langle W(t)f(t),g(t)\rangle dt=\langle U_{+}f,U_{+}g\rangle_W.
\end{gather*}
This proves $U_{-}$ is a unitary operator on $L^2(\mathbb{R};W)$. Next, as $U_{-}$ is the inverse of the unitary operator $U_{+}$ then to show the former has no eigenvalues it suffices to show the latter has no eigenvalues. Also, $U_{+}$ is a unitary operator so if it did have an eigenvalue $\lambda\in \mathbb{C}$ then it follows that $|\lambda|=1$. Hence, suppose $f\in L^2(\mathbb{R};W)$ and $\lambda\in \mathbb{C}$ with $|\lambda|=1$ such that $U_+f=\lambda f.$ We will prove that $f=0$, which will complete the proof of statement (c). As the equalities $\lambda^mf(\cdot)=(U_+^mf)(\cdot)=f(\cdot +md)$ hold as elements in $L^2(\mathbb{R};W)$ for any $m\in \mathbb{Z}$ then
\begin{align*}
    &\infty > \langle f,f\rangle_W = \int_{\mathbb{R}}\langle W(t)f(t), f(t) \rangle dt =\sum_{m\in \mathbb{Z}}\int_{[md,(m+1)d]}\langle W(t)f(t), f(t) \rangle dt\\
    &=\sum_{m\in \mathbb{Z}}\int_{[0,d]}\langle W(t+md)f(t+md), f(t+md) \rangle dt
    =\sum_{m\in \mathbb{Z}}\int_{[0,d]}\langle \lambda^m W(t)f(t), \lambda^mf(t)\rangle dt\\
    &=\sum_{m\in \mathbb{Z}}\int_{[0,d]}\langle W(t)f(t), f(t)\rangle dt
\end{align*}
which implies $\int_{[0,d]}\langle W(t)f(t), f(t)\rangle dt=0$ and hence $\langle f, f \rangle_W =0$ so that $f=0$, as desired.

(d): By Proposition \ref{NewPropLAndL0PrimeCommutesWithTranslationOps} we know that $U_{+}:\ker L\to \ker L$ is a linear operator on the vector space $\ker L$ over the field $\mathbb{C}$, which commutes with the operator $L$, and is the restriction of the operator $U_{+}:L^2(\mathbb{R};W)\rightarrow L^2(\mathbb{R};W)$ to the invariant subspace $\ker L$. Hence, if $\dim \ker L<\infty$ and $\ker L\not=\{0\}$ then this implies by elementary linear algebra of the existence of an eigenvalue of $U_{+}:\ker L\to \ker L$ and hence of $U_{+}:L^2(\mathbb{R};W)\rightarrow L^2(\mathbb{R};W)$, a contradiction of (c). This proves that if $\dim \ker L<\infty$ then $\ker L=\{0\}$ and hence by (a), $\ker L_0'=\ker L=\{0\}$.

(e): The proof of (e) follows from Proposition \ref{NewPropLAndL0PrimeCommutesWithTranslationOps} using a similar argument as the proof of (d).
\end{proof}

\begin{remark}
Any function $f$ and scalar $\lambda=e^{ikd}$ satisfying the eigenvalue problem (\ref{MotivationForDefBlochSolutionsFloquetMultipliers}) with $z\in \mathbb{C}$ is called a Bloch solution and Floquet multiplier, respectively, of the periodic DAEs $\mathcal{L}$ in which case the scalar $k$ is called a wavenumber. The multivalued function $z=z(k)$ of $k$ is then referred to as the (complex) dispersion relation. 
\end{remark}

The next theorem is key to proving our main result in this paper on the self-adjoint spectral theory of periodic DAEs.
\begin{theorem}\label{NewThmMainThmSelfAdjointnessPeriodicDAEs}
  Suppose (\ref{NewHyp0PeriodicSetting}), (\ref{NewHyp1PeriodicSetting}), and (\ref{NewHyp2PeriodicSetting}) are true. If there exists $z_0\in \mathbb{C}\setminus \mathbb{R}$ such that
  \begin{align}
      \dim \ker (\mathcal{L}-z_0I) < \infty \text{ and } \dim \ker (\mathcal{L}-\overline{z_0}I) < \infty
  \end{align}
  or
  \begin{align}
      \dim \ker (L-z_0I) < \infty \text{ and } \dim \ker (L-\overline{z_0}I) < \infty
  \end{align}
  then
  \begin{align}
      \ker (L-z_0I)=\ker (L-\overline{z_0}I)=\{0\},
  \end{align}
  \begin{align}
      L:D(L)\to L^2(\mathbb{R};W)\text{ is a self-adjoint operator on } L^2(\mathbb{R};W),
  \end{align}
  and
  \begin{align}
      L\text{ has no eigenvalues of finite multiplicity.}
  \end{align}
\end{theorem}
\begin{proof}
First, by Theorem \ref{ThmFiniteDimKerLImpliesTrivialKerL}.(d) we know that $L$ has no eigenvalues of finite multiplicity. Second, by Theorem \ref{ThmFiniteDimKerLImpliesTrivialKerL}.(a) it suffices to prove the statement assuming $\dim \ker (L-z_0I) < \infty$ and $\dim \ker (L-\overline{z_0}I) < \infty$ for some $z_0\in \mathbb{C}\setminus \mathbb{R}$. Thus, assume this is true. Then $\ker (L-z_0I)=\ker (L-\overline{z_0}I)=\{0\}$ since $L$ has no eigenvalues of finite multiplicity. Assume also that (\ref{NewHyp0PeriodicSetting}), (\ref{NewHyp1PeriodicSetting}), and (\ref{NewHyp2PeriodicSetting}) are true. It now follows from this by Remark \ref{NewRemImportantOf3HypInThePeriodicSetting} that $L:D(L)\to L^2(\mathbb{R};W)$ is a self-adjoint operator on $L^2(\mathbb{R};W)$. This proves the theorem.
\end{proof}

The following is the main result of this paper.
\begin{theorem}\label{thm:MainResultOfThePaper}
Suppose (\ref{NewHyp0PeriodicSetting}) and that the local index-1 hypotheses (see Def.\ \ref{DefIndex1Hyp}) are true for $H-z_0W, W$ with respect to $J$ on the interval $\mathbb{R}$ for some $z_0\in\mathbb{C}\setminus\mathbb{R}$.  Then (\ref{NewHyp1PeriodicSetting}) and (\ref{NewHyp2PeriodicSetting}) are true, and
\begin{gather}
      L:D(L)\to L^2(\mathbb{R};W)\text{ is a self-adjoint operator on } L^2(\mathbb{R};W),\\
      L_0':D(L_0')\to L^2(\mathbb{R};W)\text{ is essentially self-adjoint with closure $L$ (i.e., $\overline{L_0'}=L$)},\\
      L\text{ has no eigenvalues of finite multiplicity}.
  \end{gather}
\end{theorem}
\begin{proof}
Assume the hypotheses. First, we have that
\begin{align}
    (H-z_0W)^*=H-\overline{z_0}W.
\end{align}
Thus, the local index-1 hypotheses are true for both $H-z_0W, W$ and $(H-z_0W)^*, W$ with respect to $J$ on the interval $\mathbb{R}$ (by Corollary \ref{cor:SimplifiedIndex1HypImplyItForAdjoints}).
Second,
\begin{align}
    \mathcal{L}-z_0I&=\mathcal{L}_{H-z_0}, L-z_0I=L_{H-z_0W}, L_0'-z_0I=L_{H-z_0W,0}',\\
     \mathcal{L}-\overline{z_0}I&=\mathcal{L}_{H-\overline{z_0}W}, L-\overline{z_0}I=L_{H-\overline{z_0}W}, L_0'-\overline{z_0}I=L_{H-\overline{z_0}W,0}'.
\end{align}
Thus, by Theorem \ref{ThmDenselyDefDomainMinMaxOpsLocallyIndex1Hyp} it follows that the subspaces $D(L_{H-z_0W,0}')$, $D(L_{(H-z_0W)^*,0}')$, $D(L_{H-z_0W}),$ and $D(L_{(H-z_0W)^*})$ are all dense in $L^2(\mathbb{R};W)$,
\begin{gather}
      (L_{H-z_0W,0}')^*=(\overline{L_{H-z_0W,0}'})^*=L_{(H-z_0W)^*}=L_{H-\overline{z_0}W},\\
      L_{(H-z_0W)^*}=\overline{L_{(H-z_0W)^*}}=(L_{(H-z_0W)^*})^{**},\\
      (L_{H-z_0W})^*=(L_{(H-z_0W)^*})^*=\overline{L_{H-z_0W,0}'}
  \end{gather}
  and
  \begin{gather}
      (L_{H-\overline{z_0}W,0}')^*=(L_{(H-z_0W)^*,0}')^*=(\overline{L_{(H-z_0W)^*,0}'})^*=L_{H-z_0W},\\
      L_{H-z_0W}=\overline{L_{H-z_0W}}=(L_{H-z_0W})^{**},\\
      (L_{H-z_0W})^*=\overline{L_{(H-z_0W)^*,0}'}.
  \end{gather}
  As we have
  \begin{gather}
      D(L_{H-z_0W,0}')=D(L_0-z_0I),\;D(L_{(H-z_0W)^*,0}')=D(L_0-\overline{z_0}I),\\
      D(L_{H-z_0W})=D(L-z_0I),\;\;D(L_{(H-z_0W)^*,0}')=D(L-\overline{z_0}I)
  \end{gather}
  and 
  \begin{align}
      D(L_0'-z_0I)=D(L_0'),\;D(L-z_0I)=D(L),
  \end{align}
  it follows that the conclusions of Theorem \ref{ThmDenselyDefDomainMinMaxOpsLocallyIndex1Hyp}, namely, statements (i), (ii), and (iii) are all true for the interval $\mathbb{R}$. This implies that (\ref{NewHyp1PeriodicSetting}) and (\ref{NewHyp2PeriodicSetting}) are true. Thus, by Theorem \ref{NewThmMainThmSelfAdjointnessPeriodicDAEs}, in order to prove our theorem we need only prove the claim
  \begin{align}
      \dim \ker (\mathcal{L}-z_0I) < \infty \text{ and } \dim \ker (\mathcal{L}-\overline{z_0}I) < \infty.
  \end{align}
  If this was not true then either
  \begin{align}
      \dim \ker (\mathcal{L}-z_0I) = \infty \text{ or } \dim \ker (\mathcal{L}-\overline{z_0}I) = \infty
  \end{align}
  and this would imply  that
  \begin{align}
      \dim \ker (\mathcal{L}_{[0,d]}-z_0I) = \infty \text{ or } \dim \ker (\mathcal{L}_{[0,d]}-\overline{z_0}I) = \infty
  \end{align}
  (where $\mathcal{L}_{[0,d]}$ denotes the DA operator associated with $[0,d], J, H|_{[0,d]}, W|_{[0,d]}$), a contradiction of Corollary \ref{CorKeyRanKerRelForIndex1HypL}. This proves the claim. Therefore, our theorem has been proven.
\end{proof}

Using Theorem \ref{thm:MainResultOfThePaper}, we now give an example of a maximal operator $L$ that is self-adjoint, has no eigenvalues of finite multiplicity, but does have an eigenvalue of infinite multiplicity.
\begin{example}\label{ex:MaxOpSelfAdjWithEigenvalueInftMult}
Consider the $d$-periodic DAEs (for any $d>0$) in canonical form:
\begin{equation*}
       J\frac{d}{dt}f+ Hf=\lambda Wf,
\end{equation*}
where 
\begin{gather*}
    J=i\begin{bmatrix}
    1 & 0 \\
    0 & 0
    \end{bmatrix},\;
    H=\begin{bmatrix}
    0 & 0 \\
    0 & 0
    \end{bmatrix},\;W=\begin{bmatrix}
    1 & 0 \\
    0 & 1
    \end{bmatrix}=I_2,\;f=\begin{bmatrix}
    f_1  \\
    f_2 
    \end{bmatrix},
\end{gather*}
that is,
\begin{align*}
    i\begin{bmatrix}
    1 & 0 \\
    0 & 0
    \end{bmatrix}\frac{d}{dt}\begin{bmatrix}
    f_1  \\
    f_2 
    \end{bmatrix}=\lambda \begin{bmatrix}
    f_1  \\
    f_2 
    \end{bmatrix}.
\end{align*}
Then the DA operator $\mathcal{L}:D(\mathcal{L})\to [\mathcal{M}(\mathbb{R})]^2$ associated with $\mathbb{R}, J, H, W$ is
\begin{gather*}
 D(\mathcal{L})=\left\{f\in [\mathcal{M}(\mathbb{R})]^2:Jf\in[W^{1,1}_{loc}(\mathbb{R})]^2\right\}
    = \left\{\begin{bmatrix}
    f_1  \\
    f_2 
    \end{bmatrix}:f_1\in W^{1,1}_{loc}(\mathbb{R}), f_2\in \mathcal{M}(\mathbb{R})\right\},\\
    \mathcal{L}f=W^{-1}\left(\frac{d}{dt}Jf+ Hf\right)=\begin{bmatrix}
    i\frac{df_1}{dt}  \\
    0 
    \end{bmatrix},\;\;f=\begin{bmatrix}
    f_1  \\
    f_2 
    \end{bmatrix}\in D(\mathcal{L}).
\end{gather*}
Next, the Hilbert space $L^2(\mathbb{R};W)$ is just
\begin{align*}
    L^2(\mathbb{R};W)=[L^2(\mathbb{R})]^2,
\end{align*}
and the maximal and minimal operators, $L:D(L)\rightarrow L^2(\mathbb{R};W)$ and $L_0':D(L_0')\rightarrow L^2(\mathbb{R};W)$, generated by $\mathcal{L}$ are
\begin{gather*}
    Lf=\mathcal{L}f=\begin{bmatrix}
    i\frac{df_1}{dt}  \\
    0 
    \end{bmatrix},\;\;\text{for }f=\begin{bmatrix}
    f_1  \\
    f_2 
    \end{bmatrix}\in D(L),\\
    \;L_0'f=\mathcal{L}f=\begin{bmatrix}
    i\frac{df_1}{dt}  \\
    0 
    \end{bmatrix},\;\;\text{for }f=\begin{bmatrix}
    f_1  \\
    f_2 
    \end{bmatrix}\in D(L_0'),
\end{gather*}
where 
\begin{gather*}
    D(L)=\{f\in L^2(\mathbb{R};W):f\in D(\mathcal{L}),\mathcal{L}f\in L^2(\mathbb{R};W)\}\\
    =\left\{\begin{bmatrix}
    f_1  \\
    f_2 
    \end{bmatrix}:f_1\in W^{1,1}_{loc}(\mathbb{R}),f_1, \frac{df_1}{dt}, f_2\in L^2(\mathbb{R}) \right\},\\
    D(L_0')=\{f\in D(L):Jf \text{ has compact support contained in the interior of }\mathbb{R}\}\\
    =\left\{\begin{bmatrix}
    f_1  \\
    f_2 
    \end{bmatrix}:f_2\in L^2(\mathbb{R}), f_1\in W^{1,1}_{loc}(\mathbb{R}), f_1 \text{ has compact support}\right\}.
\end{gather*}

Let us now prove that $L$ has only one eigenvalue, namely, $\lambda=0$, and it is an eigenvalue of infinite multiplicity. Let $\lambda\in \mathbb{C}$, $\lambda\not=0$. Then
\begin{gather*}
    \mathcal{L}f=\lambda f \Leftrightarrow f=\begin{bmatrix}
    f_1  \\
    f_2 
    \end{bmatrix},f_2=0, f_1\in W^{1,1}_{loc}(\mathbb{R}), \frac{df_1}{dt}=-i\lambda f_1\\
    \Leftrightarrow f=\begin{bmatrix}
    f_1  \\
    f_2 
    \end{bmatrix}, f_2=0, f_1(t)=ce^{-i\lambda t}, c\in \mathbb{C}
\end{gather*}
from which it follows that
\begin{gather*}
    Lf=\lambda f \Leftrightarrow f=0.
\end{gather*}
This proves that if $\lambda\not=0$ then $\lambda$ is not an eigenvalue of $L$. Next, we will prove $\lambda=0$ is an eigenvalue of $L$ of infinite multiplicity. First,
\begin{gather*}
    \mathcal{L}f=0 \Leftrightarrow f=\begin{bmatrix}
    f_1  \\
    f_2 
    \end{bmatrix},f_2\in \mathcal{M}(\mathbb{R}), f_1\in W^{1,1}_{loc}(\mathbb{R}), \frac{df_1}{dt}=0\\
    \Leftrightarrow f=\begin{bmatrix}
    f_1  \\
    f_2 
    \end{bmatrix}, f_2\in \mathcal{M}(\mathbb{R}), f_1(t)=c, c\in \mathbb{C} \text{ (i.e., $f_1$ is a constant function)}
\end{gather*}
from which it follows that
\begin{gather*}
    Lf=0\Leftrightarrow f=\begin{bmatrix}
    f_1  \\
    f_2 
    \end{bmatrix}, f_2,f_1\in L^2(\mathbb{R}), f_1(t)=c, c\in \mathbb{C}
    \Leftrightarrow \begin{bmatrix}
    0 \\
    f_2 
    \end{bmatrix}, f_2\in L^2(\mathbb{R}),
\end{gather*}
which proves that $\lambda=0$ is an eigenvalue of $L$ with infinite multiplicity since its eigenspace $E_0$, i.e.,
\begin{gather*}
    E_{0}=\left\{\begin{bmatrix}
    0 \\
    f_2 
    \end{bmatrix}: f_2\in L^2(\mathbb{R})\right\},
\end{gather*}
is infinite dimensional.

Our final goal here is to prove that $L$ is self-adjoint by proving that the hypotheses of Theorem \ref{thm:MainResultOfThePaper} are satisfied (which also then shows that $L_0'$ is essentially self-adjoint with closure $L$). To do this, let $z_0\in \mathbb{C}$. We now want to consider whether or not the local index-$1$ hypotheses (see Def.\ \ref{DefIndex1Hyp}) are true for $H-z_0W, W$ with respect to $J$ on the interval $\mathbb{R}$, where in this example
\begin{gather*}
    H-z_0W=-z_0I_2.
\end{gather*}
First, the unitary matrix $$V=I_2\in M_2(\mathbb{C}),$$ already has $$J=V^{-1}JV=[J_{ij}]_{i,j=1,2}$$ in $2\times2$ block partitioned matrix form as (\ref{BlockStructJ}) with
\begin{align*}
    J_{ij}\in M_{n_i\times n_j}(\mathbb{C}), i,j=1,2;\;\;J_{11}=iI_1,\; J_{ij}=0,\;\;(i,j)\not = (1,1),\;\;\det(J_{11})\neq 0,
\end{align*}
where $I_1=[1]$ is the $1\times 1$ identity matrix and $n_1,n_2\in \mathbb{N}$ are defined by
\begin{align*}
    n_1:=\operatorname{rank} J=1,\;\;n_2:=\dim \ker J=\operatorname{nullity}(J)=2-n_1=1.
\end{align*}
Furthermore, both of the matrices $H=V^{-1}HV=[H_{ij}]_{i,j=1,2}, W=V^{-1}WV=[W_{ij}]_{i,j=1,2}\in M_{2}(\mathcal{M}(\mathbb{R}))$ are block partitioned already conformal to the block structure of $V^{-1}JV$ in (\ref{BlockStructJ}), where
\begin{gather*}
    H_{ij}=0, W_{ij}\in M_{n_i\times n_j}(\mathcal{M}(\mathbb{R})),\;\;i,j=1,2;\;W_{11}=W_{22}=I_1,\;W_{12}=W_{21}=0.
\end{gather*}

Next, it follows from this that the local index-$1$ hypotheses (see Def.\ \ref{DefIndex1Hyp}) for $H-z_0W, W$ with respect to $J$ on the interval $\mathbb{R}$ are not satisfied if $z_0=0$ since $H_{22}=0$ is not invertible, but are satisfied if $z_0\not =0$ since
\begin{gather*}
    \frac{-1}{z_0}I_1=(H-z_0W)_{22}^{-1}\in M_{n_2}(\mathcal{M}(\mathbb{R})),\\
    0=(H-z_0W)_{12}(H-z_0W)_{22}^{-1}W_{22}^{1/2}\in M_{n_1\times n_2}(L^{2}_{loc}(\mathbb{R})),\\
    -z_0I_1=(H-z_0W)/(H-z_0W)_{22}, W_{11}=I_1\in M_{n_1}(L^1_{loc}(\mathbb{R})),\\
    \frac{-1}{z_0}I_1=W_{22}^{1/2}(H-z_0W)_{22}^{-1}W_{22}^{1/2}\in M_{n_2}(L^{\infty}_{loc}(\mathbb{R})),\\
    0=W_{22}^{1/2}(W_{22}^{-1}W_{21}-H_{22}^{-1}H_{21})\in M_{n_2\times n_1}(L^2_{loc}(\mathbb{R})),\\
    I_1=W/W_{22}\in M_{n_1}(L^1_{loc}(\mathbb{R})).
\end{gather*}
It follows from this that the local index-$1$ hypotheses (see Def.\ \ref{DefIndex1Hyp}) for $H-z_0W, W$ with respect to $J$ on the interval $\mathbb{R}$ are true in this example if and only if $z_0\in\mathbb{C}\setminus\{0\}$. In particular, we can take $z_0=i$ from which it follows in this example that $L$ is self-adjoint by Theorem \ref{thm:MainResultOfThePaper}, has $\lambda=0$ as its only eigenvalue, and this eigenvalue has infinite multiplicity with eigenspace $E_0$ above.
\end{example}

\subsection{\label{subsec:SimplifyIndex1Hyp}Simplifying the local index-1 hypotheses}
As Theorem \ref{thm:MainResultOfThePaper} is the main result of this paper, we would like to find simpler hypotheses that are sufficient for those hypotheses in Theorem \ref{thm:MainResultOfThePaper} to be true. The next proposition and the theorem that follows are our main results in this regard. First, we begin with a lemma.
\begin{lemma}\label{lem:WeightedResolventLpProperties}
If (\ref{NewHyp0PeriodicSetting}) is satisfied and $z_0\in\mathbb{C}\setminus\mathbb{R}$, then $(W^{-1/2}HW^{-1/2}-z_0I_{n})^{-1}$ and $H-z_0W$ are invertible for a.e.\ $t\in \mathbb{R}$ and
    \begin{gather}
   W^{1/2}(H-z_0W)^{-1}W^{1/2}=(W^{-1/2}HW^{-1/2}-z_0I_{n})^{-1}\in M_{n}(L^{\infty}_{loc}(\mathbb{R})),\\
        (H-z_0W)^{-1}=W^{-1/2}(W^{-1/2}HW^{-1/2}-z_0I_{n})^{-1}W^{-1/2}\in M_{n}(\mathcal{M}(\mathbb{R})).
    \end{gather}
Furthermore, if $\det J =0$ then $(W^{-1/2}HW^{-1/2}-z_0I_{n})^{-1}$ and $(H-z_0W)_{22}$ are invertible for a.e.\ $t\in \mathbb{R}$ and
    \begin{gather}
    W_{22}^{1/2}(H-z_0W)_{22}^{-1}W_{22}^{1/2}=(W_{22}^{-1/2}H_{22}W_{22}^{-1/2}-z_0I_{n_2})^{-1}\in M_{n_2}(L^{\infty}_{loc}(\mathbb{R})),\\
        (H-z_0W)_{22}^{-1}=(H_{22}-z_0W_{22})^{-1}\\
        =W_{22}^{-1/2}(W_{22}^{-1/2}H_{22}W_{22}^{-1/2}-z_0I_{n_2})^{-1}W_{22}^{-1/2}\in M_{n_2}(\mathcal{M}(\mathbb{R})).\notag
    \end{gather}
\end{lemma}
\begin{proof}
Suppose (\ref{NewHyp0PeriodicSetting}) is satisfied. Let $z_0\in\mathbb{C}\setminus\mathbb{R}$. Now recall, for any $A\in M_m(\mathbb{C})$ such that $A^*=A$, it follows that for any $x\in \mathbb{C}^m$,
\begin{gather*}
    ||(A-z_0I_m)x||||x||\geq |\langle x, (A-z_0I_m)x\rangle|\geq |\operatorname{Im}\langle x, (A-z_0I_m)x\rangle|\\
    =|\langle x, \operatorname{Im}(A-z_0I_m)x\rangle|=|\operatorname{Im}(z_0)||x||^2
\end{gather*}
which simultaneously proves that $A-z_0I_m$ is invertible and that
\begin{gather*}
    ||(A-z_0I_m)^{-1}||\leq \frac{1}{|\operatorname{Im}(z_0)|}.
\end{gather*}
Next, if $C,B\in M_m(\mathbb{C})$ such that $C^*=C$ and $B$ is positive semidefinite and invertible then $A:=B^{-1/2}CB^{-1/2}$ satisfies $A^*=A$ and
\begin{align}
    C-z_0B=B^{1/2}(B^{-1/2}CB^{-1/2}-z_0I_m)B^{1/2}=B^{1/2}(A-z_0I_m)B^{1/2}
\end{align}
so that $C-z_0B$ is invertible with
\begin{align}
    (C-z_0B)^{-1}=B^{-1/2}(A-z_0I_m)^{-1}B^{-1/2},\;(A-z_0I_m)^{-1}=B^{1/2}(C-z_0B)^{-1}B^{1/2}.
\end{align}
Finally, recall that if $D\in M_m(\mathbb{C})$ is any invertible matrix then, in terms of its determinant $\det D$ and its adjugate $\operatorname{adj}D$ (i.e., the transpose of the cofactor matrix of $D$), the inverse formula holds:
\begin{align*}
    D^{-1}=\frac{1}{\det D}\operatorname{adj}D.
\end{align*}
The proof of this lemma now follows immediately from these elementary facts and the hypotheses on $H,W$ along with the equality $(H-z_0W)_{22}=H_{22}-z_0W_{22}$ in $M_{n_2}(\mathcal{M}(\mathbb{R}))$ in the case $\det J=0$.
\end{proof}

The next proposition gives simpler conditions, which are equivalent to the hypotheses of Theorem \ref{thm:MainResultOfThePaper}.
\begin{proposition}\label{prop:MainResultOfPaperOnIndex1HypEquiv}
Suppose (\ref{NewHyp0PeriodicSetting}) is satisfied. Then the following statements are true:
\begin{itemize}
    \item[(a)] If $\det J\not=0$ then $H,W\in M_{n}(L^{1}_{loc}(\mathbb{R}))$ if and only if the local index-1 hypotheses (see Def.\ \ref{DefIndex1Hyp}) are true for $H-z_0W, W$ with respect to $J$ on the interval $\mathbb{R}$, for some $z_0\in\mathbb{C}\setminus\mathbb{R}$ (in which case, its true for every $z_0\in\mathbb{C}$).
    \item[(b)] If $\det J=0$ and $z_0\in\mathbb{C}\setminus\mathbb{R}$ then the local index-1 hypotheses (see Def.\ \ref{DefIndex1Hyp}) are true for $H-z_0W, W$ with respect to $J$ on the interval $\mathbb{R}$ if and only if the following conditions are satisfied:
    \begin{gather}
    W_{11}\in M_{n_1}(L^1_{loc}(\mathbb{R})),\label{EquivLocIndex1CondForz0AndConjugate1}\\
     H_{11}-H_{12}(H_{22}-z_0W_{22})^{-1}H_{21}\in M_{n_1}(L^1_{loc}(\mathbb{R}))\label{EquivLocIndex1CondForz0AndConjugate2}.
\end{gather}
\end{itemize}
Moreover, the statements (a) and (b) are still true if we replace $L^1_{loc}(\mathbb{R})$ by $L^1([0,d])$.
\end{proposition}
\begin{proof}
Suppose (\ref{NewHyp0PeriodicSetting}) is satisfied. We begin by proving (a). First, suppose $H,W\in M_{n}(L^{1}_{loc}(\mathbb{R}))$. Let $z_0\in \mathbb{C}$. Then we have $H-z_0W\in M_{n}(L^{1}_{loc}(\mathbb{R}))$. Hence, if $\det J\not=0$ then it follows immediately from  Def.\ \ref{DefIndex1Hyp} that the local index-1 hypotheses (see Def.\ \ref{DefIndex1Hyp}) are true for $H-z_0W, W$ with respect to $J$ on the interval $\mathbb{R}$. Conversely, suppose $\det J\not=0$ and the local index-1 hypotheses (see Def.\ \ref{DefIndex1Hyp}) are true for $H-z_0W, W$ with respect to $J$ on the interval $\mathbb{R}$, for some $z_0\in\mathbb{C}\setminus\mathbb{R}$. Then $\operatorname{Im}z_0\not=0$ and $H-z_0W,H-\overline{z_0}W=(H-z_0W)^*\in M_{n}(L^{1}_{loc}(\mathbb{R}))$. Hence, 
\begin{align*}
    W=\frac{1}{\operatorname{Im}z_0}\frac{1}{2i}[(H-\overline{z_0}W)-(H-z_0W)]\in M_{n}(L^{1}_{loc}(\mathbb{R})) 
\end{align*}
which implies $H=(H-z_0W)+z_0W\in M_{n}(L^{1}_{loc}(\mathbb{R})).$ This proves statement (a).

We now prove statement (b). Assume $\det J=0$ and let $z_0\in\mathbb{C}\setminus\mathbb{R}$. Then the local index-1 hypotheses (see Def.\ \ref{DefIndex1Hyp} for $H-z_0W, W$ with respect to $J$ on the interval $\mathbb{R}$ by Lemma \ref{lem:SimplifiedIndex1Hyp} and Lemma \ref{lem:WeightedResolventLpProperties} are equivalent to the following:
\begin{gather}
    (H-z_0W)/(H-z_0W)_{22}, W_{11}\in M_{n_1}(L^1_{loc}(\mathbb{R})),\\
    (H-z_0W)_{12}(H-z_0W)_{22}^{-1}W_{22}^{1/2}\in M_{n_1\times n_2}(L^{2}_{loc}(\mathbb{R})),\\
    W_{22}^{1/2}(H-z_0W)_{22}^{-1}(H-z_0W)_{21}\in M_{n_2\times n_1}(L^2_{loc}(\mathbb{R})).
\end{gather}
Considering this even further, using Holder's inequality and the facts that $(H-z_0W)_{12}=H_{12}-z_0W_{12}, (H-z_0W)_{21}=H_{21}-z_0W_{21},$ and $W_{22}^{1/2}(H-z_0W)_{22}^{-1}W_{22}^{1/2}\in M_{n_2}(L^{\infty}_{loc}(\mathbb{R}))$ (by Lemma \ref{lem:WeightedResolventLpProperties}), and $W_{22}^{-1/2}W_{21}\in M_{n_2\times n_1}(L^2_{loc}(I))$ (see the proof of Lemma \ref{lem:SimplifiedIndex1Hyp}), we see that those conditions are equivalent to the following:
\begin{gather}
    (H-z_0W)/(H-z_0W)_{22}, W_{11}\in M_{n_1}(L^1_{loc}(\mathbb{R})),\\
    H_{12}(H-z_0W)_{22}^{-1}W_{22}^{1/2}\in M_{n_1\times n_2}(L^{2}_{loc}(\mathbb{R})),\\
   W_{22}^{1/2}(H-z_0W)_{22}^{-1}H_{21}\in M_{n_2\times n_1}(L^{2}_{loc}(\mathbb{R})).
\end{gather}
Therefore, by expanding
\begin{gather*}
    (H-z_0W)/(H-z_0W)_{22}=(H-z_0W)_{11}-(H-z_0W)_{12}(H-z_0W)_{22}^{-1}(H-z_0W)_{21}\\
    =H_{11}-z_0W_{11}-H_{12}(H-z_0W)_{22}^{-1}H_{21}\\
    +z_0H_{12}(H-z_0W)_{22}^{-1}W^{1/2}_{22}W^{-1/2}_{22}W_{21}
    +z_0W_{12}W^{-1/2}_{22}W^{1/2}_{22}(H-z_0W)_{22}^{-1}H_{21}\\
    -z_0^2W_{12}W^{-1/2}_{22}W^{1/2}_{22}(H-z_0W)_{22}^{-1}W^{1/2}_{22}W^{-1/2}_{22}W_{21},
\end{gather*}
we see that those conditions are equivalent to the following conditions: \begin{gather}
    W_{11}\in M_{n_1}(L^1_{loc}(\mathbb{R})),\\
    H_{12}(H_{22}-z_0W_{22})^{-1}W_{22}^{1/2}\in M_{n_1\times n_2}(L^{2}_{loc}(\mathbb{R})),\\
    W_{22}^{1/2}(H_{22}-z_0W_{22})^{-1}H_{21}\in M_{n_2\times n_1}(L^{2}_{loc}(\mathbb{R})),\\
    H_{11}-H_{12}(H_{22}-z_0W_{22})^{-1}H_{21}\in M_{n_1}(L^1_{loc}(\mathbb{R})).
\end{gather}
Now considering the imaginary part of the last matrix function [recall, for a matrix $A\in M_{m}(\mathbb{C}),$ the imaginary part of $A$ is defined by $\operatorname{Im}A=\frac{1}{2i}(A-A^*)$ and if $A$ is invertible then $\operatorname{Im}(A^{-1})=\operatorname{Im}(A^{-1}A^*(A^{-1})^*)=-A^{-1}(\operatorname{Im}A)(A^{-1})^*$ and $\operatorname{Im}(A^{-1})=\operatorname{Im}((A^{-1})^*A^*A^{-1})=-(A^{-1})^*(\operatorname{Im}A)A^{-1}$] we find that
\begin{gather}
    \hspace{-2em}\operatorname{Im}[ H_{11}-H_{12}(H_{22}-z_0W_{22})^{-1}H_{21}]=-(\operatorname{Im}z_0)H_{12}(H_{22}-z_0W_{22})^{-1}W_{22}[(H_{22}-z_0W_{22})^{-1}]^*H_{21}\\
    =-(\operatorname{Im}z_0)H_{12}(H_{22}-z_0W_{22})^{-1}W_{22}^{1/2}[H_{12}(H_{22}-z_0W_{22})^{-1}W_{22}^{1/2}]^*
\end{gather}
and
\begin{gather}
   \hspace{-2em}\operatorname{Im}[ H_{11}-H_{12}(H_{22}-z_0W_{22})^{-1}H_{21}]=-(\operatorname{Im}z_0)H_{12}[(H_{22}-z_0W_{22})^{-1}]^*W_{22}(H_{22}-z_0W_{22})^{-1}H_{21}\\
   =-(\operatorname{Im}z_0)[W_{22}^{1/2}(H_{22}-z_0W_{22})^{-1}H_{21}]^*W_{22}^{1/2}(H_{22}-z_0W_{22})^{-1}H_{21},
\end{gather}
and we see from this that those conditions are equivalent to the conditions (\ref{EquivLocIndex1CondForz0AndConjugate1})and (\ref{EquivLocIndex1CondForz0AndConjugate2}). This completes the proof of statement (b).

Finally, the statements (a) and (b) are still true if we replace $L^1_{loc}(\mathbb{R})$ by $L^1([0,d])$. The reason is that since $H,W$ are $d$-periodic functions, then $H,W\in M_{n}(L^{1}_{loc}(\mathbb{R}))$ if and only if $H,W\in M_{n}(L^{1}([0,d])$. Similarly, due to the $d$-periodicity, (\ref{EquivLocIndex1CondForz0AndConjugate1}) and (\ref{EquivLocIndex1CondForz0AndConjugate2}) are true if and only if $W_{11}, H_{11}-H_{12}(H_{22}-z_0W_{22})^{-1}H_{21}\in M_{n_1}(L^1([0,d]))$. This completes the proof of the proposition.
\end{proof}

The next theorem gives very simple conditions which are sufficient for the hypotheses in Theorem \ref{thm:MainResultOfThePaper} to be true.
\begin{theorem}\label{thm:StrongerIndex1Hyp}
Suppose (\ref{NewHyp0PeriodicSetting}) is satisfied and $\det J=0$. If
\begin{gather}
    H_{11},W_{11}, H_{12}W_{22}^{-1}H_{21}\in M_{n_1}(L^1_{loc}(\mathbb{R}))\label{StrongerIndex1Hyp}
\end{gather}
then the local index-1 hypotheses are true for $H-z_0W, W$ with respect to $J$ on the interval $\mathbb{R}$, for every $z_0\in\mathbb{C}\setminus \mathbb{R}.$ Moreover, the theorem is still true if we replace $L^1_{loc}(\mathbb{R})$ by $L^1([0,d])$ in (\ref{StrongerIndex1Hyp}).
\end{theorem}
\begin{proof}
Assume the hypotheses. Let $z_0\in \mathbb{C}\setminus\mathbb{R}$. To prove the theorem, it suffices by Proposition \ref{prop:MainResultOfPaperOnIndex1HypEquiv}.(b), to show that the conditions (\ref{EquivLocIndex1CondForz0AndConjugate1}) and (\ref{EquivLocIndex1CondForz0AndConjugate2}) are satisfied. First, condition (\ref{EquivLocIndex1CondForz0AndConjugate1}) is true by the hypotheses. Second, by our hypotheses
\begin{align}
    (W_{22}^{-1/2}H_{21})^*W_{22}^{-1/2}H_{21}=H_{12}W_{22}^{-1/2}(H_{12}W_{22}^{-1})^*=H_{12}W_{22}^{-1}H_{21}\in M_{n_1}(L^1_{loc}(\mathbb{R}))
\end{align}
implying $H_{12}W_{22}^{-1/2}\in M_{n_1\times n_2}(L^2_{loc}(\mathbb{R})), W_{22}^{-1/2}H_{21}\in M_{n_2\times n_1}(L^2_{loc}(\mathbb{R}))$. Third, it follows from this, Holder's inequality, and Lemma \ref{lem:WeightedResolventLpProperties} that
\begin{gather}
    H_{11}-H_{12}(H_{22}-z_0W_{22})^{-1}H_{21}\\ =H_{11}-H_{12}W_{22}^{-1/2}[W_{22}^{1/2}(H_{22}-z_0W_{22})^{-1}W_{22}^{1/2}]W_{22}^{-1/2}H_{21}\in M_{n_1}(L^1_{loc}(\mathbb{R})).\notag
\end{gather}
 This shows that the conditions (\ref{EquivLocIndex1CondForz0AndConjugate1}) and (\ref{EquivLocIndex1CondForz0AndConjugate2}) are satisfied which proves the theorem except for the ``Moreover,..." part of this theorem. But this follows from the fact that since $H_{11},W_{11}, H_{12}, W_{22}, H_{21}$ are $d$-periodic functions, then (\ref{StrongerIndex1Hyp}) is true if and only if $H_{11},W_{11},H_{12}W_{22}^{-1}H_{21}\in M_{n_1}(L^{1}([0,d]))$. This completes the proof of the theorem.
\end{proof}

\begin{remark}
As we have just proven, the hypotheses (\ref{StrongerIndex1Hyp}) in Theorem \ref{thm:StrongerIndex1Hyp} imply those of  (\ref{EquivLocIndex1CondForz0AndConjugate1}) and (\ref{EquivLocIndex1CondForz0AndConjugate2}) in Theorem \ref{prop:MainResultOfPaperOnIndex1HypEquiv} for every $z_0\in\mathbb{C}\setminus \mathbb{R}$, provided (\ref{NewHyp0PeriodicSetting}) is satisfied and $\det J=0$. We would like an answer to the following question we leave open: Is the converse is true or not?
\end{remark}

Next, we give an example that shows how useful Theorem \ref{thm:StrongerIndex1Hyp} can be. 
\begin{example}\label{ex:GenerizationOfExMaxOpSelfAdjWithEigenvalueInftMult}
Let $n\in \mathbb{N}$ and $J\in M_n(\mathbb{C})$ with $J^*=-J\not=0$. Consider the $d$-periodic DAEs (for any $d>0$) in canonical form:
\begin{equation*}
       J\frac{d}{dt}f+ Hf=\lambda Wf,
\end{equation*}
where
\begin{gather*}
    H=0,\;W=I_n,
\end{gather*}
that is,
\begin{equation*}
       J\frac{d}{dt}f=\lambda f.
\end{equation*}
The DA operator $\mathcal{L}:D(\mathcal{L})\to [\mathcal{M}(\mathbb{R})]^n$ associated with $\mathbb{R}, J, H, W$ is
\begin{gather*}
 D(\mathcal{L})=\left\{f\in [\mathcal{M}(\mathbb{R})]^n:Jf\in[W^{1,1}_{loc}(\mathbb{R})]^n\right\}\\
    \mathcal{L}f=W^{-1}\left(\frac{d}{dt}Jf+ Hf\right)=\frac{d}{dt}Jf.
\end{gather*}
Next, the Hilbert space $L^2(\mathbb{R};W)$ is just
\begin{align*}
    L^2(\mathbb{R};W)=[L^2(\mathbb{R})]^n,
\end{align*}
and the maximal and minimal operators, $L:D(L)\rightarrow L^2(\mathbb{R};W)$ and $L_0':D(L_0')\rightarrow L^2(\mathbb{R};W)$, generated by $\mathcal{L}$ are
\begin{gather*}
    Lf=\mathcal{L}f=\frac{d}{dt}Jf,\;\;\text{for }f\in D(L),\\
    \;L_0'f=\mathcal{L}f=\frac{d}{dt}Jf,\;\;\text{for }f\in D(L_0'),
\end{gather*}
where 
\begin{gather*}
    D(L)=\{f\in L^2(\mathbb{R};W):f\in D(\mathcal{L}),\mathcal{L}f\in L^2(\mathbb{R};W)\}\\
    =\left\{f\in [L^2(\mathbb{R})]^n:Jf\in[W^{1,1}_{loc}(\mathbb{R})]^n,\frac{d}{dt}Jf\in [L^2(\mathbb{R})]^n \right\},\\
    D(L_0')=\{f\in D(L):Jf \text{ has compact support contained in the interior of }\mathbb{R}\}\\
    =\{f\in D(L):(Jf)(t)=0 \text{ for all $t$ sufficiently large}\}.
\end{gather*}

Let us now compare and contrast the spectral theory for $L$ in the two different possible cases: (i) $\det J\not = 0$; (ii) $\det J= 0$. First, in either case (i) or (ii), it follows immediately from Theorem \ref{thm:StrongerIndex1Hyp} and Proposition \ref{prop:MainResultOfPaperOnIndex1HypEquiv} that the local index-1 hypotheses (see Def.\ \ref{DefIndex1Hyp}) are true for $H-z_0W, W$ with respect to $J$ on the interval $\mathbb{R}$ for every $z_0\in\mathbb{C}\setminus\mathbb{R}$ and hence by Theorem \ref{thm:MainResultOfThePaper} it follows that $L$ is a self-adjoint operator on $[L^2(\mathbb{R})]^n$, $L_0'$ is essentially self-adjoint with closure $L$, and $L$ has no eigenvalues of finite multiplicity. 

We will now prove explicitly that in the case (i) $\det J\not = 0$, the operator $L$ has no eigenvalues, whereas in the case (ii) $\det J = 0$, the operator $L$ has only one eigenvalue, namely, $\lambda=0$ and then we will calculate its eigenspace $E_0$ to show $\dim E_0=\infty$ thereby proving $\lambda=0$ is an eigenvalue of $L$ of infinite multiplicity.

Suppose (i) $\det J\not = 0$. Then 
\begin{align*}
    D(L)=\left\{f\in [L^2(\mathbb{R})]^n:Jf\in[W^{1,1}_{loc}(\mathbb{R})]^n,\frac{d}{dt}Jf\in [L^2(\mathbb{R})]^n \right\}\\
    =\left\{f\in [L^2(\mathbb{R})]^n:f\in[W^{1,1}_{loc}(\mathbb{R})]^n,\frac{d}{dt}f\in [L^2(\mathbb{R})]^n \right\}.
\end{align*}
Hence if $f$ is an eigenvector of $L$ with corresponding eigenvalue $\lambda$ then $f\in D(L)$ and $Lf=\lambda f$ implying $\frac{df}{dt}=\lambda Jf$ and hence $f(t)=e^{t\lambda J}f(0),\;\forall t\in \mathbb{R}$. As $(iJ)^*=iJ$, this implies that $iJ$ has a orthonormal basis of eigenvectors $v_1,\ldots, v_n\in \mathbb{C}^n$ with corresponding eigenvalues $\omega_1,\ldots, \omega_n\in \mathbb{R}$ and there exists $a_1,\ldots, a_n\in\mathbb{C}$ such that $f(0)=\sum_{j=1}^na_jv_j$ implying $f(t)=\sum_{j=1}^na_je^{-it\lambda \omega_j}v_j,\;\forall t\in \mathbb{R}$. As $f\in D(L)$ and is an eigenvector of $L$ then $0<(f,f)<\infty$, but
\begin{gather*}
    (f,f)=\int_{\mathbb{R}}f(t)^*f(t)dt=\sum_{j=1}^n\sum_{k=1}^n\int_{\mathbb{R}}[a_je^{-it\lambda \omega_j}v_j]^*[a_ke^{-it\lambda \omega_k}v_k]dt\\
    =\sum_{a_j\not=0}|a_j|^2\int_{\mathbb{R}}\overline{(e^{-it\lambda \omega_j})}e^{-it\lambda \omega_j}dt\\
    =\sum_{a_j\not=0}|a_j|^2\int_{\mathbb{R}}e^{it\omega_j(\overline{\lambda}-\lambda)}dt\\
    =\sum_{a_j\not=0}|a_j|^2\int_{\mathbb{R}}e^{-i2t\omega_j\operatorname{Im}\lambda}dt
\end{gather*}
and since the integral
\begin{gather*}
    \int_{\mathbb{R}}e^{-i2t\omega\operatorname{Im}\lambda}dt=\lim_{r\rightarrow -\infty, R\rightarrow +\infty}\int_{r}^{R}e^{-i2t\omega\operatorname{Im}\lambda}dt
\end{gather*}
doesn't exist as a finite limit for any $\omega\in\mathbb{R}$, this yields a contradiction that $L$ had an eigenvector. This proves that $L$ has no eigenvalues in the case (i) $\det J\not = 0$.

Suppose (ii) $\det J = 0$. Then since $(iJ)^*=iJ\not=0$, there is a unitary matrix $$V\in M_n(\mathbb{C}),$$ such that $V^{-1}JV=[J_{ij}]_{i,j=1,2}$ is in a $2\times2$ block partitioned matrix form as in (\ref{BlockStructJ}) with
\begin{align*}
    J_{ij}\in M_{n_i\times n_j}(\mathbb{C}), i,j=1,2,\; J_{ij}=0,\;\;(i,j)\not = (1,1),\;\;\det(J_{11})\neq 0,
\end{align*}
where $n_1,n_2\in \mathbb{N}$ are defined by
\begin{align*}
    n_1:=\operatorname{rank} J\geq 1,\;\;n_2:=\dim \ker J=\operatorname{nullity}(J)=n-n_1\geq 1.
\end{align*}
Furthermore, both of the matrices $H=V^{-1}HV=[H_{ij}]_{i,j=1,2}=0, W=V^{-1}WV=[W_{ij}]_{i,j=1,2}=I_n\in M_{n}(\mathcal{M}(\mathbb{R}))$ are block partitioned already conformal to the block structure of $V^{-1}JV$ in (\ref{BlockStructJ}), where
\begin{gather*}
    H_{ij}=0, W_{ij}\in M_{n_i\times n_j}(\mathcal{M}(\mathbb{R})),\;i,j=1,2;\;W_{11}=I_{n_1},\;W_{22}=I_{n_2},\;W_{12}=W_{21}=0.
\end{gather*}
It now follows that
\begin{gather*}
        D(L)=\left\{f\in [L^2(\mathbb{R})]^n:Jf\in[W^{1,1}_{loc}(\mathbb{R})]^n,\frac{d}{dt}Jf\in [L^2(\mathbb{R})]^n \right\}\\
    =\left\{V\begin{bmatrix}
    f_1  \\
    f_2 
    \end{bmatrix}:f_1\in [W^{1,1}_{loc}(\mathbb{R})]^{n_1},f_1, \frac{df_1}{dt}\in [L^2(\mathbb{R})]^{n_1}, f_2\in [L^2(\mathbb{R})]^{n_2} \right\}
\end{gather*}
and
\begin{gather*}
    Lf=\frac{d}{dt}Jf=\frac{d}{dt}JV\begin{bmatrix}
    f_1  \\
    f_2 
    \end{bmatrix}=V\begin{bmatrix}
    J_{11}\frac{df_1}{dt}\\
    0
    \end{bmatrix},\;\text{ for all } f=V\begin{bmatrix}
    f_1  \\
    f_2 
    \end{bmatrix}\in D(L).
\end{gather*}
Suppose that $f$ is an eigenvector of $L$ with corresponding eigenvalue $\lambda$. That is, $0\not=f=V\begin{bmatrix}
    f_1  \\
    f_2 
    \end{bmatrix}\in D(L)$ and
\begin{gather*}
    V\begin{bmatrix}
    J_{11}\frac{df_1}{dt}\\
    0
    \end{bmatrix}=Lf=\lambda f=\lambda V\begin{bmatrix}
    f_1  \\
    f_2 
    \end{bmatrix},
\end{gather*}
that is,
\begin{gather*}
    J_{11}\frac{df_1}{dt}=\lambda f_1,\;\;f_1\in [W^{1,1}_{loc}(\mathbb{R})]^{n_1},f_1, \frac{df_1}{dt}\in [L^2(\mathbb{R})]^{n_1},\\
     \lambda f_2=0,\;f_2\in [L^2(\mathbb{R})]^{n_2}.
\end{gather*}
As $J_{11}^*=-J_{11}\in M_{n_1}(\mathbb{C})$ and $\det J_{11}\not = 0$, it follows from the case (i) we just considered that we must have $\lambda=0$ and it is the only eigenvalue of $L$ with eigenspace $E_0$ given by
\begin{gather*}
    E_0=\left\{V\begin{bmatrix}
    0  \\
    f_2 
    \end{bmatrix}:f_2\in [L^2(\mathbb{R})]^{n_2}\right\}.
\end{gather*}
Moreover, it obvious that $\dim E_0=\infty$. This proves that $L$ has only one eigenvalue, namely, $\lambda=0$ and it is an eigenvalue of $L$ of infinite multiplicity. 
\end{example}



\bibliographystyle{abbrvnat}
\bibliography{mybibliography}

%
%
%

\end{document}